\documentclass[11pt,oneside,a4paper]{article}

\usepackage[english]{babel}
\usepackage[utf8]{inputenc} 
\usepackage{amsmath,amsfonts,amsthm,amssymb} 
\usepackage{geometry}		
\usepackage{authblk}
\usepackage{graphicx}
\usepackage{subfig}
\usepackage{fourier} 
\usepackage{times}
\usepackage{tikz}
\usepackage{multirow}
\usepackage{mathtools}
\usepackage{mytrees}
\usepackage{chngcntr}
\usepackage[hidelinks]{hyperref}
\usepackage{tikz}
\usetikzlibrary{shapes,shadings}
\newlength{\ml}
\newtheorem{theo}{Theorem}[section]
\newtheorem{cor}{Corollary}[theo]
\newtheorem{lemm}[theo]{Lemma}
\newtheorem{prop}[theo]{Proposition}
\theoremstyle{definition}
\newtheorem{example}{Example}
\newtheorem{definition}{Definition} 

\counterwithin{equation}{section}
\counterwithin{example}{section}
\counterwithin{definition}{section}

\newcommand{\dg}{\ensuremath{\overline{\nabla}}}

\newcommand{\xh}{\ensuremath{\hat{x}}}
\newcommand{\xy}{\ensuremath{y}}
\newcommand{\xb}{\ensuremath{\bar{x}}}

\newcommand{\gh}{\ensuremath{\bar{g}}}
\newcommand{\zb}{\ensuremath{\bar{z}}}

\usepackage[activate={true,nocompatibility},
            final,tracking=true,
            kerning=true,
            spacing=true,
            factor=1100,
            stretch=10,
            shrink=10]{microtype}
\microtypecontext{spacing=nonfrench}

\title{\textbf{Order theory for discrete gradient methods}}
\author{Sølve Eidnes}
\affil{%
    \small{Department of Mathematical Sciences, NTNU, 7491 Trondheim, Norway, and\\
    Department of Mathematics and Cybernetics, SINTEF Digital, 0373 Oslo, Norway\\
  E-mail: \texttt{solve.eidnes@sintef.no}}}
\date{January 15, 2022}

\begin{document}
\maketitle

\noindent \textbf{Abstract:} The discrete gradient methods are integrators designed to preserve invariants of ordinary differential equations. From a formal series expansion of a subclass of these methods, we derive conditions for arbitrarily high order. 
We derive specific results for the average vector field discrete gradient, from which we get P-series methods in the general case, and B-series methods for canonical Hamiltonian systems. Higher order schemes are presented, and their applications are demonstrated on the Hénon--Heiles system and a Lotka--Volterra system, and on both the training and integration of a pendulum system learned from data by a neural network. 

\vspace{6pt}


\vspace{6pt}

\section{Energy preservation and discrete gradient methods}
For an ordinary differential equation (ODE)
\begin{equation}
\dot{x} = f(x), \quad x \in \mathbb{R}^d, \quad f : \mathbb{R}^d \rightarrow \mathbb{R}^d,
\label{eq:ode}
\end{equation}
a first integral, or invariant, is a function $H : \mathbb{R}^d \rightarrow \mathbb{R}$ such that $H(x(t)) = H(x(t_0))$ along the solution curves of \eqref{eq:ode}.
If we can write
\begin{equation}\label{eq:Soff}
f(x) = S(x)\nabla H(x),
\end{equation}
where $S(x) : \mathbb{R}^{d \times d}  \rightarrow \mathbb{R}^d$ is a skew-symmetric matrix, then \eqref{eq:ode} preserves $H$: this follows from the skew-symmetry of $S(x)$, which yields
\begin{equation}\label{eq:intprop}
\frac{\mathrm{d}}{\mathrm{d}t} H(x) = \nabla H(x)^T \dot{x} = \nabla H(x)^T S(x) \nabla H(x) = 0.
\end{equation}
The converse is also true: McLachlan et al.\ showed in \cite{McLachlan99} that, whenever \eqref{eq:ode} has a first integral $H$, there exists a skew-symmetric matrix $S(x)$, bounded near every non-degenerate critical point of $H$, such that \eqref{eq:ode} can be written on what is called the \textit{skew-gradient form}:
\begin{equation}
\dot{x} = S(x) \nabla H(x).
\label{eq:odeform}
\end{equation}
The proof provided in \cite{McLachlan99} for this is based on presenting a general form of one such $S(x)$, the so-called default formula
\begin{equation}\label{eq:standardS}
S(x) =  \frac{f(x) \nabla H(x)^T - \nabla H(x) f(x)^T}{\nabla H(x)^T \nabla H(x)}.
\end{equation}
Unless $d=2$, this is generally not a unique choice of $S(x)$, as e.g.
\begin{equation*}
S(x) = \frac{f(x) g(x)^T - g(x) f(x)^T}{g(x)^T \nabla H(x)}
\end{equation*}
will satisfy \eqref{eq:Soff} for any non-vanishing function $g : \mathbb{R}^d \rightarrow \mathbb{R}^d$. Many ODEs with first integrals have a well-known skew-gradient form \eqref{eq:odeform}. This includes Poisson systems, and the important class consisting of canonical Hamiltonian ODEs. For the latter, $S$ will be constant, so that we may write
\begin{equation}\label{eq:constSform}
\dot{x} = S \nabla H(x).
\end{equation}

A numerical integrator preserving a first integral $H$ exactly is called an integral-preserving, or \textit{energy-preserving}, method. Starting in the late 1970s, a few energy-preserving methods were proposed which relied on some discrete analogue of the property \eqref{eq:intprop}, see e.g.\ \cite{Chorin1978,Itoh88,Marsden1992,Kriksin1993}. Most prominent among these is the class of methods called discrete gradient methods, defined formally by Gonzalez in \cite{Gonzalez96} and given their current name in \cite{McLachlan99}.

Given the first integral $H$, a discrete gradient $\overline{\nabla} H : \mathbb{R}^d \times \mathbb{R}^d \rightarrow \mathbb{R}^d$ is a function satisfying the conditions
\begin{align}
\overline{\nabla} H(x,y)^\text{T} (y-x) &= H(y) - H(x), \label{eq:dgcond1}\\
\overline{\nabla} H(x,x)&= \nabla H(x), \label{eq:dgcond2}
\end{align}
for all $x,\xy \in \mathbb{R}^d$. Introducing also the discrete approximation $\overline{S}(x,y,h)$ to $S(x)$, skew-symmetric and satisfying $\overline{S}(x,x,0) = S(x)$, the corresponding discrete gradient method is given by
\begin{equation}\label{eq:dgm}
\frac{\xh-x}{h} = \overline{S}(x,\xh,h) \overline{\nabla}H(x,\xh).
\end{equation}
This scheme satisfies a discrete analogue to \eqref{eq:intprop}:
\begin{equation*}
H(\xh) - H(x) = h \dg H(x,\xh)^T \overline{S}(x,\xh,h) \dg H(x,\xh) = 0.
\end{equation*}
We say that \eqref{eq:dgm} is consistent to the skew-gradient system \eqref{eq:odeform}, since $\overline{S}(x,\xh,h)$ is a consistent approximation of $S(x)$ and $\dg H(x,\xh)$ is a consistent approximation of $\nabla H(x)$.

If $d\geq 2$, there are in general infinitely many functions satisfying \eqref{eq:dgcond1}--\eqref{eq:dgcond2}. Many explicit definitions of concrete discrete gradients have been suggested, and we will discuss the most prominent among them in Section \ref{sect:expdgs}. One of these is the \textit{average vector field (AVF) discrete gradient}, first introduced in \cite{Harten83} and sometimes called the mean value discrete gradient \cite{McLachlan99}. For a given $H$, it is given by the average of $\nabla H$ on the segment $[x,\xy]$:
\begin{equation}
\dg_{\text{AVF}} H(x,\xy) = \int_0^1 \nabla H ((1-\xi)x + \xi \xy) \, \mathrm{d}\xi.
\label{eq:avfdg}
\end{equation}
When applied to the constant $S$ system \eqref{eq:constSform}, the discrete gradient method with $\overline{S}(x,\xy,h) = S$ and $\dg H = \dg_{\text{AVF}} H$ coincides with the scheme
\begin{equation}\label{eq:avfm2}
\frac{\xh-x}{h} = \int_0^1 f ((1-\xi)x + \xi \xh) \, \mathrm{d}\xi.
\end{equation}
This is sometimes viewed as a method by itself, applicable to any system \eqref{eq:ode}, in which case it is called the average vector field (AVF) method \cite{Quispel08}. This was shown in \cite{Celledoni09} to be a B-series method.

As pointed out in \cite{McLachlan99}, the discrete gradient is restricted by its definition to be at best a second order approximation to point values of $\nabla{H}$. In much of the literature on discrete gradient methods, see e.g.\ \cite{Gonzalez96,Hairer06}, the approximation $\overline{S}$ is defined as being independent of $h$. In that case, the discrete gradient scheme \eqref{eq:dgm} can at best guarantee second order convergence towards the exact solution. Over the last two decades, there have been published some notable papers on higher order discrete gradient methods. McLaren and Quispel were first out with their bootstrapping technique derived in \cite{McLaren04,McLaren09}. Given any discrete gradient $\dg H$ and an approximation to $S(x)$ given by $\overline{S}(x,\xy,h)$, they compare the Taylor expansion of the corresponding discrete gradient scheme to that of the exact solution, and thus find a new approximation $\tilde{S}(x,\xy,h)$ to $S(x)$ which yields higher order. This quickly becomes a very involved procedure, but by using a symmetric discrete gradient, they derive fourth order methods. A downside of this method is that the schemes of order higher than two require the calculation of tensors of order three or higher at every time step.

A fourth order generalization of the AVF method is proposed by the same authors in \cite{Quispel08}. This can be viewed as a fourth-order discrete gradient method for all skew-gradient systems where $S$ is constant.
Also worth mentioning in this setting is the collocation-like method introduced by Hairer \cite{Hairer09} and then generalized to Poisson systems by Cohen and Hairer \cite{Cohen11}. This is a multi-stage extension of the AVF discrete gradient method. To get higher than second order, more than one stage is required. In that case the method is not a discrete gradient method, although it is energy-preserving.

Norton et al.\ show in \cite{Norton15} that linear projection methods can be viewed as a class of discrete gradient methods for skew-gradient systems with $S(x)$ given by the default formula \eqref{eq:standardS}. In connection to this, Norton and Quispel suggest in \cite{Norton14} the class of 
approximations to \eqref{eq:standardS} given by
\begin{equation}\label{eq:nortonS}
\overline{S}(x,y,h) = \frac{\tilde{f}(x,y,h) \tilde{g}(x,y,h)^T - \tilde{g}(x,y,h) \tilde{f}(x,y,h)^T}{\hat{g}(x,y,h)^T \breve{g}(x,y,h)},
\end{equation}
where $\tilde{f}(x,y,h)$ is a consistent approximation to $f(x)$, and $\tilde{g}(x,y,h)$, $\hat{g}(x,y,h)$ and $\breve{g}(x,y,h)$ are all consistent approximations to $\nabla H(x)$. The corresponding discrete gradient method then inherits the order of the method $\xh = x + h \tilde{f}(x,\xh,h)$.

The use of the discrete gradient method expands beyond the numerical integration of ODEs.  It can be applied to the time-integration of partial differential equations (PDEs) which has constants of motion, guaranteeing preservation of a discrete approximation of an integral \cite{Furihata99, Celledoni12, Eidnes18}. Higher-order methods have been studied in this context, but then only schemes that are expanding on the AVF method \cite{Cai15, Jiang17}. Furthermore, building on the recently introduced Hamiltonian neural networks \cite{Greydanus19, Chen19}, Matsubara et al.\ have shown how discrete gradients can be used to learn energy-preserving systems from data \cite{Matsubara19}. Since the energy is given by a neural network, none of the higher-order methods reviewed above are applicable in that case, and the authors suggest to use multi-step methods to get higher than second order.

To the best of our knowledge, no one has so far suggested higher than fourth order discrete gradient methods for a general skew-gradient system \eqref{eq:odeform}. Furthermore, for this general case, all discrete gradient methods suggested of higher than second order involve tensors of order three or higher.  In this paper we present general theory for achieving methods of arbitrary order, using any discrete gradient and not depending on tensors of order higher than two. Largely inspired by the above mentioned references, especially \cite{Quispel08,McLaren04,McLaren09}, we present here a general form giving a class of approximations $\overline{S}(x,y,h)$ to any $S(x)$ in \eqref{eq:odeform}, with corresponding conditions for achieving an arbitrary order of the discrete gradient method \eqref{eq:dgm}. We do this step by step. In the next section, we derive some useful properties of a general discrete gradient and discuss the most common specific discrete gradients. Then we consider the AVF method and use order theory for B-series methods to obtain a generalization of this, with corresponding order conditions. In Section \ref{sect:avfgen}, we build on this to develop higher order discrete gradient methods for a general skew-gradient system, using the AVF discrete gradient. Then, in Section \ref{sect:gengen}, we generalize this further to allow for a free choice of the discrete gradient, thus arriving at the general form $\overline{S}(x,y,h)$ mentioned above, and a formal series expansion of the corresponding discrete gradient methods. Throughout the paper, we present several examples of higher order schemes for the different cases. In the final section, we apply some of these schemes to numerical examples: the Hénon--Heiles system with a constant $S$, a Lotka--Volterra system with a non-constant S, and a pendulum system learned by a neural network.

\section{A preliminary analysis of discrete gradients}
To simplify notation in the following derivations, we define $g:=\nabla H$. Furthermore, we suppress the first argument of $\dg H$ and define $\gh(\xy) := \dg H(x,\xy)$. We use Einstein summation convention, with a comma in the subscript to differentiate between components of covectors and derivatives. That is, we write
$\gh(\xy)^i_{,j} := \frac{\partial \gh(\xy)^i}{\partial \xy^j}$, $\gh(\xy)_{i,j} := \frac{\partial \gh(\xy)_i}{\partial \xy^j}$ and so forth. Taylor expanding $\gh(\xy)$ around $x$, we get
\begin{equation}
\begin{split}
\gh(\xy)^i =& \, \gh(x)^i + \gh(x)^i_{,j}(\xy^j-x^j) + \frac{1}{2}\gh(x)^i_{,j k}(\xy^j-x^j)(\xy^k-x^k) \\
&+ \frac{1}{6}\gh(x)^i_{,j k l}(\xy^j-x^j)(\xy^k-x^k)(\xy^l-x^l) + \mathcal{O}(\lvert \xy-x \rvert^4),
\end{split}
\label{eq:ghat}
\end{equation}
or
\begin{align}
\gh(\xy) =& \sum_{\kappa=0}^\infty \frac{1}{\kappa !} \gh^{(\kappa)}(x)(\xy-x)^\kappa.
\label{eq:taylor1}
\end{align}
By the consistency criterion \eqref{eq:dgcond2}, we have $\gh(x) = g(x)$. However, if we require the discrete gradient to be a differentiable function in its second argument, \eqref{eq:dgcond2} follows directly from \eqref{eq:dgcond1}. To see this, we write \eqref{eq:dgcond1} as
\begin{equation}
H(\xy)-H(x) = \gh(\xy)_i (\xy^i - x^i).
\label{eq:DG1_einstein}
\end{equation}
Differentiating this with respect to $\xy^j$, we get
\begin{equation}
g(\xy)_j = H(\xy)_{,j} = \gh(\xy)_{i,j} (\xy^i - x^i) + \gh(\xy)_j,
\label{eq:DDG_einstein}
\end{equation}
The case $\xy = x$ immediately gives $g(x)_j = \gh(x)_j$, or \eqref{eq:dgcond2}. 
Assuming further that $\dg H \in C^2(\mathbb{R}^d \times \mathbb{R}^d,\mathbb{R}^d)$, we can differentiate once more to get
\begin{equation}
g(\xy)_{j,k} = H(\xy)_{,jk} = \gh(\xy)_{i,jk} (\xy^i - x^i) + \gh(\xy)_{j,k} + \gh(\xy)_{k,j},
\label{eq:diff2}
\end{equation}
which means that
\begin{equation*}
g(x)_{j,k} = H(x)_{,jk} = \gh(x)_{j,k} + \gh(x)_{k,j},
\end{equation*}
or
\begin{equation}
\nabla^2 H(x) = D_{2} \dg H(x,x) + (D_{2} \dg H(x,x))^\text{T},
\label{eq:DDH}
\end{equation}
where $\nabla^2 H \coloneqq D\nabla H$ denotes the Hessian of $H$, and $D_{2} \dg H$ denotes the Jacobian of $\dg H$ with respect to its second argument.

\begin{lemm}\label{th:symdg}
If the discrete gradient $\dg H$ is symmetric, i.e.\ $\dg H(x,y) = \dg H(y,x)$ for all $x, y \in \mathbb{R}^d$, then
\begin{equation}
D_{2} \dg H(x,x) = \frac{1}{2}\nabla^2 H(x).
\label{eq:symdg}
\end{equation}
\end{lemm}
\begin{proof}
Disclosing the suppressed argument $x$ in \eqref{eq:DDG_einstein}, we have
\begin{equation*}
g(\xy)_j = \frac{\partial}{\partial \xy^j}(\gh(x,\xy)_i) (\xy^i - x^i) + \gh(x,\xy)_j,
\end{equation*}
which we can differentiate by $x^k$ to get
\begin{equation*}
0 = \frac{\partial^2}{\partial x^k\partial \xy^j}(\gh(x,\xy)_i) (\xy^i - x^i) - \frac{\partial}{\partial \xy^j}\gh(x,\xy)_k + \frac{\partial}{\partial x^k}\gh(x,\xy)_j.
\end{equation*}
If $\dg H$ is symmetric, 
$$ \frac{\partial}{\partial x^k} \gh(x,\xy)_j = \frac{\partial}{\partial x^k} \gh(\xy,x)_j. $$
Thus, for $\xy = x$ we get $\gh(x)_{k,j} = \gh(x)_{j,k}$, or $(D_{2} \dg H(x,x))^\text{T} = D_{2} \dg H(x,x)$. Inserting that in \eqref{eq:DDH}, we obtain \eqref{eq:symdg}.
\end{proof}

\begin{definition}\label{def:Q}
Given a discrete gradient $\dg H \in C^1(\mathbb{R}^d \times \mathbb{R}^d,\mathbb{R}^d)$, we define the function $Q : \mathbb{R}^d \times \mathbb{R}^d \rightarrow \mathbb{R}^{d\times d}$ by
\begin{equation}
Q(x,\xy) \coloneqq \frac{1}{2}\left( (D_2\dg H(x,\xy))^T - D_2\dg H(x,\xy) \right).
\label{eq:theQ}
\end{equation}
\end{definition}
Note that $Q(x,y)$ is a skew-symmetric matrix. From \eqref{eq:DDH}, we see that $Q(x,x) = \frac{1}{2}\nabla^2 H(x) - D_2\dg H(x,x)$. Differentiating \eqref{eq:theQ} with respect to the second argument and setting $\xy = x$, we obtain 
\begin{equation*}
(D_2 Q(x,x))_{jkl} = \frac{1}{2}\gh(x)_{k,jl} - \frac{1}{2}\gh(x)_{j,kl}.
\end{equation*}
Similarly, differentiating \eqref{eq:diff2} with respect to the second argument and setting $\xy = x$, we obtain 
\begin{equation*}
g(x)_{j,kl} = \gh(x)_{j,kl} + \gh(x)_{k,jl} + \gh(x)_{l,jk}.
\end{equation*}
Using these results, we get that, for any $v \in \mathbb{R}^d$,
\begin{align*}
(D_2 Q(x,x) (v,v))_j &= (D_2 Q(x,x))_{jkl} v^k v^l = \frac{1}{2}\gh(x)_{k,jl}v^k v^l-\frac{1}{2}\gh(x)_{j,kl}v^k v^l\\
&=\frac{1}{4}\gh(x)_{k,jl}v^k v^l + \frac{1}{4}\gh(x)_{l,jk}v^k v^l + \frac{1}{4}\gh(x)_{j,kl}v^k v^l-\frac{3}{4}\gh(x)_{j,kl}v^k v^l\\
&=\frac{1}{4}g(x)_{j,kl}v^k v^l - \frac{3}{4}\gh(x)_{j,kl}v^k v^l,
\end{align*}
or
\begin{equation}
D_2 Q(x,x) (v,v) = \frac{1}{4} D^2 \nabla H(x) (v,v) - \frac{3}{4} D_2^2 \dg H(x,x) (v,v). 
\label{eq:QD2}
\end{equation}
Continuing in this manner, we get the following result, which we will later show is crucial for the development of a general theory for higher order discrete gradient methods.
\begin{lemm}\label{th:Qlem}
For a discrete gradient $\dg H \in C^{p}(\mathbb{R}^d \times \mathbb{R}^d,\mathbb{R})$ and the corresponding $Q$ given by \eqref{eq:theQ},
\begin{equation*}
D_2^{\kappa} \dg H(x,x)v^\kappa = \frac{1}{\kappa+1}D^{\kappa}\nabla H(x)v^\kappa - \frac{2\kappa}{\kappa+1} D_2^{\kappa-1} Q(x,x) v^\kappa \quad \text{ for any } \kappa \in [1,p], v \in \mathbb{R}^d.
\end{equation*}

\end{lemm}
\begin{proof}
Differentiating \eqref{eq:diff2} $\kappa-1$ times by $\xy$ and setting $\xy = x$, we find that the $\kappa$-th derivatives of $g(x)$ can be expressed by the $\kappa$-th derivatives of $\gh(x)$ through the relation
\begin{equation}
g(x)_{j,I} = \gh(x)_{j,I} + \sum_{m=1}^\kappa \gh(x)_{i_m, \left\{j, I_m\right\}}, \quad \text{ for all } j,I,\kappa,
\label{eq:gsum}
\end{equation}
where $I = \left\{i_1,i_2,\dots,i_\kappa \right\}$ is an ordered set of $\kappa$ indices, and $I_m = I \setminus \left\{ i_m\right\} = \left\{i_1,i_2,\ldots i_{m-1},i_{m+1},\ldots,i_\kappa \right\}$, i.e.\ $I$ with the $m$-th element excluded. Similarly, by continued differentiation of \eqref{eq:theQ}, we obtain
\begin{equation*}
(D_2^{\kappa-1} Q(x,x))_{j,I} = \frac{1}{2}\gh(x)_{i_1, \left\{j, I_1\right\}} - \frac{1}{2}\gh(x)_{j,I}.
\end{equation*}
Thus
\begin{align*}
(D_2^{\kappa-1} Q(x,x) v^\kappa)_j &= (D_2^{\kappa-1} Q(x,x))_{j,I} v^I = \frac{1}{2}\gh(x)_{i_1,\left\{j, I_1\right\}}v^{\left\{j, I_1\right\}} -\frac{1}{2}\gh(x)_{j,I}v^I\\
&=\frac{1}{2\kappa} \sum_{m=1}^\kappa\gh(x)_{i_m,\left\{j, I_m\right\}}v^{\left\{j, I_m\right\}} + \frac{1}{2\kappa} \gh(x)_{j,I}v^I- \frac{1}{2\kappa} \gh(x)_{j,I}v^I -\frac{1}{2}\gh(x)_{j,I}v^I\\
&=\frac{1}{2\kappa} g(x)_{j,I} v^I - (\frac{1}{2\kappa} +\frac{1}{2})\gh(x)_{j,I}v^I = \frac{1}{2\kappa} g(x)_{j,I} v^I - \frac{\kappa+1}{2\kappa}\gh(x)_{j,I}v^I.
\end{align*}
\end{proof}

\subsection{Properties of different discrete gradients}\label{sect:expdgs}

While introducing the discrete gradient methods in \cite{Gonzalez96}, Gonzalez also gave an example of a discrete gradient satisfying \eqref{eq:dgcond1}--\eqref{eq:dgcond2}: the \textit{midpoint discrete gradient} is given by
\begin{equation*}
\dg_\text{M} H(x,\xy) \coloneqq \nabla H\left(\frac{x+\xy}{2}\right) + \frac{H(\xy)-H(x)-\nabla H\left(\frac{x+\xy}{2}\right)^T\left(\xy-x\right)}{(\xy-x)^T(\xy-x)}\left(\xy-x\right).
\end{equation*}

Even when $H$ is analytic, this discrete gradient is often not; the second order partial derivatives are in general singular in $y=x$.
For that reason, it is not suited for achieving higher order methods by the techniques we consider in this paper.

The \textit{Itoh--Abe discrete gradient method}, introduced in \cite{Itoh88}, notably does not require evaluation of the gradient. The corresponding discrete gradient, which has also been called the coordinate increment discrete gradient \cite{McLachlan99}, is defined by
\begin{equation}
\dg_\text{IA}H(x,\xy) \coloneqq \sum_{j=1}^d \alpha_j e_j,
\label{eq:iadg}
\end{equation}
where $e_j$ is the $j^{\text{th}}$ canonical unit vector and
\begin{align*}
\alpha_j &=
  \begin{cases}
    \dfrac{H(w_j) - H(w_{j-1})}{\xy^j-x^j}       & \quad \text{if } \xy^j \neq x^j,\\
    \frac{\partial H}{\partial x^j}(w_{j-1})  & \quad \text{if } \xy^j = x^j,
  \end{cases}\\
w_j &= \sum\nolimits_{i=1}^j \xy^i e_i + \sum\nolimits_{i=j+1}^n x^i e_i.
\end{align*}
While the other discrete gradients we consider in this paper are symmetric and thus second order approximations to $\nabla H$, the Itoh--Abe discrete gradient is only of first order. However, a second order discrete gradient, which we call the \textit{symmetrized Itoh--Abe (SIA) discrete gradient}, is given by
\begin{equation}
\dg_\text{SIA} H (x,\xy) \coloneqq \frac{1}{2} \left(\dg_\text{IA}H(x,\xy)+\dg_\text{IA}H(\xy,x)\right).
\label{eq:siadg}
\end{equation}

Furihata presented the discrete variational derivative method for a class of PDEs in \cite{Furihata99}, a method which has been developed further by Furihata, Matsuo and co-authors in a series of papers, e.g.\ \cite{Matsuo02,Yaguchi12}, as well as the monograph \cite{Furihata11}. As shown in \cite{Eidnes18}, these schemes can also be obtained by semi-discretizing the PDE in space and then applying a discrete gradient method on the resulting system of ODEs. We give here the definition of the specific discrete gradient that gives the schemes of Furihata and co-author, defined for a class of invariants that includes all polynomial functions:
\begin{definition}
Assume that we can write the first integral as
\begin{equation}\label{eq:furihataH}
H(x) = \sum_l c_l \prod_{j=1}^d f^l_{j}(x^j),
\end{equation}
for functions $f^l_{j} : \mathbb{R} \rightarrow \mathbb{R}$. 
The \textit{Furihata discrete gradient} $\overline{\nabla}_\text{F} H(x,\xy)$ is defined by 
\begin{equation}\label{eq:furihatadg}
\dg_\text{F}H(x,\xy) \coloneqq \sum_{j=1}^d \alpha_j e_j,
\end{equation}
where $e_j$ is the $j^{\text{th}}$ canonical unit vector and
\begin{align*}
\alpha_j &=
  \begin{cases}
    \sum\limits_l \frac{c_l}{2} \frac{f_j^l(\xy^j)-f_j^l(x^j)}{\xy^j-x^j} \left(\prod\limits_{k=1}^{j-1}{f_k^l(x^k)} + \prod\limits_{k=1}^{j-1}{f_k^l(\xy^k)}\right) \prod\limits_{k=j+1}^{d}{\frac{f_k^l(x^k)+f_k^l(\xy^k)}{2}}  & \quad \text{if } \xy^j \neq x^j,\\
    \sum\limits_l \frac{c_l}{2} \frac{\mathrm{d}f^l_{j}(x^j)}{\mathrm{d}x^j} \left(\prod\limits_{k=1}^{j-1}{f_k^l(x^k)} + \prod\limits_{k=1}^{j-1}{f_k^l(\xy^k)}\right) \prod\limits_{k=j+1}^{d}{\frac{f_k^l(x^k)+f_k^l(\xy^k)}{2}}   & \quad \text{if } \xy^j = x^j.
  \end{cases}
\end{align*}
\end{definition}

The discrete gradient introduced by Matsubara and co-authors in \cite{Matsubara19} shares some relation to the Furihata discrete gradient, in that they are both relying on a discrete analogue to the product rule for derivatives. But the former discrete gradient, obtained by what the authors call the \textit{automatic discrete differentiation algorithm}, more importantly depends on an analogue to the chain rule. By using this algorithm instead of standard automatic differentiation in a neural network that learns a dynamical system from data, the discrete gradient is obtained together with the preserved energy. Thus the learned system can be integrated in an energy-preserving manner, which distinguishes the approach of Matsubara et al. from comparable studies \cite{Greydanus19,Chen19,Zhong19}.

Lastly we consider the AVF discrete gradient \eqref{eq:avfdg}, which has several noteworthy characteristics.
\begin{lemm}\label{th:avfsym}
The $Q(x,y)$ corresponding to the AVF discrete gradient is the zero matrix, since $(D_{2} \dg_{\text{AVF}} H (x,\xy))^\text{T} = D_{2} \dg_{\text{AVF}} H(x,\xy)$.
\end{lemm}
\begin{proof}
For $\gh(\xy) \coloneqq \dg_{\text{AVF}} H(x,\xy)$, we have
\begin{align*}
\gh(\xy)_{i,j} &= \frac{\partial}{\partial \xy^j} \int_0^1 g((1-\xi)x + \xi \xy)_i \, \mathrm{d}\xi =  \int_0^1 \frac{\partial}{\partial \xy^j} g((1-\xi)x + \xi \xy)_i \, \mathrm{d}\xi \\
&= \int_0^1 \xi g((1-\xi)x + \xi \xy)_{i,j} \, \mathrm{d}\xi = \int_0^1 \xi g((1-\xi)x + \xi \xy)_{j,i} \, \mathrm{d}\xi \\
& = \gh(\xy)_{j,i}.
\end{align*}
\end{proof}

From the Integrability Lemma (see e.g.\ \cite[Lemma VI.2.7]{Hairer06}) and the above, we have that the AVF discrete gradient defines a gradient vector field:

\begin{cor}\label{cor:avfgrad}
The AVF discrete gradient is the gradient with respect to the second argument of a function $\tilde{H}(x,\xy)$. That is,
\begin{equation*}
\dg_{\text{AVF}} H(x,\xy) = \nabla_2 \tilde{H} (x,\xy),
\end{equation*}
for some $\tilde{H} : \mathbb{R}^d \times \mathbb{R}^d \rightarrow \mathbb{R}$ and all $x, y \in \mathbb{R}^d$.
\end{cor}

Any other discrete gradient will have these properties only for functions $H$ for which it coincides with the AVF discrete gradient:
\begin{prop}\label{th:symjac}
\sloppy The AVF discrete gradient is the unique discrete gradient satisfying $(D_{2} \dg H (x,\xy))^\text{T} = D_{2} \dg H(x,\xy)$ for all $H$, $x$ and $y$, and it has the formal expansion
\begin{equation}
\dg_\text{AVF} H(x,\xy) = \sum_{\kappa=0}^\infty \frac{1}{(\kappa+1)!} D^{\kappa}\nabla H(x)(\xy-x)^\kappa.
\label{eq:tayloravf}
\end{equation}
\end{prop}
\begin{proof}
Assume that $\nabla H$ is an analytic function. As in the proof of Lemma \ref{th:Qlem}, let $I = \left\{i_1,i_2,\dots,i_\kappa \right\}$ be an ordered set of $\kappa$ indices, and let $I_m$ be $I$ with the $m^{\text{th}}$ element excluded. If $\gh(\xy)_{i,j} = \gh(\xy)_{j,i}$ for all $i,j$, then also
\begin{equation}
\gh(\xy)_{i,I} = \gh(\xy)_{i_m,\left\{i,I_m\right\}} \quad \text{ for all } i,I,m.
\label{eq:ghset}
\end{equation}
Inserting \eqref{eq:ghset} in \eqref{eq:gsum} we get $g^{(\kappa)}(x) = (1+\kappa) \gh^{(\kappa)}(x)$. 
Then inserting this for $\gh^{(\kappa)}(x)$ in \eqref{eq:taylor1}, we get \eqref{eq:tayloravf}, which uniquely defines the AVF discrete gradient.
\end{proof}
A consequence of the above result is that the AVF discrete gradient is the unique discrete gradient for which the scheme \eqref{eq:dgm} with $\overline{S}(x,\xh,h)=S$ is a B-series method when applied to the system \eqref{eq:constSform}. 

As we see from the above definitions and discussion, each of the discrete gradients have their advantages and disadvantages. Gonzalez' midpoint discrete gradient is easily calculated from the energy $H$ and the gradient $\nabla H$, but it is in general only once differentiable. The Itoh--Abe and Furihata discrete gradient methods do not require knowledge of the gradient, but the former is only a first order method and the latter is only defined for $H$ of the form \eqref{eq:furihataH}. The AVF discrete gradient is the unique discrete gradient whose series expansion is given by the differentials of the gradient. It does however require an integral to be calculated.  The discrete gradient of Matsubara and co-authors has a different area of use than the others: it is only defined when $H$ is a neural network. In that case neither Gonzelez' midpoint, the AVF or the Furihata discrete gradients can be obtained.


\subsection{Third and fourth order schemes for the constant $S$ case}
Consider now only the cases where $S$ is constant, i.e.\ \eqref{eq:constSform}. By comparing the Taylor series of the exact solution and that of the discrete gradient method, and by using the properties of the discrete gradient developed above, we may achieve higher order discrete gradient methods.

In search of a third order scheme, we assume that $\xh$ is a third order in $h$ approximation of $x(t_0+h)$, and find
\begin{align*}
S \dg H(x,\xh) =& S (\nabla H(x) + D_2 \dg H(x,x) (h S \nabla H(x) + \frac{1}{2} h^2 S \nabla^2 H(x) S \nabla H(x) + \mathcal{O}(h^2))\\
&+\frac{1}{2} D_2^2 \dg H(x,x) (h S \nabla H(x)+ \mathcal{O}(h^2),h S \nabla H(x)+ \mathcal{O}(h^2)) + \mathcal{O}(h^3) \\
=& f + h S D_2\dg H f + \frac{1}{2} h^2 S D_2 \dg H f' f + \frac{1}{2} h^2 S D_2^2 \dg H (f,f) + \mathcal{O}(h^3),
\end{align*}
where we have suppressed the argument $x$ of $f$, $D_2 \dg H$ and $D_2^2 \dg H$  in the last line.
Furthermore, we use that
\begin{equation*}
Q(x,x+ \gamma h f(x)) = Q(x,x) + \gamma h D_2 Q(x,x) (f,\cdot) + \mathcal{O}(h^2)
\end{equation*}
and \eqref{eq:QD2}
to get
\begin{align*}
S Q(x,x&+ \gamma h f(x)) S \dg H(x,\xh) \\
=& \, S Q(x,x) S \dg H(x,\xh) + \gamma h S D_2 Q(x,x) (f, S \dg H(x,\xh)) + \mathcal{O}(h^2)\\
=& \,  S Q(x,x) S (\nabla H(x) + D_2 \dg H(x,x)( h S \nabla H(x) + \mathcal{O}(h^2)) \\
& \, + \gamma h S D_2 Q(x,x) (f, S (\nabla H(x) + \mathcal{O}(h)) + \mathcal{O}(h^2)\\
=& \, S Q(x,x) f + h S Q(x,x) S D_2 \dg H(x,x) f + \gamma h S D_2 Q(x,x)(f,f) + \mathcal{O}(h^2) \\
=& \, \frac{1}{2} f' f - S D_2 \dg H f + \frac{1}{2} h f' S D_2 \dg H f - h S D_2 \dg H S D_2 \dg H f \\
& \, + \frac{1}{6}\gamma h f''(f,f) - \frac{1}{2} \gamma h S D_2^2 \dg H (f,f) + \mathcal{O}(h^2),
\end{align*}
where again we suppress the argument $x$ in the last expression.
Thus the discrete gradient scheme \eqref{eq:dgm} is of order $3$ if $\dg H \in C^2(\mathbb{R}^d\times \mathbb{R}^d,\mathbb{R}^d)$ and $\overline{S}(x,\xh,h) = \overline{S}(x,h)$ is given by
\begin{equation*}
\overline{S}(x,h) = S + h S Q(x,x + \frac{2}{3} h f(x))  S  + h^2 S \big( Q(x,x) S Q(x,x) - \frac{1}{12} \nabla^2 H(x) S \nabla^2 H(x)\big) S.
\end{equation*}

Finding an approximation of $S$ that guarantees higher order of the discrete gradient method quickly becomes significantly more complicated, and results in increasingly complicated expressions for $\overline{S}(x,\xh,h)$. For example, it can be shown that one fourth order scheme of the form \eqref{eq:dgm} is given by any $\dg H \in C^3(\mathbb{R}^d\times \mathbb{R}^d,\mathbb{R}^d)$ and
\begin{equation}
\begin{split}
\overline{S}(x,\cdot,h) =& \, S + h S \big( \frac{8}{9} Q(x,z_3) + \frac{1}{9} Q(x,x)\big) S \\
&+ h^2 S \big( Q(x,z_2) S Q(x,z_2) - \frac{1}{12} \nabla^2 H(z_1) S \nabla^2 H(z_1)\big) S \\
&+ h^3 S \big( Q(x,x) S Q(x,x) S Q(x,x) \\
&- \frac{1}{12} \nabla^2 H(x) S \nabla^2 H(x) S Q(x,x) - \frac{1}{12} Q(x,x) S \nabla^2 H(x) S \nabla^2 H(x)\big)S,
\end{split}
\label{eq:4thorderDGM}
\end{equation}
where
\begin{equation*}
z_1 = x + \frac{1}{2} h f(x), \qquad z_2 = x + \frac{2}{3} h f(x), \qquad z_3 = x + \frac{3}{4} h f(z_1).
\end{equation*}

Note that if we choose a symmetric discrete gradient, we have by Lemma \ref{th:symdg} that $Q(x,x) = 0$, and many of the terms in \eqref{eq:4thorderDGM} disappear. If we use the AVF discrete gradient, \eqref{eq:4thorderDGM} simplifies to
\begin{equation}
\overline{S}(x,h) = S - \frac{1}{12} h^2 S \nabla^2 H(z_1) S \nabla^2 H(z_1) S.
\label{eq:4thordavf}
\end{equation}
This is very similar to the higher order AVF methods of Quispel and McLaren, as given in \cite{Quispel08}, applied to \eqref{eq:odeform} with $S$ constant: if we replace $z_1$ in \eqref{eq:4thordavf} by $x$, we get their third order scheme; if we replace $z_1$ by $\frac{x+\xh}{2}$, we get their symmetric fourth order scheme.

Seeing as \eqref{eq:4thorderDGM} simplifies considerably when the AVF discrete gradient is chosen, and since we in this case get a B-series method, we begin our generalization to arbitrary order by studying this case specifically in the section to follow.

\section{A generalization of the AVF method}
Let us recall the concept of B-series. Referring to the definitions in \cite[Section III.1]{Hairer06}, we let $T$ be the set of rooted trees, built recursively from starting with $\tau=\ab$ and letting $\tau = [\tau_1,\ldots,\tau_m]$ be the tree obtained by grafting the roots of the trees $\tau_1,\ldots,\tau_m$ to a new root. Furthermore, $F(\tau)$ is the elementary differential associated with the tree $\tau$, defined by $F(\ab)(x) = f(x)$ and
\begin{equation*}
F(\tau)(x) = f^{(m)}(x) \big(F(\tau_1)(x),\ldots,F(\tau_m)(x)\big),
\end{equation*}
and $\sigma(\tau)$ is the symmetry coefficient for $\tau$, defined by $\sigma(\ab) = 1$ and
\begin{equation}
\sigma(\tau) = \sigma(\tau_1) \cdots \sigma(\tau_m) \cdot \mu_1! \, \mu_2! \cdots,
\label{eq:symcoeff}
\end{equation}
where the integers $\mu_1$, $\mu_2, \ldots$ count equal trees among $\tau_1,\ldots,\tau_m$.
Then, if $\phi : T \cup\lbrace\emptyset\rbrace \rightarrow \mathbb{R}$ is an arbitrary map, a \textit{B-series} is a formal series
\begin{equation}\label{eq:bseries}
B(\phi,x) = \phi(\emptyset)x + \sum_{\tau\in T}\frac{h^{\lvert\tau\rvert}}{\sigma(\tau)}\phi(\tau)F(\tau)(x).
\end{equation}
The exact solution of \eqref{eq:ode} can be written as the B-series $B(\frac{1}{\gamma},x)$, where the coefficient $\gamma$ satisfies $\gamma(\emptyset) = \gamma(\ab) = 1$ and
\begin{equation}
\gamma(\tau) = \lvert\tau\rvert \, \gamma(\tau_1) \cdots \gamma(\tau_m) \quad \text{ for } \tau = [\tau_1,\ldots,\tau_m],
\label{eq:gamma}
\end{equation}
where $\lvert\tau\rvert$ is the order, i.e.\ the number of nodes, of $\tau$.

\begin{definition}\label{def:avfmexgen}
The \textit{generalized AVF method} is given by
\begin{equation}
\frac{\xh-x}{h} = \left(I + \sum_{n=2}^{p-1} h^n \sum_{j} b_{nj} \left( \prod_{k=1}^n f'(z_{njk}) + (-1)^n \prod_{k=1}^n f'(z_{nj(n-k+1)}) \right) \right) \int_0^1 f((1-\xi)x + \xi \xh) \, \mathrm{d}\xi,
\label{eq:avfmexgen}
\end{equation}
where each $z_{njk} \coloneqq z_{njk}(x,\xh,h) = B(\phi_{njk},x)$ can be written as a B-series with $\phi(\emptyset) = 1$.
\end{definition}

Note that we may alternatively write \eqref{eq:avfmexgen} in the slightly more compact form
\begin{equation*}
\frac{\xh-x}{h} = \sum_{n=0}^{p-1} h^n \sum_{j} b_{nj} \left( \prod_{k=1}^n f'(z_{njk}) + (-1)^n \prod_{k=1}^n f'(z_{nj(n-k+1)}) \right) \int_0^1 f((1-\xi)x + \xi \xh) \, \mathrm{d}\xi
\end{equation*}
with $\sum_j b_{0j} = \frac{1}{2}$.

\begin{theo}
When applied to \eqref{eq:ode} with $f(x) = S\nabla H(x)$, where $S$ is a constant skew-symmetric matrix, the scheme \eqref{eq:avfmexgen} preserves $H$, in that $H(\xh) = H(x)$.
\end{theo}
\begin{proof}
With $f(x) = S\nabla H(x)$, \eqref{eq:avfmexgen} becomes
\begin{equation*}
\frac{\xh-x}{h} = \overline{S}(x,\xh,h) \dg_\text{AVF} H(x,\xh),
\end{equation*}
with
\begin{equation*}
\overline{S}(x,\xh,h) = S + \sum_{n=2}^{p-1} h^n \sum_{j} b_{nj} \left( \prod_{k=1}^n S\nabla^2H(z_{njk}) + (-1)^n \prod_{k=1}^n S\nabla^2H(z_{nj(n-k+1)}) \right) S.
\end{equation*}
We have
\begin{align*}
\left(\prod_{k=1}^n \right. & \left. S\nabla^2H(z_{njk}) \cdot S + (-1)^n \prod_{k=1}^n S\nabla^2H(z_{nj(n-k+1)} ) \cdot S \right)^T \\
& =  S^T \prod_{k=1}^n \big(\nabla^2H(z_{nj(n-k+1)})^T S^T\big) + (-1)^n S^T \prod_{k=1}^n \big(\nabla^2H(z_{njk})^T S^T\big) \\
& = (-1)^{i+1} S \prod_{k=1}^n \nabla^2H(z_{nj(n-k+1)}) S - S \prod_{k=1}^n \nabla^2H(z_{njk}) S \\
& = - \left(\prod_{k=1}^n S\nabla^2H(z_{njk}) \cdot S + (-1)^n \prod_{k=1}^n S\nabla^2H(z_{nj(n-k+1)} ) \cdot S \right),
\end{align*}
and thus $\overline{S}(x,\xh,h)$ is a skew-symmetric matrix.
\end{proof}

Before considering the order conditions of the generalized AVF method, let us recall a couple of results from the literature on B-series.

\begin{lemm}[{\cite[Lemma III.1.9]{Hairer06}}]
Let $B(a,x)$ be a B-series with $a(\emptyset) = 1$. Then $h f(B(a,x)) = B(a',x)$ is also a B-series, with $a'(\emptyset) = 0$, $a'(\ab) = 1$ and otherwise
\begin{equation*}
a'(\tau) = a(\tau_1) \cdots a(\tau_m) \quad \text{ for } \tau = [\tau_1,\ldots,\tau_m].
\end{equation*}
\label{th:hlw}
\end{lemm}

\begin{lemm}[{\cite[Theorem 2.2]{Norsett79}}]
Let $B(a,x)$ and $B(b,x)$ be two B-series with $a(\emptyset) = 1$ and $b(\emptyset) = 0$. Then $h f'(B(a,x))B(b,x) = B(a\times b,x)$, i.e.\ a B-series, with $(a \times b)(\emptyset) = (a \times b)(\ab) = 0$ and otherwise
\begin{align*}
(a \times b)(\tau) &= 
\sum_{i=1}^m \prod_{j=1, j \neq i}^m a(\tau_j)b(\tau_i) \quad \text{ for } \tau = [\tau_1,\ldots,\tau_m].
\end{align*}
\label{th:norsett}
\end{lemm}

Proposition 1 in \cite{Celledoni09} states that the standard AVF method is a B-series method. We build on the proof of that proposition to prove the following result.

\begin{prop}\label{th:avf_bseries}
The generalized AVF method \eqref{eq:avfmexgen} is a B-series method.
\end{prop}
\begin{proof}
First we define $\hat{e} : T \cup\lbrace\emptyset\rbrace \rightarrow \mathbb{R}$ by $\hat{e}(\emptyset) = 1$ and $\hat{e}(\tau) = 0$ for all $\tau \neq \emptyset$. Then, assuming that the solution $\xh$ of \eqref{eq:avfmexgen} can be written as the B-series $\xh = B(\Phi,x)$, we find the B-series
\begin{align*}
h \int_0^1 f\big((1-\xi)x + \xi \xh\big) \, \mathrm{d}\xi &= h \int_0^1 f\big(B((1-\xi)\hat{e}+\xi\Phi,x)\big) \, \mathrm{d}\xi \\
&= \int_0^1 B\big(((1-\xi)\hat{e}+\xi\Phi)',x\big) \, \mathrm{d}\xi = B\big(\int_0^1 ((1-\xi)\hat{e}+\xi\Phi)'  \, \mathrm{d}\xi ,x\big).
\end{align*}
Setting $\theta\coloneqq \int_0^1 ((1-\xi)\hat{e}+\xi\Phi)'  \, \mathrm{d}\xi =  \int_0^1((1-\xi)\hat{e})'\, \mathrm{d}\xi + \int_0^1 (\xi\Phi)'  \, \mathrm{d}\xi = \int_0^1 (\xi\Phi)'  \, \mathrm{d}\xi$, we get
\begin{align}
\theta(\emptyset) = 0, \quad \theta(\ab) = 1, \quad \theta([\tau_1,\ldots,\tau_m]) = \frac{1}{m+1}\Phi(\tau_1) \cdots \Phi(\tau_m).
\label{eq:avfb}
\end{align}
Then we may rewrite \eqref{eq:avfmexgen} as
\begin{align*}
\xh &= x + \left( I + \sum_{n=2}^{p-1} h^n \sum_j b_{nj} \left( \prod_{k=1}^n f'(B(\phi_{njk},x)) + (-1)^n \prod_{k=1}^n f'(B(\phi_{nj(n-k+1)},x)) \right) \right) B(\theta,x)\\
&= x + B(\theta,x)+ \sum_{n=2}^{p-1} \sum_j b_{nj} \left( B(\phi_{nj1} \times \cdots \times \phi_{njn} \times \theta,x) + (-1)^n B(\phi_{njn} \times \cdots \times \phi_{nj1} \times \theta,x) \right) \\
&= B(\Phi,x),
\end{align*}
with
\begin{equation}\label{eq:Phi}
\Phi = \hat{e} + \theta+\sum_{n=2}^{p-1} \sum_j b_{nj} \left( \phi_{nj1} \times \cdots \times \phi_{njn} \times \theta + (-1)^n \, \phi_{njn} \times \cdots \times \phi_{nj1} \times \theta\right).
\end{equation}
\end{proof}

Comparing the B-series of the exact solution and the B-series of the solution of \eqref{eq:avfmexgen}, and noting that the elementary differentials are independent, we immediately get the following result.
\begin{theo}
The generalized AVF method \eqref{eq:avfmexgen} is of order $p$ if and only if
\begin{equation}
\Phi(\tau) = \frac{1}{\gamma(\tau)} \quad \text{ for } \lvert\tau\rvert \leq p,
\label{eq:avfoc}
\end{equation}
where $\Phi$ is given by \eqref{eq:Phi} and $\gamma$ is given by \eqref{eq:gamma}.
\end{theo}

\begin{table}[!ht]
\begin{center}
\begin{tabular}{|c|c|c|c|c|c|}
\hline
$\lvert\tau\rvert$ & $F(\tau)^i$ & $\tau$ & $\sigma(\tau)$ & $\gamma(\tau)$ & $\Phi(\tau)$\\ \hline
$1$	& $f^i$ & $\ab$ & $1$ & $1$ & $1$ \\ \hline
$2$	& $f^i_jf^j$ & $\aabb$ & $1$ &  $2$ & $\frac{1}{2}$\\ \hline
$3$	& $f^i_{jk}f^jf^k$ & $\aababb$ & $2$ &  $3$ & $\frac{1}{3}$ \\ 
	& $f^i_jf^j_kf^k$& $\aaabbb$ & $1$ &  $6$ & $\frac{1}{4}+2\sum_j b_{2j}$ \\ \hline
$4$	& $f^i_{jkl}f^jf^kf^l$ & $\aabababb$ & $6$ &  $4$ & $\frac{1}{4}$\\
	& $f^i_{jk}f^jf^k_lf^l$ & $\aabaabbb$ & $1$ &  $8$ & $\frac{1}{6}+\sum_{j,k}b_{2j}\phi_{2jk}(\ab)$\\
	& $f^i_{j}f^j_{kl}f^kf^l$ & $\aaababbb$ & $2$ &  $12$ & $\frac{1}{6} + 2\sum_{j,k} b_{2j}\phi_{2jk}(\ab)$\\
	& $f^i_{j}f^j_{k}f^k_lf^l$ & $\aaaabbbb$ & $1$ &  $24$ & $\frac{1}{8}+2\sum_j b_{2j}$\\ \hline
\end{tabular}
\end{center}
\caption{Elementary differentials and their coefficients in the B-series of the solution of \eqref{eq:avfmexgen}, up to fourth order.}
\label{tab:eldiffs1}
\end{table}

The terms $\Phi(\tau)$ can be found from \eqref{eq:Phi} by applying Lemma \ref{th:norsett} recursively, as illustrated by the following example.

\begin{example}
Consider $\tau = \aabaaabbbb$, and assume we have found $\Phi$ for all trees up to and including order four already, as given in Table \ref{tab:eldiffs1}. We have
$$
\theta(\aabaaabbbb) = \frac{1}{3}\Phi(\ab)\Phi(\aaabbb) = \frac{1}{3}(\frac{1}{4}+2\sum_jb_{2j}) = \frac{1}{12}+\frac{2}{3}\sum_jb_{2j}.
$$
Then we calculate
\begin{align*}
(\phi_{2j1}\times\phi_{2j2}\times \theta)(\aabaaabbbb) &= \phi_{2j1}(\ab)(\phi_{2j2}\times \theta)(\aaabbb) + \phi_{2j1}(\aaabbb)(\phi_{2j2}\times \theta)(\ab)\\
& = \phi_{2j1}(\ab)(\phi_{2j2}\times \theta)(\aaabbb) = \phi_{2j1}(\ab)\phi_{2j2}(\emptyset)\theta(\aabb) = \frac{1}{2}\phi_{2j1}(\ab),
\end{align*}
where we have used in the second equality that $(\phi_{2j2}\times \theta)(\ab) = \phi_{2j2}(\emptyset)\theta(\emptyset) = 0$. Similarly we find $(\phi_{2j2}\times\phi_{2j1}\times \theta)(\aabaaabbbb) = \frac{1}{2}\phi_{2j2}(\ab)$. Furthermore,
\begin{align*}
(\phi_{3j1}\times\phi_{3j2}\times\phi_{3j3}\times \theta)(\aabaaabbbb) &= \phi_{3j1}(\ab)(\phi_{3j2}\times\phi_{3j3}\times \theta)(\aaabbb) = \phi_{3j1}(\ab)\phi_{3j2}(\emptyset)(\phi_{3j3}\times \theta)(\aabb)\\
& = \phi_{3j1}(\ab)\phi_{3j3}(\emptyset)\theta(\ab) = \phi_{3j1}(\ab),
\end{align*}
and $(\phi_{3j3}\times\phi_{3j2}\times\phi_{3j1}\times \theta)(\aabaaabbbb) = \phi_{3j3}(\ab)$. Hence,
$$
\Phi(\aabaaabbbb) = \frac{1}{12}+\frac{2}{3}\sum_jb_{2j} + \frac{1}{2}\sum_j b_{2j}(\phi_{2j1}(\ab)+\phi_{2j2}(\ab)) + \sum_j b_{3j} (\phi_{3j1}(\ab)-\phi_{3j3}(\ab)).
$$
Now, if we assume the order condition \eqref{eq:avfoc} to be satisfied for all trees up to and including order four, we can replace $\sum_j b_{2j} = -\frac{1}{24}$ and $\sum_j b_{2j}(\phi_{2j1}(\ab)+\phi_{2j2}(\ab)) = -\frac{1}{24}$ in the above expression, use that $\gamma(\aabaaabbbb) = 30$, and get that \eqref{eq:avfoc} is satisfied for $\aabaaabbbb$ if and only if
\begin{equation}\label{eq:ordconex}
\sum_j b_{3j} (\phi_{3j1}(\ab)-\phi_{3j3}(\ab)) = -\frac{1}{720}.
\end{equation}
\end{example}

\subsection{Construction of higher order schemes}\label{sect:highorderavf}

As the size of the trees grows, finding $\Phi(\tau)$ from \eqref{eq:Phi} can become quite a cumbersome operation. Furthermore, we observe from Table \ref{tab:eldiffs1} that there are some equivalent order conditions for different trees.
Before presenting more convenient techniques for finding order conditions for the generalized AVF method, let us define some more concepts related to B-series and trees.

First, recall that the Butcher product of two trees $u = [u_1,\ldots,u_m]$ and $v = [v_1,\ldots,v_n]$ is given by $u \circ v = [u_1,u_2,\ldots,u_m,v]$. This operation is neither associative nor commutative, and in contrast to the practice in \cite{Hairer06}, we here take the product of several factors without parentheses to mean evaluation from right to left:
$$
u_1 \circ u_2 \circ \cdots \circ u_k \coloneqq u \circ ( u_2 \circ ( \cdots \circ u_k)).
$$
Given a forest $\mu = (\tau_1,\ldots,\tau_m)$, the tree obtained by grafting the roots of every tree in $\mu$ to a new root is denoted by $[\mu] =  [\tau_1,\ldots,\tau_m]$. Moreover, $\mu^{-1}(\tau)$ denotes the forest such that $[\mu^{-1}(\tau)] = \tau$.
We extend the maps $\phi: T \cup\lbrace\emptyset\rbrace \rightarrow \mathbb{R}$ and $\gamma: T \cup\lbrace\emptyset\rbrace \rightarrow \mathbb{R}$ to forests by the letting $\phi(\mu) = \prod_{i=1}^m \phi(\tau_i)$ and $\gamma(\mu) = \prod_{i=1}^m \gamma(\tau_i)$ for $\mu = (\tau_1,\ldots,\tau_m)$.

Consider now a tree $\tau$ consisting of $\lvert\tau\rvert$ nodes. We may number every tree from $1$ to $\lvert\tau\rvert$, starting at the root and going from left to right on the increasing levels above. For a given node $i \in \lbrace 1,\ldots,\lvert\tau\rvert\rbrace$ on level $n+1$, there exists a unique set of forests $\hat{\tau}^i = \lbrace\mu^i_1,\ldots,\mu^i_{n+1}\rbrace$ such that
$$
\tau=[\mu^i_1] \circ [\mu^i_2] \circ \cdots \circ [\mu^i_{n+1}].
$$
That is, labeling node $i$,
\begin{equation*}
\tau = 
\vcenter{\hbox{
     \tikz[grow=up,
      level distance=2.5\ml,
      sibling distance=2.5\ml,
      every node/.style={inner sep=0.22\ml}]{
    \node[circle,draw,fill=black]{}
    child{ node[circle,draw,fill=black]{}  child[dashed]{ node[circle,draw,fill=black,label=right:$i$]{} child[solid]{ node{$\mu_{n+1}^i$}}}
    child{ node{$\mu_2^i$}}}
    child{ node{$\mu_1^i$}}
    ;}
    }}
\end{equation*}

\begin{prop}\label{th:Phitwo}
The $\Phi$ of \eqref{eq:avfoc} can alternatively be found by
\begin{equation}\label{eq:Phitwo2}
\Phi(\tau) = \hat{e}(\tau) + \theta(\tau) +\sum_{i \text{ s.t. } n\geq 2}\Lambda(\hat{\tau}^i)
\end{equation}
where $\hat{e}(\emptyset) = 1$ and $\hat{e}(\tau) = 0$ for all $\tau \neq \emptyset$, $\theta(\emptyset) = 0$, $\theta(\ab) = 1$,
\begin{equation*}
\theta([\tau_1,\ldots,\tau_m]) = \frac{1}{m+1}\Phi(\tau_1)\cdots\Phi(\tau_m),
\end{equation*}
and
\begin{equation}
\Lambda(\hat{\tau}^i) = \theta([\mu_{n+1}^i]) \sum_j b_{nj} \left( \phi_{nj1}(\mu_1^i) \cdots \phi_{njn}(\mu_n^i)  + (-1)^n \phi_{njn}(\mu_1^i) \cdots \phi_{nj1}(\mu_n^i)  \right).\label{eq:Lambda}
\end{equation}
\end{prop}
\begin{proof}
Define $n_i$ so that $n_i+1$ is the level of node $i$. Collect the children of node $i$ in the set $C_i$. We have
\begin{equation*}
[\mu_{n_i+1}^i] = [\mu_{n_k}^k] \circ [\mu_{n_k+1}^k] \quad \text{ for all } k\in C_i,
\end{equation*}
and thus
\begin{equation*}
(a\times b)([\mu_{n_i+1}^i]) = \sum_{k\in C_i} a(\mu_{n_k}^k) b([\mu_{n_k+1}^k]).
\end{equation*}
Note also that $\mu_{n_i}^i = \mu_{n_i}^k = \mu_{n_k-1}^k$ if $k\in C_i$. Then we get
\begin{align*}
(\phi_{nj1} \times \cdots \times \phi_{njn} \times \theta)(\tau) &= (\phi_{nj1} \times \cdots \times \phi_{njn} \times \theta)([\mu_1^{1}])\\
&=\sum_{i_1\in C_1} \phi_{nj1}(\mu_1^{i_1})(\phi_{nj2} \times \cdots \times \phi_{njn} \times \theta)([\mu_2^{i_1}])\\
&=\sum_{i_1\in C_1} \phi_{nj1}(\mu_1^{i_1}) \sum_{i_2\in C_{i_1}} \phi_{nj2}(\mu_2^{i_2}) (\phi_{nj3} \times \cdots \times \phi_{njn} \times \theta)([\mu_3^{i_2}])\\
&=\sum_{i_1\in C_1}  \sum_{i_2\in C_{i_1}} \phi_{nj1}(\mu_1^{i_2})\phi_{nj2}(\mu_2^{i_2}) (\phi_{nj3} \times \cdots \times \phi_{njn} \times \theta)([\mu_3^{i_2}])\\
&\quad \vdots\\
&=\sum_{i_1\in C_1} \sum_{i_2\in C_{i_1}} \cdots \sum_{i_n\in C_{i_{n-1}}} \phi_{nj1}(\mu_1^{i_n}) \cdots \phi_{njn}(\mu_n^{i_n})\theta([\mu_{n+1}^{i_n}])\\
&=\sum_{i \text { on level } n+1} \phi_{nj1}(\mu_1^{i}) \cdots \phi_{njn}(\mu_n^{i})\theta([\mu_{n+1}^{i}]).
\end{align*}
Inserting this and the corresponding result for $(\phi_{njn} \times \cdots \times \phi_{nj1} \times \theta)(\tau)$ in \eqref{eq:Phi},
we get \eqref{eq:Lambda}.
\end{proof}

In \cite{Faou04,Chartier06}, conditions are derived for a B-series method to be energy preserving when applied to the system \eqref{eq:constSform}. In \cite{Quispel08}, while giving the AVF method as one such method, Quispel and McLaren present a general form of what they call energy-preserving linear combinations of rooted trees:
\begin{equation*}
\omega = 
\vcenter{\hbox{
     \tikz[grow=up,
      level distance=2.5\ml,
      sibling distance=2.5\ml,
      every node/.style={inner sep=0.22\ml}]{
    \node[circle,draw,fill=black]{}
    child{ node[circle,draw,fill=black]{}  child[dashed]{ node[circle,draw,fill=black]{} child[solid]{ node[circle,draw,fill=black]{}} child[solid]{ node{$\mu_n$}}}
    child{ node{$\mu_2$}}}
    child{ node{$\mu_1$}}
    ;}
    }}
    \quad +(-1)^n \quad
\vcenter{\hbox{
     \tikz[grow=up,
      level distance=2.5\ml,
      sibling distance=2.5\ml,
      every node/.style={inner sep=0.22\ml}]{
    \node[circle,draw,fill=black]{}
    child{ node[circle,draw,fill=black]{}  child[dashed]{ node[circle,draw,fill=black]{} child[solid]{ node[circle,draw,fill=black]{}} child[solid]{ node{$\mu_1$}}}
    child{ node{$\mu_{n-1}$}}}
    child{ node{$\mu_{n}$}}
    ;}
    }}
\end{equation*}
Here we give their result as a lemma, which is proved later by the proof of the more general Theorem \ref{th:preservingtrees}.
\begin{lemm}\label{th:quisplem}
Let $\mu_1, \ldots, \mu_n$ be $n$ arbitrary forests. Then, if $f(x)=S\nabla H(x)$ for some skew-symmetric constant matrix $S$, we have that $F(\omega)(x) \cdot \nabla H(x) = 0$ for
\begin{equation}
\omega = [\mu_1] \circ [\mu_2] \circ \cdots [\mu_n] \circ [\emptyset] + (-1)^n \, [\mu_n] \circ [\mu_{n-1}] \circ \cdots [\mu_1] \circ [\emptyset].
\label{eq:enpreslin}
\end{equation}
\end{lemm}

There is a connection between \eqref{eq:Phitwo2} and Lemma \ref{th:quisplem} such that instead of order conditions for every tree, we can calculate order conditions for every energy-preserving linear combination. To see this we start by collecting the leaf nodes, i.e.\ nodes with no children, of the tree $\tau$ in a set $I_l$ and the other nodes in the set $I_n$. If node $i \in I_n$, we may then use the relation
\begin{equation*}
\Lambda(\lbrace\mu^i_1,\ldots,\mu^i_{n},\mu^i_{n+1}\rbrace) = \theta([\mu_{n+1}^i]) \Lambda(\lbrace\mu^i_1,\ldots,\mu^i_{n},\emptyset\rbrace)
\end{equation*}
to find $\Lambda(\hat{\tau}^i)$ from the previously calculated $\Lambda$ for a smaller tree. Then if lower order conditions are satisfied, we have numerical values for these $\Lambda$.
The leaf nodes on the other hand, with their corresponding $\hat{\tau}^i = \lbrace\mu_1^i,\ldots,\mu_n^i,\emptyset\rbrace$, gives an energy-preserving linear combination \eqref{eq:enpreslin} which $\tau$ belongs to. If $i$ is on level two, this combination is simply $\tau-\tau=0$, and accordingly $\Lambda$ is not calculated for these nodes in \eqref{eq:Phitwo2}. Moreover, leaves on the same level have identical $\hat{\tau}^i$. Thus, a tree with leaves on $m$ different levels above level two will belong to at most $m$ non-zero energy-preserving linear combinations \eqref{eq:enpreslin}. For each of these combinations there is a corresponding order condition, with the left hand side given by \eqref{eq:Lambda}. The right hand side can be found by considering the individual trees.

If we assume the conditions for order $<p$ to be satisfied, we may replace \eqref{eq:avfoc} by
\begin{equation}
\sum_{i\in I_l}\Lambda(\hat{\tau}^i) = \frac{1}{\gamma(\tau)} - \hat{e}(\tau) - \sum_{i\in I_n}\frac{\Lambda(\lbrace\mu^i_1,\ldots,\mu^i_{n},\emptyset\rbrace)}{(\lvert\mu_{n+1}^i\rvert +1) \gamma(\mu_{n+1}^i)},
\label{eq:ordercondn2}
\end{equation}
where $\lvert\mu\rvert$ denotes the number of trees in the forest $\mu$. Note that $\Lambda(\lbrace \emptyset\rbrace) = 1$ and hence $\Lambda(\hat{\tau}^1) = \theta(\tau)$. Then we can calculate the numerical value for the right hand side and, if $\tau$ has leaves on only one level $>2$, find an order condition for both $\tau$ and the other tree in the combination \eqref{eq:enpreslin}. This warrants an example.

\begin{example}
Consider again the tree $\tau = \aabaaabbbb$, which is part of the energy-preserving linear combination
$
\omega = \aabaaabbbb-\aaaababbbb.
$
Ignoring the two nodes on level $2$, there are three nodes to calculate $\Lambda$ for: $i=1$, $i=4$ and $i=5$. We find
\begin{align*}
\Lambda(\hat{\tau}^1) &= \frac{1}{(2+1)\gamma(\ab)\gamma(\aaabbb)} = \frac{1}{3\cdot 1\cdot 6} = \frac{1}{18}\\
\Lambda(\hat{\tau}^4) &= \frac{1}{(1+1) \gamma(\ab)}\Lambda(\lbrace\ab,\emptyset,\emptyset \rbrace)=\frac{1}{2\gamma(\ab)}\Bigg(\frac{1}{\gamma(\aabaabbb)}-\frac{1}{3\gamma(\ab)\gamma(\aabb)} \Bigg) = \frac{1}{2}\Big( \frac{1}{8}-\frac{1}{6}\Big) = -\frac{1}{48},\\
\Lambda(\hat{\tau}^5) &= \Lambda(\lbrace \ab,\emptyset,\emptyset,\emptyset\rbrace) = \sum_j b_{3j}\big(\phi_{3j1}(\ab)-\phi_{3j3}(\ab)\big).
\end{align*}
The right hand side of \eqref{eq:ordercondn2} becomes
\begin{equation*}
\frac{1}{\gamma(\tau)} - \frac{1}{18} - (-\frac{1}{48}) = \frac{1}{30} - \frac{1}{18} + \frac{1}{48} =-\frac{1}{720},
\end{equation*}
and we have the order condition \eqref{eq:ordconex} for the linear combination $\aabaaabbbb - \aaaababbbb$.
\end{example}

If there are leaves on $r>1$ different levels above level two, things get slightly more complicated. Then we get $r$ different terms on the left hand side of \eqref{eq:ordercondn2} and we need to consider the order condition for $\tau$ and the $r$ trees it forms energy-preserving linear combinations with, so that we get an equation for every energy-preserving combination of these trees, also those not including $\tau$. This is illustrated by the following example.

\begin{example}
The tree $\aabaabaabbbb$ forms energy-preserving combinations with both $\aaabaababbbb$ and $\aaabbaababbb$. Thus we have to calculate \eqref{eq:ordercondn2} for all three trees to find order conditions for the corresponding linear combinations. Starting with $\tau = \aabaabaabbbb$, which has three nodes above level two, two leaves and one non-leaf, we get
\begin{align*}
\Lambda(\hat{\tau}^4) &= \Lambda(\lbrace \ab,\aabb,\emptyset\rbrace) = \sum_j b_{2j}\big(\phi_{2j1}(\ab)\phi_{2j2}(\aabb)+\phi_{2j2}(\ab)\phi_{2j1}(\aabb)\big),\\
\Lambda(\hat{\tau}^5) &= \frac{1}{(1+1) \gamma(\ab)}\Lambda(\lbrace\ab,\ab,\emptyset \rbrace)=\frac{1}{2\gamma(\ab)} \frac{1}{2} \Bigg(\frac{1}{\gamma(\aabaababbb)}-\frac{1}{3\gamma(\ab)\gamma(\aababb)} \Bigg) = \frac{1}{2}\frac{1}{2}\Big( \frac{1}{15}-\frac{1}{9}\Big) = -\frac{1}{90},\\
\Lambda(\hat{\tau}^6) &= \Lambda(\lbrace \ab,\ab,\emptyset,\emptyset\rbrace) = \sum_j b_{3j}\big(\phi_{3j1}(\ab)\phi_{3j2}(\ab)-\phi_{3j3}(\ab)\phi_{3j2}(\ab)\big) \\
&= \sum_j b_{3j}\phi_{3j2}(\ab)(\phi_{3j1}-\phi_{3j3})(\ab).
\end{align*}
For the right hand side of \eqref{eq:ordercondn2}, we get
\begin{equation*}
\frac{1}{\gamma(\tau)} - \frac{1}{(2+1)\gamma(\ab)\gamma(\aabaabbb)}-(-\frac{1}{90}) = \frac{1}{48} - \frac{1}{3 \cdot 1 \cdot 8}+\frac{1}{90} = -\frac{7}{720},
\end{equation*}
and hence the order condition for $\aabaabaabbbb$ is
\begin{equation}\label{eq:excond1}
\sum_j b_{2j}\big(\phi_{2j1}(\ab)\phi_{2j2}(\aabb)+\phi_{2j2}(\ab)\phi_{2j1}(\aabb)\big) +  \sum_j b_{3j}\phi_{3j2}(\ab)(\phi_{3j1}-\phi_{3j3})(\ab) = -\frac{7}{720}.
\end{equation}
Similarly we calculate \eqref{eq:ordercondn2} for $\aaabaababbbb$,
\begin{equation}\label{eq:excond2}
\sum_{jk} b_{2j} \phi_{2jk}(\aababb) - 2 \sum_j b_{3j}\phi_{3j2}(\ab)(\phi_{3j1}-\phi_{3j3})(\ab) = -\frac{1}{120},
\end{equation}
and for $\aaabbaababbb$,
\begin{equation}\label{eq:excond3}
\sum_{jk} b_{2j} \phi_{2jk}(\aababb) + 2 \sum_j b_{2j}\big(\phi_{2j1}(\ab)\phi_{2j2}(\aabb)+\phi_{2j2}(\ab)\phi_{2j1}(\aabb)\big) = -\frac{1}{36}.
\end{equation}
Combining \eqref{eq:excond1}, \eqref{eq:excond2} and \eqref{eq:excond3}, we get the equivalent system of equations
\begin{align}
\sum_{j} b_{3j}\phi_{3j2}(\ab)(\phi_{3j1}(\ab)-\phi_{3j3}(\ab)) &= \frac{1}{240}+\alpha,\label{eq:excondall1}\\
\sum_{j} b_{2j} (\phi_{2j1}(\ab)\phi_{2j2}(\aabb) + \phi_{2j1}(\aabb)\phi_{2j2}(\ab)) &= -\frac{1}{72}-\alpha,\\
\sum_{j,k}b_{2j}\phi_{2jk}(\aababb)&= 2\alpha,\label{eq:excondall3}
\end{align}
where the choice of $\alpha \in \mathbb{R}$ is arbitrary. The order conditions \eqref{eq:excondall1}--\eqref{eq:excondall3} can be associated to the linear combinations $\aabaabaabbbb - \aaabaababbbb$, $\aaabbaababbb + \aabaabaabbbb$ and $\aaabaababbbb + \aaabbaababbb$, respectively.
\end{example}

\begin{table}[!ht]
\begin{center}
\begin{tabular}{|c|c|c|}
\hline
$\lvert\tau\rvert$ & $\omega$ & Order condition\\ \hline
$1$	& $\ab$ & -- \\ \hline
$2$	& -- & -- \\ \hline
$3$	& $\aaabbb$ & $\sum_j b_{2j} = -\frac{1}{24}$ \\ \hline
$4$	& $\aaababbb+\aabaabbb$ & $\sum_{j,k} b_{2j} \phi_{2jk}(\ab)= -\frac{1}{24}$ \\ \hline
$5$	& $\aababaabbb+\aaabababbb$ & $\sum_{j,k}b_{2j}\phi_{2jk}(\ab)^2=-\frac{1}{40}$  \\
	& $\aabaababbb$ & $\sum_j b_{2j}\phi_{2j1}(\ab)\phi_{2j2}(\ab)=-\frac{1}{90}$ \\
	& $\aabaaabbbb-\aaaababbbb$ & $\sum_j b_{3j}(\phi_{3j1}(\ab)-\phi_{3j3}(\ab))=-\frac{1}{720}$ \\
	& $\aaabbaabbb+\aaabaabbbb$ & $\sum_{j,k}b_{2j}\phi_{2jk}(\aabb)=-\frac{1}{60}$ \\
	& $\aaaaabbbbb$ & $\sum_j b_{4j}=\frac{1}{240}$ \\ \hline
$6$	& $\aabababaabbb+\aaababababbb$ & $\sum_{j,k}b_{2j}\phi_{2jk}(\ab)^3=-\frac{1}{60}$  \\
	& $\aababaababbb + \aabaabababbb$ & $\sum_j b_{2j}(\phi_{2j1}(\ab)^2\phi_{2j2}(\ab)+\phi_{2j1}(\ab)\phi_{2j2}(\ab)^2)=-\frac{1}{72}$ \\
	& $\aababaaabbbb-\aaaabababbbb$ & $\sum_j b_{3j}(\phi_{3j1}(\ab)^2-\phi_{3j3}(\ab)^2)=-\frac{1}{720}$ \\
	& $\aabaabbaabbb+\aaababaabbbb$ & $\sum_{j,k}b_{2j}\phi_{2jk}(\ab)\phi_{2jk}(\aabb)=-\frac{1}{96}$ \\
	& $\aabaabaabbbb - \aaabaababbbb$ & $\sum_{j} b_{3j}\phi_{3j2}(\ab)(\phi_{3j1}(\ab)-\phi_{3j3}(\ab)) = \frac{1}{240}+\alpha_1$ \\
	& $\aabaaaabbbbb + \aaaaababbbbb$ & $\sum_{j} b_{4j}(\phi_{4j1}(\ab)+\phi_{4j4}(\ab)) = \frac{1}{240}$ \\
	& $\aaabbaababbb + \aabaabaabbbb$ & $\sum_{j} b_{2j} (\phi_{2j1}(\ab)\phi_{2j2}(\aabb) + \phi_{2j1}(\aabb)\phi_{2j2}(\ab)) = -\frac{1}{72}-\alpha_1 $ \\
	& $\aaabbaaabbbb - \aaaabaabbbbb$ & $\sum_{j} b_{3j} (\phi_{3j1}(\aabb) - \phi_{3j3}(\aabb)) = -\frac{1}{180}-\alpha_2$ \\
	& $\aaabaababbbb + \aaabbaababbb$ & $\sum_{j,k}b_{2j}\phi_{2jk}(\aababb)= 2\alpha_1$ \\
	& $\aaabaaabbbbb + \aaabbaaabbbb$ & $\sum_{j,k}b_{2j}\phi_{2jk}(\aaabbb)= \alpha_2$ \\
	& $\aaabaaabbbbb + \aaaabaabbbbb$ & $\sum_{j} b_{4j}(\phi_{4j2}(\ab)+\phi_{4j3}(\ab)) = -\frac{1}{1440}-\alpha_2$ \\
	\hline
\end{tabular}
\end{center}
\caption{Energy-preserving linear combinations of elementary differentials, and their associated order conditions for the scheme \eqref{eq:avfmexgen}, up to sixth order. The coefficients $\alpha_1, \alpha_2 \in \mathbb{R}$ are arbitrary.}
\label{tab:avfoc}
\end{table}
By considering the order conditions in Table \ref{tab:avfoc}, we find a fifth order scheme of the form \eqref{eq:avfmexgen} given by
\begin{equation}\label{eq:5thavfm}
\begin{split}
\frac{\xh-x}{h} = & \bigg(I - \frac{5}{136} h^2 \big( f'(z_2)f'(z_3)  + f'(z_3)f'(z_2) \big) - \frac{1}{102} h^2 f'(x)f'(x)\\
& + \frac{1}{288}h^3  \big( f'(x) f'(x) f'(z_1) + f'(z_1) f'(x) f'(x) \big) \\
& + \frac{1}{120} h^4 f'(x) f'(x) f'(x) f'(x) \bigg)
\int_0^1 f((1-\xi)x + \xi \xh) \, \mathrm{d}\xi,
\end{split}
\end{equation}
where
\begin{equation*}
z_1 = x + \frac{2}{5} h f(x), \qquad z_2 = x + \frac{17+\sqrt{17}}{30} h f(z_1), \qquad z_3 = x + \frac{17-\sqrt{17}}{30} h f(z_1).
\end{equation*}
A symmetric sixth order scheme is given by
\begin{equation}\label{eq:6thavfm}
\begin{split}
\frac{\xh-x}{h} = & \, \bigg(I - \frac{13}{360} h^2 f'\big(\xb+\frac{\sqrt{13}}{26} h f(\xb-\frac{3\sqrt{13}}{26} h f(\xb))\big) f'\big(\xb-\frac{\sqrt{13}}{26} h f(\xb+\frac{3\sqrt{13}}{26} h f(\xb))\big)  \\ 
& - \frac{13}{360} h^2 f'\big(\xb-\frac{\sqrt{13}}{26} h f(\xb+\frac{3\sqrt{13}}{26} h f(\xb))\big) f'\big(\xb+\frac{\sqrt{13}}{26} h f(\xb-\frac{3\sqrt{13}}{26} h f(\xb))\big)  \\ 
& - \frac{1}{180} h^2 f'(x) f'(x) - \frac{1}{180} h^2 f'(\xh) f'(\xh)  \\ 
& + \frac{1}{720}h^3  f'(\xb-\frac{1}{2}hf(\xb)) f'(\xb) f'(\xb+\frac{1}{2}hf(\xb)) \\
& -  \frac{1}{720}h^3  f'(\xb+\frac{1}{2}hf(\xb)) f'(\xb) f'(\xb-\frac{1}{2}hf(\xb)) \\
& + \frac{1}{120} h^4 f'(\xb) f'(\xb) f'(\xb) f'(\xb) \bigg)
\int_0^1 f((1-\xi)x + \xi \xh) \, \mathrm{d}\xi,
\end{split}
\end{equation}
where $\xb = \frac{x+\xh}{2}$. If we wish to calculate the matrix in front of the integral explicitly, we have a non-symmetric sixth order scheme given by
\begin{equation}\label{eq:6thavfmexp}
\begin{split}
\frac{\xh-x}{h} = & \, \bigg(I - \frac{13}{360} h^2 \big(f'(z_6)f'(z_7) + f'(z_7)f'(z_6)\big) - \frac{1}{180} h^2 \big(f'(x) f'(x) + f'(z_1) f'(z_1)\big)  \\
& + \frac{1}{720}h^3 \big( f'(x) f'(z_2) f'(z_3) -   f'(z_3) f'(z_2) f'(x)\big) \\
& + \frac{1}{120} h^4 f'(z_2) f'(z_2) f'(z_2) f'(z_2) \bigg)
\int_0^1 f((1-\xi)x + \xi \xh) \, \mathrm{d}\xi,
\end{split}
\end{equation}
with
\begin{equation*}
\begin{split}
&z_1 = x + \frac{1}{4} h f(x) + \frac{3}{4} h f\big(x+\frac{2}{3} h f(x + \frac{1}{3}h f(x))\big), \qquad z_2 = x + \frac{1}{2} h f(x), \qquad z_3 = x + h f(z_2),\\
&z_4 = \frac{1}{2}(x+z_3) - \frac{3\sqrt{13}}{26} h f(z_2), \qquad z_5 = \frac{1}{2}(x+z_3) + \frac{3\sqrt{13}}{26} h f(z_2), \\
&z_6 = \frac{1}{2}(x+z_1) + \frac{\sqrt{13}}{26} h f(z_4), \qquad z_7 = \frac{1}{2}(x+z_1) - \frac{\sqrt{13}}{26} h f(z_5).
\end{split}
\end{equation*}

\section{AVF discrete gradient methods for general skew-gradient systems}\label{sect:avfgen}

We will now build on the results of the previous section by generalizing the results to the situation where $S(x)$ in the skew-gradient system \eqref{eq:odeform} is not necessarily constant. 
Consider therefore now an ODE of the form \eqref{eq:odeform}, and set again $g \coloneqq \nabla H$. By Taylor expansion of $x$ around $t=t_0$ we get
\begin{align*}
x(t_0+h) =& x + h S g + \frac{h^2}{2}(S' g S g + S g' S g) + \frac{h^3}{6}(S'' g (S g, S g) + 2 S' g' (S g, S g) + S g'' (S g, S g) \\
&+ S' g S' g S g + S' g S g' S g + S g' S' g S g + S g' S g' S g) + \mathcal{O}(h^4),
\end{align*}
where $x\coloneqq x(t_0)$, and $S$, $g$ and their derivatives are evaluated in $x$.
Introducing the notation $f^\circ \coloneqq S' g$ and $f^\bullet \coloneqq S g'$, we can write this in the abbreviated form
\begin{equation}
\begin{split}
x(t_0+h) =& x + h f + \frac{h^2}{2}(f^\circ f + f^\bullet f) + \frac{h^3}{6}(f^{\circ\circ} (f, f) + 2 f^{\circ\bullet} (f, f) + f^{\bullet\bullet} (f, f) \\
&+ f^{\circ} f^{\circ} f + f^{\circ} f^{\bullet} f + f^{\bullet} f^{\circ} f + f^{\bullet} f^{\bullet} f) + \mathcal{O}(h^4).
\end{split}
\label{eq:exactexp}
\end{equation}

\subsection{Skew-gradient systems and P-series}
A \textit{P-series} is given by
\begin{equation}\label{eq:fullpseries}
\begin{split}
P(\phi,(x,y))  = 
\begin{pmatrix}
  \phi(\emptyset)x + \sum_{\tau\in TP_{\ab}}\frac{h^{\lvert\tau\rvert}}{\sigma(\tau)}\phi(\tau)F(\tau)(x,y) \\
  \phi(\emptyset)y + \sum_{\tau\in TP_{\AB}}\frac{h^{\lvert\tau\rvert}}{\sigma(\tau)}\phi(\tau)F(\tau)(x,y) 
 \end{pmatrix},
\end{split}
\end{equation}
where $TP$ is the set of rooted bi-colored trees and $TP_{\ab}$ and $TP_{\AB}$ are the subsets of $TP$ whose roots are black and white, respectively \cite[Section III.2]{Hairer06}. 
The bi-colored trees are built recursively; starting with $\ab$ and $\AB$, we let $\tau = [\tau_1,\ldots,\tau_m]_{\ab}$ be the tree you get by grafting the roots of $\tau_1,\ldots,\tau_m$ to a black root and $\tau = [\tau_1,\ldots,\tau_m]_{\AB}$ the tree you get by grafting $\tau_1,\ldots,\tau_m$ to a white root. No subscript, i.e.\ $\tau = [\tau_1,\ldots,\tau_m]$, means grafting to a black root.

The exact solution of a partitioned system
\begin{equation}
\begin{split}
\dot{x} &= f(x,y), \quad x(t_0) = x_0,\\
\dot{y} &= g(x,y), \quad y(t_0) = y_0,
\end{split}
\label{eq:partsys}
\end{equation}
can be written as $(x(t_0+h),y(t_0+h)) = P(1/\gamma,(x_0,y_0))$, where the coefficient $\gamma$ is given by $\gamma(\emptyset) = \gamma(\ab) =\gamma(\AB) = 1$ and \eqref{eq:gamma}. As noted in \cite{Cohen11}, setting $f(x,y) \coloneqq S(y) \nabla H(x)$, the skew-gradient system \eqref{eq:odeform} can be written as \eqref{eq:partsys} with $g=f$. When $g=f$, 
all coefficients and the elementary differentials $F(\tau)$ in \eqref{eq:fullpseries} are given independently of the color of the root. 
Thus for the system \eqref{eq:odeform}, it suffices to consider
\begin{equation}\label{eq:pseries}
P(\phi,x) = \phi(\emptyset)x + \sum_{\tau\in TP_{\ab}}\frac{h^{\lvert\tau\rvert}}{\sigma(\tau)}\phi(\tau)F(\tau)(x),
\end{equation}
and we have that the exact solution of \eqref{eq:odeform} can be written as $x(t_0+h) = P(1/\gamma,x_0)$. Breaking slightly with convention, we define a P-series to be the single row version \eqref{eq:pseries} in the remainder of this paper. Denoting black-rooted subtrees by $\tau_i$ and white-rooted subtrees by $\bar{\tau}_i$, the elementary differentials $F(\tau)$ for the skew-gradient system are given recursively by $F(\ab)(x) = F(\AB)(x) = S(x)\nabla H(x)$, and
\begin{equation}\label{eq:eldiff}
F(\tau)(x) = S^{(l)}D^m\nabla H(F(\tau_1)(x),\ldots,F(\tau_m)(x),F(\bar{\tau}_1)(x),\ldots,F(\bar{\tau}_l)(x))
\end{equation}
for both $\tau = [\tau_1,\ldots,\tau_m,\bar{\tau}_1,\ldots,\bar{\tau}_l]_{\ab}$ and $\tau = [\tau_1,\ldots,\tau_m,\bar{\tau}_1,\ldots,\bar{\tau}_l]_{\AB}$.
The bi-colored trees in $TP_{\ab}$ and their corresponding elementary differentials $F$ are given up to order three in Table \ref{tab:trees}. The number of trees grows very quickly with the order; see \url{https://oeis.org/A000151}.
\begin{table}[!ht]
\begin{center}
\begin{tabular}{|c|c|c|c|c|c|c|}
\hline
$\lvert\tau\rvert$ & $F(\tau)^i$ & $F(\tau)$ & $\tau$ & $\alpha(\tau)$ & $\gamma(\tau)$ & $\sigma(\tau)$ \\ \hline
$1$ & $S^i_j g^j$ & $f$ & $\ab$ & $1$ & $1$ & $1$  \\ \hline
$2$ & $S^i_{jk} g^j S^{k}_l g^l$ & $f^\circ f$ & $\aABb$ & $1$ & $2$ & $1$ \\
 	& $S^i_j g^j_{k} S^{k}_l g^l$ & $f^\bullet f$ & $\aabb$ & $1$ & $2$ & $1$ \\ \hline
$3$ & $S^{i}_{j k m} g^{j} S^{k}_l g^l S^{m}_n g^n$ & $f^{\circ\circ} (f,f)$ & $\aABABb$ & $1$ & $3$ & $2$ \\ 
	& $S^{i}_{jk} g^j_{m} S^{k}_l g^l S^{m}_n g^n$ & $f^{\circ\bullet} (f,f)$ & $\aABabb$ & $2$ & $3$ & $1$ \\ 
	& $S^{i}_j g^j_{k m} S^{k}_l g^l S^{m}_n g^n$ & $f^{\bullet\bullet} (f,f)$ & $\aababb$ & $1$ & $3$ & $2$ \\ 
	& $S^{i}_{jk} g^{j} S^{k}_{lm} g^l S^{m}_n g^n$ & $f^{\circ} f^{\circ} f$ & $\aAABBb$ & $1$ & $6$ & $1$ \\ 
	& $S^{i}_j g^j_{k} S^{k}_{lm} g^{l} S^{m}_n g^n$ & $f^{\bullet} f^{\circ} f$ & $\aaABbb$ & $1$ & $6$ & $1$ \\ 
	& $S^{i}_{jk} g^{j} S^{k}_l g^{l}_m S^{m}_n g^n$ & $f^{\circ} f^{\bullet} f$ & $\aAabBb$ & $1$ & $6$ & $1$ \\ 
	& $S^{i}_j g^j_{k} S^{k}_l g^l_{m} S^{m}_n g^n$ & $f^{\bullet} f^{\bullet} f$ & $\aaabbb$ & $1$ & $6$ & $1$ \\ \hline
\end{tabular}
\end{center}
\caption{Bi-colored trees and their elementary differentials up to third order.}
\label{tab:trees}
\end{table}

The following lemma is Lemma III.2.2 in \cite{Hairer06} amended to fit our setting.
\begin{lemm}
Let $P(a,x)$ and $P(b,x)$ be two P-series with $a(\emptyset) = b(\emptyset) = 1$. Then
\begin{equation*}
h S(P(a,x)) \nabla H(P(b,x)) = P(a \vee b,x),
\end{equation*}
where $(a \vee b)(\emptyset) = 0$, $(a \vee b)(\ab) = 1$, and
\begin{equation*}
(a \vee b)(\tau) = a(\tau_1) \cdots a(\tau_m) b(\bar{\tau}_1) \cdots b(\bar{\tau}_l) \quad \text{ for } \tau = [\tau_1,\ldots,\tau_m,\bar{\tau}_1,\ldots,\bar{\tau}_l]_{\ab}.
\end{equation*}
\label{th:hlw2}
\end{lemm}

\begin{prop}\label{th:avfdgps}
The AVF discrete gradient scheme
\begin{equation}
\frac{\xh-x}{h} = S\bigg(\frac{x+\xh}{2}\bigg) \int_0^1 \nabla H ((1-\xi)x + \xi \xh) \, \mathrm{d}\xi
\label{eq:avfdgm}
\end{equation}
is a second order P-series method.
\end{prop}
\begin{proof}
As in the proof of Proposition \ref{th:avf_bseries}, we define $\hat{e}$ by $\hat{e}(\emptyset) = 1$ and $\hat{e}(\tau) = 0$ for all $\tau \neq \emptyset$. Now, assume that the solution $\xh$ of \eqref{eq:avfdgm} can be written as the P-series $\xh = P(\Phi,x)$. Then, using Lemma \ref{th:hlw2}, we find the P-series
\begin{align*}
h \, S\bigg(\frac{x+\xh}{2}\bigg) \int_0^1 \nabla H ((1-\xi)x + \xi \xh) \, \mathrm{d}\xi &= h \, S\big(P\big(\frac{1}{2}\hat{e}+\frac{1}{2}\Phi,x\big)\big) \, \int_0^1 \nabla H \big( P((1-\xi)\hat{e} + \xi \Phi, \, x)\big) \, \mathrm{d}\xi \\
&= \int_0^1 h \, S\big(P\big(\frac{1}{2}\hat{e}+\frac{1}{2}\Phi, \, x\big)\big) \,\nabla H \big( P((1-\xi)\hat{e} + \xi \Phi, \, x)\big) \, \mathrm{d}\xi \\
&= P \bigg( \int_0^1 \big(\big( \frac{1}{2}\hat{e}+\frac{1}{2}\Phi \big) \vee \big((1-\xi)\hat{e}+\xi\Phi\big) \big) \, \mathrm{d}\xi, \, x \bigg).
\end{align*}
Thus we get $\Phi = \hat{e} + \int_0^1 \big(\big( \frac{1}{2}\hat{e}+\frac{1}{2}\Phi \big) \vee \big((1-\xi)\hat{e}+\xi\Phi\big) \big) \, \mathrm{d}\xi = \hat{e} + \int_0^1 \big(\big( \frac{1}{2} \Phi \big) \vee \big(\xi\Phi\big) \big) \, \mathrm{d}\xi$. That is,
\begin{align*}
\Phi(\emptyset) = 1, \quad \Phi(\ab) = 1, \quad \Phi([\tau_1,\ldots,\tau_m,\bar{\tau}_1,\ldots,\bar{\tau}_l]) = \frac{1}{(m+1) 2^l}\Phi(\tau_1) \cdots \Phi(\tau_m)\Phi(\bar{\tau}_1) \cdots \Phi(\bar{\tau}_l).
\end{align*}
Writing out the first few terms of the series, we have
\begin{equation*}
\begin{split}
\xh =& \, x + h f + \frac{h^2}{2}(f^\circ f + f^\bullet f) + h^3(\frac{1}{8}f^{\circ\circ} (f, f) + \frac{1}{4} f^{\circ\bullet} (f, f) + \frac{1}{6}f^{\bullet\bullet} (f, f) \\
&+ \frac{1}{4}f^{\circ} f^{\circ} f + \frac{1}{4} f^{\circ} f^{\bullet} f + \frac{1}{4} f^{\bullet} f^{\circ} f + \frac{1}{4} f^{\bullet} f^{\bullet} f) + \mathcal{O}(h^4),
\end{split}
\end{equation*}
which, after comparing with the expanded exact solution \eqref{eq:exactexp}, we see is of order two.
\end{proof}

The following lemma is obtained in a manner similar to Lemma \ref{th:norsett}, i.e.\ Theorem $2.2$ in \cite{Norsett79}, and hence we present it without its proof.
\begin{lemm}
Let $P(a,x)$, $P(b,x)$ and $P(c,x)$ be three P-series with $a(\emptyset) = b(\emptyset) = 1$ and $c(\emptyset) = 0$. Then 
$$h \, S(P(a,x))\nabla^2 H(P(b,x)) P(c,x) = P((a,b) \times c,x)$$
with $((a,b) \times c)(\emptyset) = ((a,b) \times c)(\ab) = 0$ and otherwise
\begin{align}\label{eq:porder}
((a,b) \times c)(\tau) &= 
\sum_{i=1}^m \prod_{j=1, j \neq i}^m \prod_{k=1}^l a(\bar{\tau}_k) b(\tau_j) c(\tau_i) \quad \text{ for } \tau = [\tau_1,\ldots,\tau_m,\bar{\tau}_1,\ldots,\bar{\tau}_l].
\end{align}
\label{th:norsettP}
\end{lemm}
Note that $\left\lbrace\emptyset\right\rbrace$ counts as both a black-rooted and a white-rooted tree. Hence we have e.g.
\begin{equation*}
((a,b) \times c)(\aABabb) = a(\AB)b(\emptyset)c(\ab) = a(\ab) c(\ab),
\end{equation*}
where we also use that $a(\AB) = a(\ab)$.

We now present a subclass of the AVF discrete gradient method, for which we will find order conditions using Lemma \ref{th:hlw2} and Lemma \ref{th:norsettP}. This subclass is every AVF discrete gradient method for which the approximation of $S(x)$ can be written in the form
\begin{equation}
\begin{split}
\overline{S}(x,\xh,h) = & \sum_{n=0}^{p-1} h^n \sum_j b_{nj} \Bigg(\prod_{k=1}^n S(\zb_{njk}) \nabla^2 H(z_{njk}) \cdot S(\zb_{nj(n+1)}) \\
& + (-1)^n \, S(\zb_{nj(n+1)}) \prod_{k=1}^n \nabla^2 H(z_{nj(n-k+1)}) S(\zb_{nj(n-k+1)})\Bigg),
\end{split}
\label{eq:avfS}
\end{equation}
where, if $\xh$ is the solution of
\begin{equation*}
\frac{\xh-x}{h} = \overline{S}(x,\xh,h)\dg_{\text{AVF}} H(x,\xh),
\end{equation*}
each $z_{njk} \coloneqq z_{njk}(x,\xh,h) = P(\phi_{njk},x)$ and each $\zb_{njk} \coloneqq \zb_{njk}(x,\xh,h) = P(\psi_{njk},x)$ can be written as a P-series with $\phi_{njk}(\emptyset) = \psi_{njk}(\emptyset) = 1$ for all $n,j,k$. We require that $\sum_j b_{0j} = \frac{1}{2}$, which ensures that \eqref{eq:avfS} is a consistent approximation of $S(x)$.
\begin{theo}
The discrete gradient scheme \eqref{eq:dgm} with the AVF discrete gradient \eqref{eq:avfdg} and the approximation of $S(x)$ given by \eqref{eq:avfS} is a P-series method.
\end{theo}
\begin{proof}
Generalizing the argument in the proof of Proposition \ref{th:avfdgps}, we find the P-series
\begin{align*}
h \, S\big(P(a,x)\big) \int_0^1 \nabla H ((1-\xi)x + \xi \xh) \, \mathrm{d}\xi &= P \bigg( \int_0^1 \big(a \vee \big((1-\xi)\hat{e}+\xi\Phi\big) \big) \, \mathrm{d}\xi, \, x \bigg),
\end{align*}
where $\bar{\theta}(a) \coloneqq \int_0^1 \big(a \vee \big((1-\xi)\hat{e}+\xi\Phi\big) \big) \, \mathrm{d}\xi = \int_0^1 \big(a \vee \xi\Phi\big) \big) \, \mathrm{d}\xi$, so that $\bar{\theta}(a)(\emptyset) = 0$, $\bar{\theta}(a)(\ab) = 1$, and
\begin{align}\label{eq:astarb}
\bar{\theta}(a)([\tau_1,\ldots,\tau_m,\bar{\tau}_1,\ldots,\bar{\tau}_l]) = \frac{1}{m+1}\Phi(\tau_1) \cdots \Phi(\tau_m)a(\bar{\tau}_1) \cdots a(\bar{\tau}_l).
\end{align}
Thus we may write the solution $\xh$ found from applying the scheme \eqref{eq:dgm} with the AVF discrete gradient \eqref{eq:avfdg} and $\overline{S}(x,\xh,h)$ given by \eqref{eq:avfS} as
\begin{equation}\label{eq:avfps} 
\begin{split}
\xh &= x + \sum_{n=0}^{p-1} h^n \sum_j b_{nj} \left( \prod_{k=1}^n S(P(\psi_{njk},x)) \nabla^2  H(P(\phi_{njk},x)) \cdot P(\bar{\theta}(\psi_{nj(n+1)}),x) \right.\\
& \quad \left. + (-1)^n \prod_{k=1}^n S(P(\psi_{nj(n-k+2)},x))\nabla^2 H(P(\phi_{nj(n-k+1)},x)) \cdot P(\bar{\theta}(\psi_{nj1}),x) \right)\\
&= x + \sum_{n=0}^{p-1} \sum_j b_{nj} \left( P((\psi_{nj1}, \phi_{nj1}) \times (\psi_{nj2}, \phi_{nj2}) \times \cdots \times (\psi_{njn} , \phi_{njn}) \times \bar{\theta}(\psi_{nj(n+1)}),x) \right.\\
& \quad \left. + (-1)^n P((\psi_{nj(n+1)}, \phi_{njn}) \times (\psi_{njn}, \phi_{nj(n-1)}) \times \cdots \times (\psi_{nj2} , \phi_{nj1}) \times \bar{\theta}(\psi_{nj1}),x) \right) \\
&= P(\Phi,x), 
\end{split}
\end{equation}
with
\begin{equation}
\begin{split}
\Phi = & \hat{e} + \sum_{n=0}^{p-1} \sum_j b_{nj} \, \big( (\psi_{nj1}, \phi_{nj1}) \times \cdots \times (\psi_{njn}, \phi_{njn}) \times \bar{\theta}(\psi_{nj(n+1)}) \\
& + (-1)^n \, (\psi_{nj(n+1)}, \phi_{njn}) \times \cdots \times (\psi_{nj2}, \phi_{nj1}) \times \bar{\theta}(\psi_{nj1})\big).
\end{split}
\label{eq:Phigen}
\end{equation}
\end{proof}
\begin{theo}
The AVF discrete gradient method with $\overline{S}$ given by \eqref{eq:avfS} is of order $p$ if and only if
\begin{equation}
\Phi(\tau) = \frac{1}{\gamma(\tau)} \quad \text{ for } \lvert\tau\rvert \leq p.
\label{eq:avfdgmoc}
\end{equation}
\end{theo}

The values $\Phi(\tau)$ can be found from \eqref{eq:Phigen} using \eqref{eq:porder} recursively and then \eqref{eq:astarb}.
However, a more convenient approach is derived in the next subsection. 

\subsection{Order conditions}\label{sect:hos}
This subsection is devoted to generalizations of the results in Section \ref{sect:highorderavf} to the cases where $S(x)$ is not necessarily constant. To that end, for a tree $\tau \in TP_{\ab}$, we cut off all branches between black and white nodes and denote the mono-colored tree we are left with by $\tau^{b}$. We number the nodes in that tree as before, from $1$ to $\lvert\tau^{b}\rvert$, and reattach the cut-off parts to the tree to get $\tau$ again. 
Let $\mu$ denote a forest of black-rooted trees and $\eta$ a forest of white-rooted trees. Then, for a given node $i \in [1,\ldots,\lvert\tau^{b}\rvert]$ on level $n+1$, there exists a unique set of forests $\hat{\tau}^i = \lbrace(\mu^i_1,\eta^i_1),\ldots,(\mu^i_{n+1},\eta^i_{n+1})\rbrace$ such that
\begin{equation*}\label{eq:tauhat}
\tau=[(\mu^i_1,\eta^i_1)] \circ [(\mu^i_2,\eta^i_2)] \circ \cdots \circ [(\mu^i_{n+1},\eta^i_{n+1})].
\end{equation*}
That is,
\begin{equation*}
\tau = 
\vcenter{\hbox{
     \tikz[grow=up,
      level distance=2.5\ml,
      sibling distance=2.5\ml,
      every node/.style={inner sep=0.22\ml}]{
    \node[circle,draw,fill=black]{}
    child{ node[circle,draw,fill=black]{}  child[dashed]{ node[circle,draw,fill=black,label=right:$i$]{} child[solid]{ node{$\eta^i_{n+1}$}} child[solid]{ node{$\mu^i_{n+1}$}} }
    child{ node{$\eta^i_2$}} child{ node{$\mu^i_2$}}}
    child{ node{$\eta^i_1$}}
    child{ node{$\mu^i_1$}}
    ;}
    }}
\end{equation*}
Now we can generalize Proposition \ref{th:Phitwo} as follows.

\begin{prop}\label{th:Phitwogen}
The $\Phi$ of \eqref{eq:Phigen} can be found by
\begin{equation}\label{eq:Phitwo2gen}
\Phi(\tau) = \hat{e}(\tau) + \sum_{i=1}^{\lvert\tau^{b}\rvert}\Lambda(\hat{\tau}^i)
\end{equation}
where $\hat{e}(\emptyset) = 1$ and $\hat{e}(\tau) = 0$ for all $\tau \neq \emptyset$, and
\begin{equation}\label{eq:Lambdagen}
\begin{split}
\Lambda(\hat{\tau}^i) = & \theta([\mu_{n+1}^i]) \sum_j b_{nj} \bigg( \psi_{nj1}(\eta_1^i)\phi_{nj1}(\mu_1^i) \cdots \psi_{njn}(\eta_n^i)\phi_{njn}(\mu_n^i)\psi_{nj(n+1)}(\eta_{n+1}^i)  \\
& + (-1)^n \psi_{nj(n+1)}(\eta_{1}^i)\phi_{njn}(\mu_1^i)\psi_{njn}(\eta_2^i) \cdots \phi_{nj1}(\mu_n^i)\psi_{nj1}(\eta_{n+1}^i)  \bigg),
\end{split}
\end{equation}
with
\begin{equation*}
\theta([\tau_1,\ldots,\tau_m]) = \frac{1}{m+1}\Phi(\tau_1)\cdots\Phi(\tau_m).
\end{equation*}
\end{prop}
\begin{proof}
Defining $n_i$ and $C_i$ as in the proof of Proposition \ref{th:Phitwo}, we have
\begin{align*}
&[(\mu_{n_i+1}^i,\eta_{n_i+1}^i)] = [(\mu_{n_k}^k,\eta_{n_k}^k)] \circ [(\mu_{n_k+1}^k,\eta_{n_k+1}^k)]  \quad \text{ for all } k\in C_i,\\
&((a,b)\times c)([(\mu_{n_i+1}^i,\eta_{n_i+1}^i)]) = \sum_{k\in C_i} a(\eta_{n_k}^k)b(\mu_{n_k}^k) c([\mu_{n_k+1}^k,\eta_{n_k+1}^k]).
\end{align*}
Observe that $\bar{\theta}(a)([\mu,\eta]) = a(\eta) \theta([\mu])$. For $n=0$ we have
\begin{equation*}
\bar{\theta}(\psi_{0j1})(\tau) = \bar{\theta}(\psi_{0j1})([\mu_1^{1},\eta_1^{1}]) = \psi_{0j1}(\eta_{1}^{1})\theta([\mu_{1}^{1}]),
\end{equation*}
and for $n>0$ we get
\begin{align*}
((\psi_{nj1}, &\phi_{nj1}) \times \cdots \times (\psi_{njn}, \phi_{njn}) \times \bar{\theta}(\psi_{nj(n+1)}))(\tau)\\
&= ((\psi_{nj1}, \phi_{nj1}) \times \cdots \times (\psi_{njn}, \phi_{njn}) \times \bar{\theta}(\psi_{nj(n+1)}))([\mu_1^{1},\eta_1^{1}])\\
&=\sum_{i_1\in C_1} \psi_{nj1}(\eta_1^{i_1})\phi_{nj1}(\mu_1^{i_1})((\psi_{nj2},\phi_{nj2}) \times \cdots \times (\psi_{njn},\phi_{njn}) \times \bar{\theta}(\psi_{nj(n+1)}))([\mu_2^{i_1},\eta_2^{i_1}])\\
&\quad \vdots\\
&=\sum_{i_1\in C_1} \cdots \sum_{i_n\in C_{i_{n-1}}} \psi_{nj1}(\eta_1^{i_n})\phi_{nj1}(\mu_1^{i_n}) \cdots \psi_{njn}(\eta_n^{i_n})\phi_{njn}(\mu_n^{i_n})\bar{\theta}(\psi_{nj(n+1)})([\mu_{n+1}^{i_n},\eta_{n+1}^{i_n}])\\
&=\sum_{i \text { on level } n+1} \psi_{nj1}(\eta_1^{i})\phi_{nj1}(\mu_1^{i}) \cdots \psi_{njn}(\eta_n^{i})\phi_{njn}(\mu_n^{i})\psi_{nj(n+1)}(\eta_{n+1}^{i})\theta([\mu_{n+1}^{i}]).
\end{align*}
Inserting this and the corresponding result for $((\psi_{nj(n+1)}, \phi_{njn}) \times \cdots \times (\psi_{nj2}, \phi_{nj1}) \times \bar{\theta}(\psi_{nj1}))(\tau)$ in \eqref{eq:Phigen}, we get \eqref{eq:Lambdagen}.
\end{proof}

Note that if $\tau$ only has black nodes, we have $\Lambda(\hat{\tau}^1) = \theta(\tau) \sum_j b_{0j} (\psi_{0j1}(\emptyset)+\psi_{0j1}(\emptyset)) = \theta(\tau)$, and also $\Lambda(\hat{\tau}^i) = 0$ for all nodes $i$ on level $2$. Thus \eqref{eq:Phitwo2gen} simplifies to \eqref{eq:Phitwo2}.

Like for the constant $S$ case, the order conditions can be given for energy-preserving linear combinations of elementary differentials instead for each elementary differential. In the following generalization of Lemma \ref{th:quisplem}, we state that the energy-preserving linear combinations of bi-colored rooted trees are given by
\begin{equation*}
\omega = 
\vcenter{\hbox{
     \tikz[grow=up,
      level distance=2.5\ml,
      sibling distance=2.5\ml,
      every node/.style={inner sep=0.22\ml}]{
    \node[circle,draw,fill=black]{}
    child{ node[circle,draw,fill=black]{}  child[dashed]{ node[circle,draw,fill=black]{} child[solid]{ node[circle,draw,fill=black]{} child[solid]{ node{$\eta_{n+1}$}}} child[solid]{ node{$\eta_n$}} child[solid]{ node{$\mu_n$}}}
    child{ node{$\eta_2$}} child{ node{$\mu_2$}}}
    child{ node{$\eta_1$}}
    child{ node{$\mu_1$}}
    ;}
    }}
    \quad +(-1)^n \quad
\vcenter{\hbox{
     \tikz[grow=up,
      level distance=2.5\ml,
      sibling distance=2.5\ml,
      every node/.style={inner sep=0.22\ml}]{
    \node[circle,draw,fill=black]{}
    child{ node[circle,draw,fill=black]{}  child[dashed]{ node[circle,draw,fill=black]{} child[solid]{ node[circle,draw,fill=black]{} child[solid]{ node{$\eta_{1}$}}} child[solid]{ node{$\eta_2$}} child[solid]{ node{$\mu_1$}}}
    child{ node{$\eta_{n}$}} child{ node{$\mu_{n-1}$}}}
    child{ node{$\eta_{n+1}$}}
    child{ node{$\mu_{n}$}}
    ;}
    }}
\end{equation*}

\begin{theo}\label{th:preservingtrees}
Let $\mu_1, \mu_2, \ldots, \mu_n$ be arbitrary forests of black-rooted trees and $\eta_1, \eta_2, \ldots, \eta_{n+1}$ arbitrary forests of white-rooted trees. Given $f(x)=S(x)\nabla H(x)$, where $S(x)$ is a skew-symmetric matrix, and elementary differentials defined by \eqref{eq:eldiff}, the linear combinations of trees given by
\begin{equation}
\omega = [(\mu_1,\eta_1)] \circ \cdots \circ [(\mu_{n},\eta_{n})] \circ [\eta_{n+1}] + (-1)^n \, [(\mu_{n},\eta_{n+1})] \circ \cdots \circ [(\mu_{1},\eta_{2})] \circ [\eta_{1}]
\label{eq:enpreslingen}
\end{equation}
are energy-preserving in the sense that $F(\omega)(x) \cdot \nabla H(x) = 0$.

\end{theo}
\begin{proof}
For any forest of black-rooted trees $\mu_j$, we have $F([\mu_j]\circ[\emptyset]) = S B_j S\nabla H$ for some symmetric matrix $B_j$, suppressing the argument $x$. Similarly, for a forest of white-rooted trees $\eta_j$, we have $F([\eta_j]) = W_j\nabla H$ for some skew-symmetric matrix $W_j$. Note that the empty forest is considered both a black-rooted and a white-rooted forest, and accordingly we have $F([\emptyset]\circ[\emptyset]) = F(\aabb) =  S(\nabla^2H)S\nabla H $ and $F([\emptyset]) = F(\ab) = S\nabla H$.
For these matrices $B_j$ and $W_j$ corresponding to the forests $\mu_j$ and $\eta_j$, we get
\begin{equation*}
F\big([(\mu_1,\eta_1)] \circ \cdots \circ [(\mu_{n},\eta_{n})] \circ [\eta_{n+1}]\big) = W_1 B_1 W_2 B_2 \cdots B_n W_{n+1} \nabla H.
\end{equation*}
We have
\begin{equation*}
(W_1 B_1 W_2 B_2 \cdots B_n W_{n+1})^T =
\begin{cases}
-W_{n+1} B_n W_n B_{n-1} \cdots B_1 W_1 & \text{if } n \text{ even,}\\
W_{n+1} B_n W_n B_{n-1} \cdots B_1 W_1 & \text{if } n \text{ odd.}
\end{cases}
\end{equation*}
Thus $F(\omega)(x)$ is a skew-symmetric matrix times $\nabla H(x)$, and the statement in the above theorem follows directly.
\end{proof}

\begin{example}\label{ex:bbbbw}
We show that the combination $\aaabaAbbBb+ \aABabaabbb$ is energy-preserving.\\
\begin{align*}
\aaabaAbbBb: & \quad (f^\bullet f^{\bullet\bullet}(f,f^\circ f))^i = S^{i}_jg^j_{k}S^{k}_lg^l_{mo}S^{m}_n g^n S^{o}_{pq} g^p S^{q}_r g^r = S^{i}_j g^j_{k}S^{k}_l g^l_{mo}S^{m}_n g^n S^{o}_{pq} S^{q}_r g^r g^p.\\
\aABabaabbb: & \quad (f^{\circ\bullet\bullet}(f,f,f^\bullet f))^i = S^{i}_{jk} g^j_{mo}S^{k}_lg^lS^{m}_ng^nS^{o}_p g^p_{q}S^{q}_r g^r = S^{i}_{jk}S^{k}_l g^l g^j_{mo}S^{m}_ng^n S^{o}_r g^r_{q} S^{q}_p g^p.
\end{align*}
For this linear combination in the form \eqref{eq:enpreslingen}, we have $\eta_1 = \eta_2 = \emptyset, \eta_3 = \AB$, $\mu_1 =\emptyset, \mu_2 = \ab$, with the corresponding matrices $W_1=W_2=S, (W_3)^{i}_j = S^{i}_{jk} S^{k}_l g^l$ and $B_1 = \nabla^2 H, (B_2)^{j}_m = g^j_{km} S^{k}_lg^l$. Thus we get
$$\aaabaAbbBb + \aABabaabbb = f^\bullet f^{\bullet\bullet}(f,f^\circ f) + f^{\circ\bullet\bullet}(f,f,f^\bullet f) = Z\nabla H,$$
where $Z \coloneqq S(\nabla^2 H)S B_2 W_3 + W_3 B_2 S (\nabla^2 H)S$ is a skew-symmetric matrix.
\end{example}

For bi-colored trees, we define a node on the tree $\tau$ to be a leaf if it is a leaf on the corresponding cut tree $\tau^{b}$ by the definition of leaves given in the previous section. 
We let $I_l$ be the set of leaves and $I_n$ the set of non-leaf nodes which are also in $\tau^{b}$, so that $I_l \cup I_n = [ 1,\ldots,\lvert\tau^{b}\rvert]$. 
In contrast to the case with mono-colored trees, a leaf $i$ on level one or two of a bi-colored tree may give rise to a non-zero energy-preserving linear combination; it does so if and only if $\eta^i_{k} \neq \emptyset$ for any $k = 1,2$. Accordingly, $\Lambda(\hat{\tau}^i)$ is calculated in \eqref{eq:Phitwo2gen} also when $n=0,1$. Furthermore, two leaves $i$ and $j$ on the same level will belong to two different energy-preserving combinations if $\eta^i_{n+1} \neq \eta^j_{n+1}$. Therefore we now simply state that a tree with $r$ leaves, also including the lower two levels, belong to at most $r$ non-zero linear combinations. We thus get $r$ terms on the left hand side of
\begin{equation}
\sum_{i\in I_l}\Lambda(\hat{\tau}^i) = \frac{1}{\gamma(\tau)} - \hat{e}(\tau) - \sum_{i\in I_n}\frac{\Lambda(\lbrace(\mu^i_1,\eta_1^i),\ldots,(\mu^i_{n},\eta^i_{n}),(\emptyset,\eta^i_{n+1})\rbrace)}{(\lvert\mu_{n+1}^i\rvert +1) \gamma(\mu_{n+1}^i)},
\label{eq:ordercondn2gen}
\end{equation}
which is equivalent to \eqref{eq:avfdgmoc} if we assume the conditions for lower order to be satisfied.

\begin{example}
Consider the tree $\tau = \aAbaabBb$, which is part of the energy-preserving linear combination $\aAbaabBb - \aABababb$. Assume that the order conditions up to and including order three are all satisfied. The cut tree $\tau^{b} = \aababb$ has three nodes of which two are leaves. Node number $2$ is a leaf on level $2$ with $\eta^2_1 = \eta^2_2 = \emptyset$, and thus gives $\Lambda(\hat{\tau}^2)=0$. We find for the other two,
\begin{align*}
\Lambda(\hat{\tau}^1) &= \Lambda(\lbrace ((\ab,\aABb),\emptyset)\rbrace) = \frac{1}{(\lvert\mu_{1}^1\rvert+1) \gamma(\mu_{1}^1) }\Lambda\big(\lbrace (\emptyset,\emptyset)\rbrace\big) = \frac{1}{(2+1) \gamma(\ab)\gamma(\aABb)}\frac{1}{\gamma(\ab)} = \frac{1}{6},\\
\Lambda(\hat{\tau}^3) &= \Lambda(\lbrace (\ab,\emptyset),(\emptyset,\AB)\rbrace) = \sum_j b_{1j}(\phi_{1j1}(\ab)\psi_{1j2}(\AB)-\psi_{1j1}(\AB)\phi_{1j1}(\ab))= \sum_j b_{1j}\phi_{1j1}(\ab)(\psi_{1j2}-\psi_{1j1})(\ab).
\end{align*}
For the right hand side of \eqref{eq:ordercondn2gen} we get
\begin{equation*}
\frac{1}{\gamma(\tau)}-\Lambda(\hat{\tau}^1) = \frac{1}{8}-\frac{1}{6} = -\frac{1}{24},
\end{equation*}
and thus the order condition
\begin{equation*}
\sum_j b_{1j}\phi_{1j1}(\ab)(\psi_{1j2}-\psi_{1j1})(\ab) = -\frac{1}{24}
\end{equation*}
for the energy-preserving linear combination $\aAbaabBb - \aABababb$.
\end{example}

\begin{table}[!ht]
\begin{center}
\begin{tabular}{|c|c|c|}
\hline
$\lvert\tau\rvert$ & $\omega$ & Order condition\\ \hline
$1$	& $\ab$ & $2 \sum_j b_{0j} = 1$ \\ \hline
$2$	& $\aABb$ & $2 \sum_j b_{0j} \psi_{0j1}(\ab) = \frac{1}{2}$ \\ \hline
$3$	& $\aABABb$ & $2 \sum_j b_{0j} \psi_{0j1}(\ab)^2 = \frac{1}{3}$ \\
	& $\aaabbb$ & $\sum_j b_{2j} = -\frac{1}{24}$ \\
	& $\aAabBb$ & $2 \sum_j b_{0j} \psi_{0j1}(\aabb) = \frac{1}{6}$ \\
	& $\aAABBb$ & $2 \sum_j b_{0j} \psi_{0j1}(\aABb) = \frac{1}{6}$ \\
	& $\aaABbb - \aABabb$ & $\sum_j b_{1j} (\psi_{1j2}-\psi_{1j1})(\ab) = -\frac{1}{12}$ \\ \hline
$4$	& $\aABABABb$ & $2 \sum_j b_{0j} \psi_{0j1}(\ab)^3 = \frac{1}{4}$ \\
	& $\aaabABbb$ & $2 \sum_j b_{2j} \psi_{2j2}(\ab) = -\frac{1}{24}$ \\
	& $\aAababBb$ & $2 \sum_j b_{0j} \psi_{0j1}(\aababb)=\frac{1}{12}$ \\
	& $\aAABabBb$ & $2 \sum_j b_{0j} \psi_{0j1}(\aABabb)=\frac{1}{12}$ \\
	& $\aAABABBb$ & $2 \sum_j b_{0j} \psi_{0j1}(\aABABb)=\frac{1}{12}$ \\
	& $\aABAabBb$ & $2 \sum_j b_{0j} \psi_{0j1}(\ab)\psi_{0j1}(\aabb)=\frac{1}{8}$ \\
	& $\aABAABBb$ & $2 \sum_j b_{0j} \psi_{0j1}(\ab)\psi_{0j1}(\aABb)=\frac{1}{8}$  \\
	& $\aaababbb + \aabaabbb$ & $\sum_j b_{2j} (\phi_{2j1}+\phi_{2j2})(\ab) = -\frac{1}{24}$ \\
	& $\aaABABbb - \aABABabb$ & $\sum_j b_{1j} (\psi_{1j2}(\ab)^2-\psi_{1j1}(\ab)^2) = -\frac{1}{12}$\\
	& $\aAbaabBb - \aABababb$ & $\sum_j b_{1j}\phi_{1j1}(\ab)(\psi_{1j1}-\psi_{1j0})(\ab) = -\frac{1}{24}$ \\
	& $\aAaabbBb$ & $2 \sum_j b_{0j}\psi_{0j1}(\aaabbb) = \frac{1}{24}$ \\
	& $\aAAabBBb$ & $2 \sum_j b_{0j}\psi_{0j1}(\aAabBb) = \frac{1}{24}$ \\
	& $\aAaABbBb$ & $2 \sum_j b_{0j}\psi_{0j1}(\aaABbb) = \frac{1}{24}$ \\
	& $\aAAABBBb$ & $2 \sum_j b_{0j}\psi_{0j1}(\aAABBb) = \frac{1}{24}$ \\
	& $\aaaAbbBb + \aaBAabbb$ & $\sum_j b_{2j}(\psi_{2j1}+\psi_{2j3})(\ab) = -\frac{1}{24}$ \\
	& $\aaAabBbb - \aAabBabb$ & $\sum_j b_{1j}(\psi_{1j2}-\psi_{1j1})(\aabb) = -\frac{1}{24}$ \\
	& $\aaAAbBBb - \aAABBabb$ & $\sum_j b_{1j}(\psi_{1j2}-\psi_{1j1})(\aABb) = -\frac{1}{24}$ \\ \hline
\end{tabular}
\end{center}
\caption{Linear combinations $\omega$ of bi-colored black-rooted trees corresponding to energy-preserving elementary differentials of $f(x) = S(x)\nabla H(x)$, where $S(x)$ is a skew-symmetric matrix, as well as their associated order conditions for the discrete gradient method \eqref{eq:dgm} with the AVF discrete gradient \eqref{eq:avfdg} and $\overline{S}(x,\xh,h)$ given by \eqref{eq:avfS}.}
\label{tab:avfdgmoc}
\end{table}

Even though the number of black-rooted bi-colored trees grows very quickly, e.g.\ to $26$ for $\lvert\tau\rvert=4$ and $107$ for $\lvert\tau\rvert=5$, finding and satisfying the order conditions is not as daunting a task as it might first appear. First of all, it suffices to find order conditions for the non-zero linear combinations given by \eqref{eq:enpreslingen}. Moreover, a couple key observations simplifies the process further:
\begin{itemize}
\item The large number of trees $\tau$ for which $\tau^{b} = \ab$, i.e.\ trees with no black nodes on level $2$, are all energy-preserving. They can be written $\tau = [\eta_1^1]$, and their order condition is given by
\begin{equation*}
2 \sum_j b_{0j} \psi_{0j1}(\eta_1^1) = \frac{1}{\gamma(\tau)}.
\end{equation*}
\item For trees that are identical except for the colors of the descendants of white nodes, it suffices to calculate one order condition. E.g.\ for $\aAgbgbBb$ we have the order condition $2b_{0j} \psi_{0j1}(\AbbbbB) = \frac{1}{12}$, where each of the gray nodes may be black or white. To satisfy these conditions, it is natural to require that $\bar{z}_{0j1}$ in \eqref{eq:avfS} is a B-series \textit{up to} order $p-1$.
\end{itemize}

From the order conditions displayed in Table \ref{tab:avfdgmoc} we find that one second order scheme is given by \eqref{eq:dgm} using the AVF discrete gradient \eqref{eq:avfdg} and an explicit skew-symmetric approximation of $S$ given by $\overline{S}(x,\cdot,h) = S(x + \frac{1}{2} h f(x))$.
A third order scheme is obtained if we instead use the skew-symmetric approximation of $S$ explicitly given by 
\begin{equation}
\begin{split}
\overline{S}(x,\cdot,h) =& \, \frac{1}{4} S(x) + \frac{3}{4}S(z_2) + \frac{1}{4} h \, \big( S(z_1)\nabla^2 H(x)S(x) - S(x)\nabla^2 H(x)S(z_1)\big) \\
& - \frac{1}{12} h^2 \, S(x)\nabla^2 H(x)S(x)\nabla^2 H(x)S(x),
\end{split}
\label{eq:3thSavf}
\end{equation}
where $z_1 = x + \frac{1}{3} h f(x)$, $z_2 = x + \frac{2}{3} h f(z_1)$.

A symmetric fourth order scheme is given by 
\eqref{eq:dgm} using the AVF discrete gradient \eqref{eq:avfdg} and the skew-symmetric approximation of $S$
\begin{equation}\label{eq:4thSimp}
\begin{split}
\overline{S}(x,\xh,h) =& \, \frac{1}{2} S\Big(\xb - \frac{1}{\sqrt{12}} h f\big(\xb + \frac{1}{\sqrt{12}} h f(\xb)\big) \Big) + \frac{1}{2} S\Big(\xb + \frac{1}{\sqrt{12}} h f\big(\xb - \frac{1}{\sqrt{12}} h f(\xb)\big) \Big)\\
& + \frac{1}{2} h \, S\big(\xb+\frac{1}{12}h f(\xb)\big)\nabla^2 H(\xb)S\big(\xb-\frac{1}{12}h f(\xb)\big) \\
& - \frac{1}{2} h \, S\big(\xb-\frac{1}{12}h f(\xb)\big)\nabla^2 H(\xb)S\big(\xb+\frac{1}{12}h f(\xb)\big) \\
& - \frac{1}{12} h^2 \, S(\xb)\nabla^2 H(\xb)S(\xb)\nabla^2 H(\xb)S(\xb),
\end{split}
\end{equation}
where $\xb = \frac{x+\xh}{2}$.
Another fourth order scheme is obtained if we instead use the explicit skew-symmetric approximation of $S$ found by
\begin{equation}
\begin{split}
\overline{S}(x,\cdot,h) =& \, \frac{1}{2} (S(z_5+z_6)+S(z_5-z_6)) + \frac{1}{12} h \, \big( S(z_2)\nabla^2 H(z_1)S(x) - S(x)\nabla^2 H(z_1)S(z_2)\big) \\
& - \frac{1}{12} h^2 \, S(z_1)\nabla^2 H(z_1)S(z_1)\nabla^2 H(z_1)S(z_1),
\end{split}
\label{eq:4thSavf}
\end{equation}
where
\begin{align*}
z_1 &= x + \frac{1}{2} h f(x), & z_3 = x + h f(z_2), & \hspace{30pt} z_5 = \frac{1}{3}(x+z_1+z_2) + \frac{1}{12} (-z_3+z_4),\\
z_2 &= x + h f(z_1), & z_4 = x + h f(z_3), & \hspace{30pt} z_6 = \frac{\sqrt{3}}{36}(7 x-2 z_1-4z_2+z_3-2z_4).
\end{align*}

\section{Order conditions for general discrete gradient methods}\label{sect:gengen}
We will now generalize the results of the two previous sections to discrete gradient methods with a general discrete gradient, as defined by \eqref{eq:dgcond1}--\eqref{eq:dgcond2}. To that end, we introduce two new series in the vein of B- and P-series, as well as related tree structures.

\subsection{The constant $S$ case}\label{sect:genconst}

Consider mono-colored rooted trees whose nodes can have two different shapes: the circle shape of the nodes in trees of B-series, but also a triangle shape. Let $TG$ be the set of such trees whose leaves are always circles. 
That is, from the first tree $\ab$, every tree $\tau \in TG$ can be built recursively through 
\begin{equation*}
[\tau_1,\ldots,\tau_m]_{\ab}, \quad [\tau_1,\ldots,\tau_m]_{\tb}, \quad \tau_1,\ldots,\tau_m \in TG,
\end{equation*}
which denotes the grafting of the trees $\tau_1,\ldots,\tau_m$ to a root ${\ab}$ or ${\tb}$, respectively.
The elementary differentials $F(\tau)$ corresponding to a tree $\tau \in TG$ are likewise defined recursively by $F(\ab)(x) = f(x) = S \nabla H(x)$ and
\begin{equation*}
F(\tau)(x) =
\begin{cases}
S D^m\nabla H(x)(F(\tau_1)(x),\ldots,F(\tau_m)(x)) & \text{for } \tau = [\tau_1,\ldots,\tau_m]_{\ab}, \\
S D_2^{m-1}Q(x,x) (F(\tau_1)(x),\ldots,F(\tau_m)(x)) & \text{for } \tau = [\tau_1,\ldots,\tau_m]_{\tb}.
\end{cases}
\end{equation*}
We can then define a generalization of B-series which includes these elementary differentials.
\begin{definition}
A G-series is a formal series of the form
\begin{equation}\label{eq:gseries}
G(\phi,x) = \phi(\emptyset)x + \sum_{\tau\in TG}\frac{h^{\lvert\tau\rvert}}{\sigma(\tau)}\phi(\tau)F(\tau)(x),
\end{equation}
where $\phi : TG \cup \lbrace\emptyset\rbrace \rightarrow \mathbb{R}$ is an arbitrary mapping, and the symmetry coefficient $\sigma$ is given by \eqref{eq:symcoeff}.
\end{definition}

The G-series of the exact solution is given by $x(t_0+h) = G(\xi,x(t_0))$, with
\begin{equation}
\xi(\tau) = 
\begin{cases}
\frac{1}{\gamma(\tau)} & \text{if } \tau \in T,\\
0 & \text{otherwise}.
\end{cases}
\label{eq:gexact}
\end{equation}
For use in the remainder of this paper, we generalize the Butcher product by the definition
\begin{equation*}
u \circ v = [u_1,\ldots,u_m,v]_{\st}, \quad \text{for } u = [u_1,\ldots,u_m]_{\st}, \quad \st \in \lbrace \ab,\tb\rbrace.
\end{equation*}
Furthermore, we let $\lvert\tau\rvert$ denote the total number of nodes in $\tau$, and $\lvert\tau\rvert_{\st}$ the number of nodes of type $\st$. Let $SG$ be the set of tall trees in $TG$; that is, the set of trees with only one node on each level. For a tree $\tau \in TG$, number every tree from $1$ to $\lvert\tau\rvert$, as before. For any node $i$ on level $n+1$, we define the stem $s^i \in SG$ to be the tall tree consisting of the nodes connecting the root to node $i$, including the root and node $i$. Denote the $j^\text{th}$ node of $s^i$ by $s^i_j$, so that $s^i_1$ is the root and $s^i_{n+1} = i$. Then we have a unique set of forests $\hat{\tau}^i = \lbrace\mu^i_1,\ldots,\mu^i_{n+1}\rbrace$ such that
$$
\tau=[\mu^i_1]_{s^i_1} \circ [\mu^i_2]_{s^i_2} \circ \cdots \circ [\mu^i_{n+1}]_{s^i_{n+1}}.
$$
That is,
\begin{equation*}
\tau = 
\vcenter{\hbox{
     \tikz[grow=up,
      level distance=2.5\ml,
      sibling distance=2.5\ml,
      every node/.style={inner sep=0.22\ml}]{
    \node{$s_{1}^i$}
    child{ node{$s_{2}^i$}  child[dashed]{ node{$s^i_{n+1}$} child[solid]{ node{$\mu_{n+1}^i$}}}
    child{ node{$\mu_2^i$}}}
    child{ node{$\mu_1^i$}}
    ;}
    }}
\end{equation*}

The following lemma is a generalization of Lemma \ref{th:norsett} to G-series. Its proof is very similar to the proof of \cite[Theorem 2.2]{Norsett79}, and hence omitted. 
\begin{lemm}
Let $G(a,x)$ and $G(b,x)$ be two G-series with $a(\emptyset) = 1$ and $b(\emptyset) = 0$. Then the G-series $h S \nabla^2 H(G(a,x))G(b,x) = G(a\times b,x)$ is given by $(a \times b)(\emptyset) = (a \times b)(\ab) = 0$ and otherwise
\begin{equation*}
(a \times b)(\tau) = 
\begin{cases}
\sum_{i=1}^m \prod_{j=1, j \neq i}^m a(\tau_j)b(\tau_i) & \text{for } \tau = [\tau_1,\ldots,\tau_m]_{\ab}, \\
0 & \text{for } \tau = [\tau_1,\ldots,\tau_m]_{\tb}.
\end{cases}
\end{equation*}
Moreover, $h S Q(x,G(a,x))G(b,x) = G(a\otimes b,x)$, with $(a \otimes b)(\emptyset) = (a \otimes b)(\ab) = 0$ and otherwise
\begin{equation*}
(a \otimes b)(\tau) = 
\begin{cases}
0 & \text{for } \tau = [\tau_1,\ldots,\tau_m]_{\ab}, \\
\sum_{i=1}^m \prod_{j=1, j \neq i}^m a(\tau_j)b(\tau_i) & \text{for } \tau = [\tau_1,\ldots,\tau_m]_{\tb}.
\end{cases}
\end{equation*}
\end{lemm}

To every stem $s\in SG$ of height $n+1 = \lvert s \rvert$, we associate coefficients $b_{sj}$ and $\phi_{sjk}$. Letting $s_k$ be the $k^{\text{th}}$ node of $s$, we define the function
\begin{equation*}
R(\phi_{sjk},x) \coloneqq
\begin{cases}
    \nabla^2 H(G(\phi_{sjk},x))      & \quad \text{if } s_k = \ab,\\
    Q(x,G(\phi_{sjk},x))   & \quad \text{if } s_k = \tb.
\end{cases}
\end{equation*}
Then we have $h S R(\phi_{sjk},x) G(b,x) = G(\phi_{sjk} \diamond b)$, with $(\phi_{sjk} \diamond b)(\emptyset) = (\phi_{sjk} \diamond b)(\ab) = 0$ and
\begin{equation*}
(\phi_{sjk} \diamond b)(\tau) = 
\begin{cases}
\sum_{i=1}^m \prod_{j=1, j \neq i}^m \phi_{sjk}(\tau_j)b(\tau_i) & \text{for } \tau = [\tau_1,\ldots,\tau_m]_{s_k},\\
0 & \text{if root of } \tau \neq s_k.
\end{cases}
\end{equation*}

Consider now the class of skew-symmetric and consistent approximations to $S$ that can be written in the form
\begin{equation}
\begin{split}
\overline{S}(x,\xh,h) =& \sum_{s \in SG} h^{n} \sum_{j} b_{sj} \left( \prod_{k=1}^n S R(\phi_{sjk},x) + (-1)^{\lvert s\rvert_{\ab}-1} \prod_{k=1}^n S R(\phi_{sj(n-k+1)},x) \right) S 
\end{split}
\label{eq:GseriesS}
\end{equation}
whenever $\xh$ is the solution of
\begin{equation*}
\frac{\xh-x}{h} = \overline{S}(x,\xh,h)\dg H(x,\xh),
\end{equation*}
with $\phi_{sjk}(\emptyset) = 1$ for every $s,j,k$, and with $\sum_j b_{\ab j} = \frac{1}{2}$.

\begin{lemm}\label{th:avf_gseries}
The discrete gradient method \eqref{eq:dgm} with $\overline{S}(x,\xh,h)$ given by \eqref{eq:GseriesS} and $\dg H \in C^\infty(\mathbb{R}^d\times \mathbb{R}^d,\mathbb{R}^d)$ is a G-series method when applied to a constant $S$ skew-gradient system \eqref{eq:constSform}.
\end{lemm}
\begin{proof}
Assume that the solution $\xh$ of \eqref{eq:dgm} with $\overline{S}(x,\xh,h)$ given by \eqref{eq:GseriesS} can be written as the G-series $\xh = G(\Phi,x)$. Then, using Lemma \ref{th:Qlem} and $\dg H(x,x) = \nabla H(x)$,
\begin{align*}
h S \dg H(x,\xh) =& \, h S \sum_{m=0}^\infty \frac{1}{m!}D_2^{m} \dg H(x,x)(G(\Phi,x)-x)^m \\
=& \, h S \sum_{m=0}^\infty \frac{1}{(m+1)!}D^m\nabla H(x)(G(\Phi,x)-x)^m \\
& -  h S \sum_{m=1}^\infty \frac{2m}{(m+1)!} D_2^{m-1} Q(x,x) (G(\Phi,x)-x)^m.
\end{align*}
Arguing as in the proof of Lemma III.1.9 in \cite{Hairer06}, we get $h \, S \, \dg H(x,\xh) = G(\theta,x)$, with $\theta(\emptyset) = 0$, $\theta(\ab) = 1$, and
\begin{equation}\label{eq:bone}
\begin{split}
\theta([\tau_1,\ldots,\tau_m]_{\ab}) &= \frac{1}{m+1}\Phi(\tau_1) \cdots \Phi(\tau_m), \\
\theta([\tau_1,\ldots,\tau_m]_{\tb}) &= \frac{-2 m}{m+1}\Phi(\tau_1) \cdots \Phi(\tau_m).
\end{split}
\end{equation}
Then we can write \eqref{eq:dgm} with $\overline{S}(x,\xh,h)$ given by \eqref{eq:GseriesS} as
\begin{align*}
\xh &= x + \sum_{s \in SG} h^{n} \sum_{j} b_{sj} \left( \prod_{k=1}^n S R(\phi_{sjk},x) + (-1)^{\lvert s\rvert_{\ab}-1} \prod_{k=1}^n S R(\phi_{sj(n-k+1)},x) \right) G(\theta,x)\\
&=x + G(\theta,x) + \sum_{s \in SG, \text{ } n>0} \sum b^{s}_j \big( G(\phi_{sj1} \diamond \cdots \diamond \phi_{sjn} \diamond \theta, x) + (-1)^{\lvert s\rvert_{\ab}-1} G(\phi_{sjn} \diamond \cdots \diamond \phi_{sj1} \diamond \theta,x)  \big)\\
&=G(\Phi,x),
\end{align*}
with
\begin{equation}\label{eq:Phiall}
\begin{split}
\Phi =& \hat{e} + \theta + \sum_{s \in SG, \text{ } n>0} \sum_j b_{sj} \left( \phi_{sj1} \diamond \cdots \diamond \phi_{sjn} \diamond \theta + (-1)^{\lvert s\rvert_{\ab}-1} \phi_{sjn} \diamond \cdots \diamond \phi_{sj1} \diamond \theta \right).
\end{split}
\end{equation}
\end{proof}

\begin{theo}
The discrete gradient method \eqref{eq:dgm} with $\overline{S}(x,\xh,h)$ given by \eqref{eq:GseriesS} and $\dg H \in C^{\infty}(\mathbb{R}^d\times \mathbb{R}^d,\mathbb{R}^d)$ is of order $p$ if and only if
\begin{equation}
\Phi(\tau) = \xi(\tau) \quad \text{ for } \lvert\tau\rvert \leq p,
\label{eq:avfocall}
\end{equation}
where $\Phi$ is given by \eqref{eq:Phiall} and the $\xi$ is given by \eqref{eq:gexact}.
\end{theo}
We remark that $\dg H \in C^{\infty}(\mathbb{R}^d\times \mathbb{R}^d,\mathbb{R}^d)$ is a necessary condition for the method to be a G-series method for all $S$ and $H$, but not for its order; $\dg H \in C^{p-1}(\mathbb{R}^d\times \mathbb{R}^d,\mathbb{R}^d)$ is sufficient for the scheme to be of order $p$. The following proposition is presented without its proof, which follows along the lines of the proof of Proposition \ref{th:Phitwo}.

\begin{prop}\label{th:Phitwoall}
The $\Phi$ of \eqref{eq:avfocall} satisfies
\begin{equation}\label{eq:Phitwo2all}
\Phi(\tau) = \hat{e}(\tau) + \theta(\tau) +\sum_{i \text{ s.t. } n\geq 1}\Lambda(\hat{\tau}^i,s^i)
\end{equation}
where $\hat{e}(\emptyset) = 1$ and $\hat{e}(\tau) = 0$ for all $\tau \neq \emptyset$, $\theta$ is given by \eqref{eq:bone}, and
\begin{equation}
\begin{split}
\Lambda(\hat{\tau}^i,s^i) = \theta([\mu_{n+1}^i]_{s_{n+1}^i}) \big( & \sum_j b_{s^ij} \phi_{s^ij1}(\mu_1^i) \cdots \phi_{s^ijn}(\mu_n^i)  \\
&+ (-1)^{\lvert s^i\rvert_{\ab}-1} \sum_j b_{\hat{s}^ij} \phi_{\hat{s}^ijn}(\mu_1^i) \cdots \phi_{\hat{s}^ij1}(\mu_n^i)\big),
\end{split}\label{eq:Lambdaall}
\end{equation}
with $\hat{s}^i$ given by $\hat{s}^i_k = s^i_{n-k+1}$ for $k=1,\ldots,n$, and $\hat{s}^i_{n+1} = s^i_{n+1}$.
\end{prop}

As for the AVF method, one does not need to find the order conditions for every tree; it suffices to find the order condition for each energy-preserving linear combination of the form
\begin{equation}
\omega = [\mu_1]_{s_1} \circ [\mu_2]_{s_2} \circ \cdots [\mu_n]_{s_n} \circ [\emptyset]_{\ab} + (-1)^n \, [\mu_n]_{s_n} \circ [\mu_{n-1}]_{s_{n-1}} \circ \cdots [\mu_1]_{s_1} \circ [\emptyset]_{\ab}.
\label{eq:linenall}
\end{equation}
The above does not give every energy-preserving linear combination of the elementary differentials of G-series; it gives the combinations one gets in the scheme \eqref{eq:dgm} with $\overline{S}(x,\xh,h)$ given by \eqref{eq:GseriesS}. Now, let again $I_l$ and $I_n$ denote the sets of leaf nodes and non-leaf nodes, respectively. If we assume the conditions for order $< p$ to be satisfied, we have an equivalent order condition to \eqref{eq:avfocall} by
\begin{equation}
\sum_{i\in I_l}\Lambda(\hat{\tau}^i,s^i) = \xi(\tau) - \hat{e}(\tau) - \sum_{i\in I_n} \Lambda(\hat{\tau}^i,s^i),
\label{eq:ordercondn2all}
\end{equation}
where we may use the relation
$$
\Lambda(\lbrace\mu^i_1,\ldots,\mu^i_{n},\mu^i_{n+1}\rbrace,s^i) = \hat{\theta}([\mu_{n+1}^i]_{s_{n+1}^i}) \Lambda(\lbrace\mu^i_1,\ldots,\mu^i_{n},\emptyset\rbrace,\bar{s}^i)
$$
to calculate $\Lambda(\hat{\tau}^i)$ for $i \in I_n$. Here $\bar{s}^i$ is $s^i$ with $s^i_{n+1}$ replaced by $\ab$, and $\hat{\theta}(\emptyset) = 0$, $\hat{\theta}(\ab) = 1$, and
\begin{equation}
\begin{split}
\hat{\theta}([\tau_1,\ldots,\tau_m]_{\ab}) &= \frac{1}{m+1}\xi(\tau_1) \cdots \xi(\tau_m),\\
\hat{\theta}([\tau_1,\ldots,\tau_m]_{\tb}) &= \frac{-2 m}{m+1}\xi(\tau_1) \cdots \xi(\tau_m).
\end{split}
\label{eq:bhat}
\end{equation}
Note that $\Lambda(\hat{\tau}^1,s^1) = \hat{\theta}(\tau)$.

\begin{example}
Consider $\tau = \tabtabbb$, which is part of two combinations of the form \eqref{eq:linenall}:
$\omega = \tabtabbb + \ttababbb$ and $\omega = 2 \tabtabbb$.
We calculate
\begin{align*}
\Lambda(\hat{\tau}^1,s^1) &= \Lambda(\lbrace(\ab,\tabb)\rbrace,\tb) = \hat{\theta}(\tabtabbb) = 0,\\
\Lambda(\hat{\tau}^2,s^2) &= \Lambda(\lbrace\tabb,\emptyset\rbrace,\tabb) = \sum_j b_{s^2j}\phi_{s^2j1}(\tabb)+\sum_j b_{\hat{s}^2j}\phi_{\hat{s}^2j1}(\tabb) = 2\sum_j b_{s^2j}\phi_{s^2j1}(\tabb)\\
\Lambda(\hat{\tau}^3,s^3) &= \Lambda(\lbrace\ab,\ab\rbrace,\ttbb) = \hat{\theta}(\tabb)\Lambda(\lbrace\ab,\emptyset\rbrace,\tabb) = \hat{\theta}(\tabb)(-\frac{1}{2}\hat{\theta}(\tababb)) = -1 (-\frac{1}{2}(-\frac{4}{3}))=-\frac{2}{3},\\
\Lambda(\hat{\tau}^4,s^4) &= \Lambda(\lbrace\ab,\emptyset,\emptyset\rbrace,\ttabbb) = \sum_j b_{s^4j}\phi_{s^4j1}(\ab)+\sum_j b_{\hat{s}^4j}\phi_{\hat{s}^4j2}(\ab) = \sum_{j,k} b_{s^4j}\phi_{s^4jk}(\ab).
\end{align*}
Thus \eqref{eq:ordercondn2all} becomes
\begin{equation*}
2 \sum_j b_{s^2 j} \phi_{s^2j1}(\tabb) + \sum_{j,k} b_{s^4j}\phi_{s^4jk}(\ab) = \frac{2}{3}
\end{equation*}
for $\tau = \tabtabbb$. We do similar calculations for $\ttababbb$, and get \eqref{eq:ordercondn2all} for that to be
\begin{equation*}
2 \sum_{j,k} b_{s^4j}\phi_{s^4jk}(\ab) = \frac{4}{3}.
\end{equation*}
Thus we have the order condition
\begin{equation}
\sum_{j,k} b_{s^4j}\phi_{s^4jk}(\ab) = \frac{2}{3}
\label{eq:ordcondex}
\end{equation}
for $\omega = \tabtabbb + \ttababbb$, and
\begin{equation*}
\sum_j b_{s^2 j} \phi_{s^2j1}(\tabb)  = 0
\end{equation*}
for $\omega = 2 \tabtabbb$.
Note that although the tree $\ttababbb$ gives an energy-preserving elementary differential, this by itself is not of the form \eqref{eq:linenall}.
\end{example}

\begin{table}[!ht]
\begin{center}
\begin{tabular}{|c|c|c|c|}
\hline
$\lvert\tau\rvert$ & $\omega$ & $s$ & Order condition\\ \hline
$1$	& $\ab$ & $\ab$ & $\sum_j b_{sj} = \frac{1}{2}$ \\ \hline
$2$	& $\tabb$ & $\tabb$ & $\sum_j b_{sj} = \frac{1}{2}$ \\ \hline
$3$	& $\tababb$ & $\tabb$ & $\sum_{j} b_{sj} \phi_{sj1}(\ab) = \frac{1}{3}$ \\
	& $\ttabbb$ & $\ttabbb$ & $\sum_j b_{sj} = \frac{1}{2}$ \\
	& $\atabbb - \taabbb$ & $\atabbb$ & $\sum_{j} b_{sj} -\sum_{j} b_{\bar{s}j} = 0$  \\
	& $\aaabbb$ & $\aaabbb$ & $\sum_j b_{sj} = -\frac{1}{24}$ \\ \hline
$4$	& $\tabababb$ & $\tabb$ & $\sum_{j} b_{sj} \phi_{sj1}(\ab)^2 = \frac{1}{4}$ \\
	& $\tabtabbb$ & $\tabb$ & $\sum_j b_{s j} \phi_{sj1}(\tabb)  = 0$ \\
	& $\tabaabbb$ & $\tabb$ & $\sum_j b_{s j} \phi_{sj1}(\aabb)  = \frac{1}{6}$ \\
	& $\ttababbb + \tabtabbb$ & $\ttabbb$ & $\sum_{j,k} b_{sj}\phi_{sjk}(\ab) = \frac{2}{3}$ \\
	& $\atababbb - \tabaabbb$ & $\atabbb$ & $\sum_{j} b_{sj} \phi_{sj2}(\ab) - \sum_{j} b_{\bar{s}j} \phi_{\bar{s}j1}(\ab) = 0$  \\
	& $\taababbb - \aabtabbb$ & $\taabbb$ & $\sum_{j} b_{sj} \phi_{sj2}(\ab) - \sum_{j} b_{\bar{s}j} \phi_{\bar{s}j1}(\ab) = 0$  \\
	& $\aaababbb+\aabaabbb$ & $\aaabbb$ & $\sum_{j,k} b_{sj} \phi_{sjk}(\ab)= -\frac{1}{24}$ \\ 
	& $\tttabbbb$ & $\tttabbbb$ & $\sum_{j} b_{sj} = \frac{1}{2}$ \\
	& $\ttaabbbb-\attabbbb$ & $\ttaabbbb$ & $\sum_{j} b_{sj} -\sum_{j} b_{\bar{s}j} = 0$ \\
	& $\taaabbbb+\aatabbbb$ & $\taaabbbb$ & $\sum_{j} b_{sj} - \sum_{j} b_{\bar{s}j} = 0$ \\
	& $\ataabbbb$ & $\ataabbbb$ & $\sum_{j} b_{sj} = 0 $ \\ \hline
\end{tabular}
\end{center}
\caption{Energy-preserving linear combinations of the form \eqref{eq:linenall} and their associated order conditions for the discrete gradient method \eqref{eq:dgm} with $\overline{S}(x,\xh,h)$ given by \eqref{eq:GseriesS}.}
\label{tab:avfocall}
\end{table}

From the order conditions in Table \ref{tab:avfocall}, we can find an $\overline{S}(x,\xh,h)$ so that $\eqref{eq:dgm}$ becomes a fourth order scheme for any $\dg H \in C^{3}(\mathbb{R}^d\times \mathbb{R}^d,\mathbb{R}^d)$. For instance, the stem $s=\ttabbb$ has the related order conditions $\sum_j b_{sj} = \frac{1}{2}$ and $\sum_{j,k} b_{sj}\phi_{sjk}(\ab) = \frac{2}{3}$, which sets the requirements for the term 
$$h^2 \sum_{j} b_{sj} (S Q(x,z_{1j}) S Q(x,z_{2j})+S Q(x,z_{2j}) S Q(x,z_{1j})) S.$$
Choosing $b_{s1} = \frac{1}{2}$ and $z_{11} = z_{21} = x+\frac{2}{3}h f(x)$, we have fulfilled these conditions. Likewise, finding terms that satisfy the other order conditions, we get an approximation of $S$ that ensures fourth order convergence, like the $\overline{S}(x,\xh,h)$ given by \eqref{eq:4thorderDGM}.

\subsection{The general case}\label{sect:gengeng}
Allowing for $S$ to be a function of the solution, we define now the set $TV$ of bi-colored trees whose nodes are either circles of triangles, and whose leaves on the cut tree $\tau^{b}$, defined as the mono-colored tree left when all branches between black and white nodes are cut off, are always circles. Denoting as before black-rooted subtrees by $\tau_i$ and white-rooted subtrees by $\bar{\tau}_i$, the elementary differentials of trees $\tau \in TV$ are given by $F(\ab)(x) = F(\AB)(x) = f(x) = S \nabla H(x)$ and
\begin{equation*}
F(\tau)(x) =
\begin{cases}
S^{(l)} D^m\nabla H(x)(F(\tau_1)(x),\ldots,F(\bar{\tau}_l)(x)) & \text{for } \tau = [\tau_1,\ldots,\tau_m,\bar{\tau}_1,\ldots,\bar{\tau}_l,]_{\abg}, \\
S^{(l)} D_2^{m-1}Q(x,x) (F(\tau_1)(x),\ldots,F(\bar{\tau}_l)(x)) & \text{for } \tau = [\tau_1,\ldots,\tau_m,\bar{\tau}_1,\ldots,\bar{\tau}_l,]_{\tbg}, 
\end{cases}
\end{equation*}
where $\abg$ can be either $\ab$ or $\AB$ and $\tbg$ can be either $\tb$ or $\TB$. Let $TV_{\ab}$ denote the set of trees in $TV$ with black roots, either of the shape $\ab$ or $\tb$.
\begin{definition}
A V-series is a formal series of the form
\begin{equation}\label{eq:vseries}
V(\phi,x) = \phi(\emptyset)x + \sum_{\tau\in TV_{\ab}}\frac{h^{\lvert\tau\rvert}}{\sigma(\tau)}\phi(\tau)F(\tau)(x),
\end{equation}
where $\phi : TV_{\ab} \cup \lbrace\emptyset\rbrace \rightarrow \mathbb{R}$ is an arbitrary mapping, and the symmetry coefficient $\sigma$ is given by \eqref{eq:symcoeff}.
\end{definition}
Proofs of the theorems in this subsection can be obtained similarly to the proofs in sections \ref{sect:avfgen} and \ref{sect:genconst}, and are therefore omitted.

We redefine
\begin{equation*}
R(\phi_{sjk},x) \coloneqq
\begin{cases}
    \nabla^2 H(G(\phi_{sjk},x))      & \quad \text{if } s_k = \ab,\\
    Q(x,V,(\phi_{sjk},x))   & \quad \text{if } s_k = \tb.
\end{cases}
\end{equation*}
and consider now approximations of $S(x)$ that can be written as
\begin{equation}
\begin{split}
\overline{S}(x,\xh,h) =& \sum_{s \in SG} h^{n} \sum_{j} b_{sj} \Bigg( \prod_{k=1}^n S(V(\psi_{sjk},x)) R(\phi_{sjk},x) \cdot S(V(\psi_{sj(n+1)},x)) \\
&+ (-1)^{\lvert s\rvert_{\ab}-1} S(V(\psi_{sj(n+1)},x)) \prod_{k=1}^n R(\phi_{sj(n-k+1)},x) S(V(\psi_{sj(n-k+1)},x)) \Bigg) 
\end{split}
\label{eq:VseriesS}
\end{equation}
whenever $\xh$ is the solution of
\begin{equation*}
\frac{\xh-x}{h} = \overline{S}(x,\xh,h)\dg H(x,\xh),
\end{equation*}
with $\phi_{sjk}(\emptyset) = \psi_{sjk}(\emptyset) = 1$ for every 
$s,j,k$, and with $\sum_j b_{\ab j} = \frac{1}{2}$.

\begin{theo}
The discrete gradient scheme \eqref{eq:dgm} with the approximation of $S(x)$ given by \eqref{eq:VseriesS} and $\dg H \in C^{\infty}(\mathbb{R}^d\times \mathbb{R}^d,\mathbb{R}^d)$ is a V-series method. It can be written $\xh = V(\Phi,x)$, with
\begin{equation}
\begin{split}
\Phi = & \hat{e} + \sum_{s \in SG} \sum_j b_{sj} \, \big( (\psi_{sj1}, \phi_{sj1}) \diamond \cdots \diamond (\psi_{sjn}, \phi_{sjn}) \diamond \hat{\theta}(\psi_{sj(n+1)}) \\
& + (-1)^n \, (\psi_{sj(n+1)}, \phi_{sjn}) \diamond \cdots \diamond (\psi_{sj2}, \phi_{sj1}) \diamond \hat{\theta}(\psi_{sj1})\big).
\end{split}
\label{eq:Phigengen}
\end{equation}
where
\begin{equation}\label{eq:bonegen}
\begin{split}
&\hat{\theta}(a)([\tau_1,\ldots,\tau_m,\bar{\tau}_1,\ldots,\bar{\tau}_l]_{\ab}) = \frac{1}{m+1}\Phi(\tau_1) \cdots \Phi(\tau_m)a(\bar{\tau}_1) \cdots a(\bar{\tau}_l), \\
&\hat{\theta}(a)([\tau_1,\ldots,\tau_m,\bar{\tau}_1,\ldots,\bar{\tau}_l]_{\tb}) = \frac{-2 m}{m+1}\Phi(\tau_1) \cdots \Phi(\tau_m)a(\bar{\tau}_1) \cdots a(\bar{\tau}_l).
\end{split}
\end{equation}
The scheme is of order $p$ if and only if
\begin{equation}
\Phi(\tau) = \xi(\tau) \quad \text{ for } \lvert\tau\rvert \leq p,
\label{eq:gengenoc}
\end{equation}
where
\begin{equation}
\xi(\tau) = 
\begin{cases}
\frac{1}{\gamma(\tau)} & \text{if } \tau \in TP,\\
0 & \text{otherwise}.
\end{cases}
\label{eq:gexactgen}
\end{equation}
\end{theo}

As in Section \ref{sect:hos}, we cut the branches between black and white nodes, regardless of the shape of the nodes, and denote this tree by $\tau^{b}$. Number the nodes and reattach the cut-off parts. For the node $i$ and the corresponding stem $s^i$, there exists a unique set of forests $\hat{\tau}^i = \lbrace(\mu^i_1,\eta^i_1),\ldots,(\mu^i_{n+1},\eta^i_{n+1})\rbrace$ such that 
\begin{equation*}
\tau=[(\mu_1^i,\eta^i_1)]_{s^i_1} \circ \cdots [(\mu^i_{n},\eta^i_{n})]_{s^i_n} \circ [(\mu^i_{n+1},\eta^i_{n+1})]_{s^i_{n+1}}
\end{equation*}

\begin{prop}\label{th:Phitwoallgen}
The $\Phi$ of \eqref{eq:Phigengen} satisfies
\begin{equation}\label{eq:Phitwo2allgen}
\Phi(\tau) = \hat{e}(\tau) + \sum_{i=1}^{\lvert\tau^{b}\rvert} \Lambda(\hat{\tau}^i,s^i)
\end{equation}
where $\hat{e}(\emptyset) = 1$ and $\hat{e}(\tau) = 0$ for all $\tau \neq \emptyset$, and
\begin{equation}
\begin{split}
\Lambda(\hat{\tau}^i,s^i) = & \theta([\mu_{n+1}^i]_{s_{n+1}^i}) \big( \sum_j b_{s^ij} \psi_{s^ij1}(\eta_1^i) \phi_{s^ij1}(\mu_1^i) \cdots \phi_{s^ijn}(\mu_n^i)\psi_{s^ij(n+1)}(\eta_{n+1}^i)  \\
&+ (-1)^{\lvert s^i\rvert_{\ab}-1} \sum_j b_{\hat{s}^ij} \psi_{\hat{s}^ij(n+1)}(\eta_1^i) \phi_{\hat{s}^ijn}(\mu_1^i) \cdots \phi_{\hat{s}^ij1}(\mu_n^i)\psi_{\hat{s}^ij1}(\eta_n^i)\big),
\end{split}\label{eq:Lambdaallgen}
\end{equation}
with $\theta$ given by \eqref{eq:bone} and $\hat{s}^i$ given by $\hat{s}^i_k = s^i_{n-k+1}$ for $k=1,\ldots,n$, and $\hat{s}^i_{n+1} = s^i_{n+1}$.
\end{prop}

The number of trees in $TV$ grows very quickly. However, in our task of finding higher order schemes we may use the lessons of the previous sections, and require that the arguments of $S$, $\nabla^2 H$ and $Q$ in \eqref{eq:VseriesS} are B-series up to order $p-1$. Then we only need to find order conditions for energy-preserving linear combinations of the form 
\begin{equation}\label{eq:lincomball}
\omega = [(\mu_1,\eta_1)]_{s_1} \circ \cdots \circ [(\mu_{n},\eta_{n})]_{s_n} \circ [\eta_{n+1}]_{\ab} + (-1)^n \, [(\mu_n,\eta_{n+1})]_{s_n} \circ \cdots \circ [(\mu_{1},\eta_{2})]_{s_1} \circ [\eta_{1}]_{\ab},
\end{equation}
where $\mu_i$ and $\eta_i$ are forests of trees in $TP_{\ab}$ and $TP_{\AB}$ respectively, for $i=1,\ldots,n+1$. Thus we can disregard any tree with $\TB$ in it.
Furthermore, we may color all nodes of the trees in $\mu_i$ and $\eta_i$ except the roots gray, and let the elementary differentials corresponding to these trees be the same as the elementary differentials of B-trees.

We find the order conditions
\begin{equation*}
\sum_{i\in I_l}\Lambda(\hat{\tau}^i,s^i) = \xi(\tau) - \hat{e}(\tau) - \sum_{i\in I_n} \Lambda(\hat{\tau}^i,s^i),
\end{equation*}
by using the relation
$$
\Lambda(\lbrace(\mu^i_1,\eta_1^i),\ldots,(\mu^i_{n+1},\eta^i_{n+1})\rbrace,s^i) = \hat{\theta}([\mu_{n+1}^i]_{s_{n+1}^i}) \Lambda(\lbrace(\mu^i_1,\eta_1^i),\ldots,(\emptyset,\eta^i_{n+1})\rbrace,\bar{s}^i)
$$
to calculate $\Lambda(\hat{\tau}^i)$ for $i \in I_n$. The $\hat{\theta}$ is given by \eqref{eq:bhat}, and $\bar{s}^i$ is $s^i$ with $s^i_{n+1}$ replaced by $\ab$.

\begin{example}
Consider $\tAgbBabb$, which is part of the energy-preserving linear combination $\taAgbBbb + \tAgbBabb$. We have two black nodes, and calculate
\begin{align*}
\Lambda(\hat{\tau}^1,s^1) &= \Lambda(\lbrace(\ab,\agbb)\rbrace,\tb) = \hat{\theta}(\tabb)\Lambda(\lbrace(\emptyset,\agbb)\rbrace,\ab) = -\xi(\ab)\xi(\aAgbbb) = -\frac{1}{6},\\
\Lambda(\hat{\tau}^2,s^2) &= \Lambda(\lbrace(\emptyset,\agbb),(\emptyset,\emptyset)\rbrace,\tabb) = \sum_j b_{s^2j} \psi_{s^2j1}(\agbb) + \sum_j b_{\bar{s}^2j} \psi_{\bar{s}^2j2}(\agbb) = \sum_j b_{s^2j} (\psi_{s^2j1}+\psi_{s^2j2})(\agbb).
\end{align*}
Hence the order condition associated to this linear combination is
\begin{equation*}
\sum_j b_{s^2j} (\psi_{s^2j1}+\psi_{s^2j2})(\agbb) = \frac{1}{6}.
\end{equation*}
\end{example}

\begin{table}[!ht]
\begin{center}
\begin{tabular}{|c|c|c|c|}
\hline
$\lvert\tau\rvert$ & $\omega$ & $s$ & Order condition\\ \hline
$1$	& $\ab$ & $\ab$ & $\sum_j b_{sj} = \frac{1}{2}$ \\ \hline
$2$	& $\aABb$ & $\ab$ & $2 \sum_j b_{sj} \psi_{sj1}(\ab) = \frac{1}{2}$ \\
	& $\tabb$ & $\tabb$ & $\sum_j b_{sj} = \frac{1}{2}$ \\ \hline
$3$	& $\aABABb$ & $\ab$ & $2 \sum_j b_{sj} \psi_{sj1}(\ab)^2 = \frac{1}{3}$ \\
	& $\aAgbbb$ & $\ab$ & $2 \sum_j b_{sj} \psi_{sj1}(\agbb) = \frac{1}{6}$ \\
	& $\tababb$ & $\tabb$ & $\sum_{j} b_{sj} \phi_{sj1}(\ab) = \frac{1}{3}$ \\
	& $\taABbb + \tABabb$ & $\tabb$ & $\sum_j b_{sj} (\psi_{sj1}+\psi_{sj2})(\ab) = \frac{1}{2}$ \\
	& $\aaABbb - \aABabb$ & $\aabb$ & $\sum_j b_{sj} (\psi_{sj2}-\psi_{sj1})(\ab) = -\frac{1}{12}$ \\
	& $\ttabbb$ & $\ttabbb$ & $\sum_j b_{sj} = \frac{1}{2}$ \\
	& $\atabbb - \taabbb$ & $\atabbb$ & $\sum_{j} b_{sj} - \sum_{j} b_{\bar{s}j} = 0$ \\
	& $\aaabbb$ & $\aaabbb$ & $\sum_j b_{sj} = -\frac{1}{24}$ \\  \hline
$4$	& $\aABABABb$ & $\ab$ & $2 \sum_j b_{sj} \psi_{sj1}(\ab)^3 = \frac{1}{4}$ \\
	& $\aABAgbBb$ & $\ab$ & $2 \sum_j b_{sj} \psi_{sj1}(\ab)\psi_{0j1}(\agbb)=\frac{1}{8}$ \\
	& $\aAgbgBbb$ & $\ab$ & $2 \sum_j b_{sj} \psi_{sj1}(\agbgbb)=\frac{1}{12}$ \\
	& $\aAggbBbb$ & $\ab$ & $2 \sum_j b_{sj}\psi_{sj1}(\aggbbb) = \frac{1}{24}$ \\
	& $\tabababb$ & $\tabb$ & $\sum_{j} b_{sj} \phi_{sj1}(\ab)^2 = \frac{1}{4}$ \\
	& $\tabagbbb$ & $\tabb$ & $\sum_j b_{sj} \phi_{sj1}(\agbb)  = \frac{1}{6}$ \\
	& $\tAbAaBBb$ & $\tabb$ & $\sum_j b_{sj} \psi_{sj1}(\ab)\psi_{sj2}(\ab) = \frac{1}{8}$ \\
	& $\tAbaabBb + \tABababb$ & $\tabb$ & $\sum_j b_{sj} \phi_{sj1}(\ab)(\psi_{sj1}+\psi_{sj2})(\ab) = \frac{1}{3}$ \\ 
	& $\taABABbb + \tABABabb$ & $\tabb$ & $\sum_j b_{sj} (\psi_{sj1}(\ab)^2+\psi_{sj2}(\ab)^2) = \frac{1}{3} $\\ 
	& $\taAgbBbb + \tAgbBabb$ & $\tabb$ & $\sum_j b_{sj} (\psi_{sj1}+\psi_{sj2})(\agbb) = \frac{1}{6} $ \\
	& $\aAbaabBb - \aABababb$ & $\aabb$ & $\sum_j b_{sj}\phi_{sj1}(\ab)(\psi_{sj2}-\psi_{sj1})(\ab) = -\frac{1}{24}$ \\
	& $\aaABABbb - \aABABabb$ & $\aabb$ & $\sum_j b_{sj} (\psi_{sj2}(\ab)^2-\psi_{sj1}(\ab)^2) = -\frac{1}{12}$\\
	& $\aaAgbBbb - \aAgbBabb$ & $\aabb$ & $\sum_j b_{sj}(\psi_{sj2}-\psi_{sj1})(\agbb) = -\frac{1}{24}$ \\
	& $\ttababbb + \tabtabbb$ & $\ttabbb$ & $\sum_{j,k} b_{sj}\phi_{sjk}(\ab) = \frac{2}{3}$ \\
	& $\ttabABbb$ & $\ttabbb$ & $\sum_j b_{sj} \psi_{sj2}(\ab) = \frac{2}{3} $ \\
	& $\ttaAbbBb + \ttBAabbb$ & $\ttabbb$ & $ \sum_j b_{sj} (\psi_{sj1}+\psi_{sj3})(\ab)=\frac{1}{2}$ \\
	& $\taababbb - \aabtabbb$ & $\taabbb$ & $\sum_{j} b_{sj} \phi_{sj2}(\ab) - \sum_{j} b_{\bar{s}j} \phi_{\bar{s}j1}(\ab) = 0$  \\ \hline
\end{tabular}
\end{center}
\caption{Energy-preserving linear combinations of the form \eqref{eq:lincomball}. Continued in Table \ref{tab:last2}.}
\label{tab:last1}
\end{table}

\begin{table}[!ht]
\begin{center}
\begin{tabular}{|c|c|c|c|}
\hline
$\lvert\tau\rvert$ & $\omega$ & $s$ & Order condition\\ \hline
$4$	& $\taaAbbBb - \atBAabbb$ & $\taabbb$ & $\sum_j b_{sj} (\psi_{sj1}-\psi_{sj3})(\ab) = \frac{1}{12}$ \\ 
	& $\atababbb - \tabaabbb$ & $\atabbb$ & $\sum_{j} b_{sj} \phi_{sj2}(\ab) - \sum_{j} b_{\bar{s}j} \phi_{\bar{s}j1}(\ab) = 0$  \\
	& $\ataAbbBb - \taBAabbb$ & $\atabbb$ & $\sum_j b_{sj}\psi_{sj3}(\ab)-\sum_j b_{\bar{s}j}\psi_{\bar{s}j1}(\ab) = 0$ \\	
	& $\aaababbb+\aabaabbb$ & $\aaabbb$ & $\sum_{j,k} b_{sj} \phi_{sjk}(\ab)= -\frac{1}{24}$ \\ 
	& $\aaabABbb$ & $\aaabbb$ & $2 \sum_j b_{2j} \psi_{2j1}(\ab) = -\frac{1}{24}$ \\
	& $\aaaAbbBb + \aaBAabbb$ & $\aaabbb$ & $\sum_j b_{sj}(\psi_{sj1}+\psi_{sj3})(\ab) = -\frac{1}{24}$ \\
	& $\tttabbbb$ & $\tttabbbb$ & $\sum_{j} b_{sj} = \frac{1}{2}$ \\
	& $\ttaabbbb-\attabbbb$ & $\ttaabbbb$ & $\sum_{j} b_{sj} - \sum_{j} b_{\bar{s}j} = 0$ \\
	& $\taaabbbb+\aatabbbb$ & $\taaabbbb$ & $\sum_{j} b_{sj} = -\frac{1}{12} $ \\
	& $\ataabbbb$ & $\ataabbbb$ & $\sum_{j} b_{sj} = 0 $ \\  \hline
\end{tabular}
\end{center}
\caption{Energy-preserving linear combinations of the form \eqref{eq:lincomball}. Continuing from Table \ref{tab:last1}.}
\label{tab:last2}
\end{table}

We consider the order conditions for trees with $\lvert\tau\rvert \leq 3$ displayed in Table \ref{tab:last1}, and find that 
\begin{equation}
\begin{split}
\overline{S}(x,\cdot,h) =& \frac{1}{4} S(x) + \frac{3}{4}S(z_3) \\
& + h S(z_2) Q(x,z_3) S(z_2) + \frac{1}{4} h \, \big( S(z_1)\nabla^2 H(x)S(x) - S(x)\nabla^2 H(x)S(z_1)\big) \\
& + h^2 S(x)Q(x,x)S(x)Q(x,x)S(x) - \frac{1}{12} h^2 \, S(x)\nabla^2 H(x)S(x)\nabla^2 H(x)S(x),
\end{split}
\label{eq:3thSgen}
\end{equation}
where
\begin{equation*}
z_1 = x + \frac{1}{3} h f(x), \qquad z_2 = x + \frac{1}{2} h f(x), \qquad z_3 = x + \frac{2}{3} h f(z_1),
\end{equation*}
guarantees third order convergence of the scheme \eqref{eq:dgm} if $\dg H(x) \in C^2(\mathbb{R}^d\times \mathbb{R}^d,\mathbb{R}^d)$. An approximation of $S(x)$ satisfying all the order conditions in tables \ref{tab:last1} and \ref{tab:last2} is given by
\begin{equation}
\begin{split}
\overline{S}(x,\cdot,h) =& \frac{1}{2} (S(z_{11}+z_{12})+S(z_{11}-z_{12})) + \frac{1}{12} h \, \big( S(z_6)\nabla^2 H(z_2)S(x) - S(x)\nabla^2 H(z_2)S(z_6)\big) \\
& + \frac{3}{7} h \big( S(z_3)Q(x,z_5)S(z_4) + S(z_4)Q(x,z_5) S(x,z_3)\big) \\
& +\frac{8}{105} h S(x)Q(x,z_7)S(x) + \frac{1}{15} h S(x)Q(x,x)S(x)\\
& + h^2 \, S(z_2)Q(x,z_5)S(z_8)Q(x,z_5)S(z_2)\\
& - \frac{1}{12} h^2 \, S(z_2)\nabla^2 H(z_2)S(z_2)\nabla^2 H(z_2)S(z_2)\\
& + \frac{1}{6} h^2 (S(z_2)-S(x))\nabla^2 H(x) S(x) Q(x,x) S(x) \\
&-  \frac{1}{6} h^2 S(x) Q(x,x) S(x) \nabla^2 H(x) (S(z_2)-S(x))\\
&+ h^3 S(x) Q(x,x) S(x) Q(x,x) S(x) Q(x,x) \\
&- \frac{1}{12} h^3 S(x) \nabla^2 H(x) S(x) \nabla^2 H(x) S(x) Q(x,x) S(x)\\
& - \frac{1}{12} h^3 S(x) Q(x,x) S(x) \nabla^2 H(x) S(x) \nabla^2 H(x)S(x),
\end{split}
\label{eq:4thSgen}
\end{equation}
with
\begin{align*}
z_1 &= x + \frac{1}{3} h f(x), & z_5 = x + \frac{2}{3} h f(z_2), & \hspace{30pt} z_{9} = x + h f(z_6),\\
z_2 &= x + \frac{1}{2} h f(x), & z_6 = x + h f(z_2), & \hspace{30pt} z_{10} = x + h f(z_9),\\
z_3 &= x + \frac{7-\sqrt{7}}{12} h f(z_1), & z_7 = x + \frac{5}{4} h f(z_2), & \hspace{30pt} z_{11} = \frac{1}{3}(x+z_2+z_6) + \frac{1}{12} (-z_9+z_{10}),\\
z_4 &= x + \frac{7+\sqrt{7}}{12} h f(z_1), & z_{8} = x + \frac{4}{3} h f(z_2), & \hspace{30pt} z_{12} = \frac{\sqrt{3}}{36}(7 x-2 z_2-4z_6+z_9-2z_{10}),
\end{align*}
and hence a discrete gradient scheme with this $\overline{S}(x,\cdot,h)$ and any $\dg H \in C^3(\mathbb{R}^d\times \mathbb{R}^d,\mathbb{R}^d)$ will be of fourth order.

One advantage of choosing the AVF discrete gradient is that the resulting scheme generally requires fewer computations at each time step. This is clearly evident in the above example: if $\dg = \dg_{\text{AVF}}$, then \eqref{eq:3thSgen} collapses to \eqref{eq:3thSavf}, and \eqref{eq:4thSgen} collapses to \eqref{eq:4thSavf}. However, if the AVF discrete gradient is difficult to calculate, there can also be  much to gain in computational cost by choosing a symmetric discrete gradient, like the symmetrized Itoh--Abe discrete gradient \eqref{eq:siadg} or the Furihata discrete gradient \eqref{eq:furihatadg}. Then one can ignore the order condition for any combination \eqref{eq:lincomball} for which $s_j = \tb$ and $\mu_j = \emptyset$ for some $j \in[1,n]$, since this corresponds to elementary differentials involving $Q(x,x)$, which we recall is zero when the discrete gradient is symmetric. If we consider the conditions for fourth order presented in tables \ref{tab:last1} and \ref{tab:last2}, this eliminates 17 of the 22 conditions for trees with $\tb$ in the stem.
By considering the remaining order conditions we get that, if $\dg H \in C^3(\mathbb{R}^d\times \mathbb{R}^d,\mathbb{R}^d)$ and $\dg H(x,y) = \dg H(y,x)$, the discrete gradient scheme \eqref{eq:dgm} is of fourth order if
\begin{equation}
\begin{split}
\overline{S}(x,\cdot,h) =& \frac{1}{2} (S(z_5+z_6)+S(z_5-z_6)) + \frac{1}{12} h \, \big( S(z_2)\nabla^2 H(z_1)S(x) - S(x)\nabla^2 H(z_1)S(z_2)\big) \\
& + \frac{8}{9} h S(z_1) Q(x,z_7) S(z_1) - \frac{1}{12} h^2 \, S(z_1)\nabla^2 H(z_1)S(z_1)\nabla^2 H(z_1)S(z_1),
\end{split}
\label{eq:4thSsym}
\end{equation}
with
\begin{align*}
z_1 &= x + \frac{1}{2} h f(x), & z_3 = x + h f(z_2), & \hspace{30pt} z_5 = \frac{1}{3}(x+z_1+z_2) + \frac{1}{12} (-z_3+z_4),\\
z_2 &= x + h f(z_1), & z_4 = x + h f(z_3), & \hspace{30pt} z_6 = \frac{\sqrt{3}}{36}(7 x-2 z_1-4z_2+z_3-2z_4),\\
z_7 &= x + \frac{3}{4}h f(z_1).
\end{align*}
If $S$ is constant, \eqref{eq:4thSsym} simplifies to
\begin{equation}\label{eq:4thsymS}
\overline{S}(x,\cdot,h) = S + \frac{8}{9} h S Q(x,z_7) S - \frac{1}{12} h^2 \, S\nabla^2 H(z_1)S\nabla^2 H(z_1)S.
\end{equation}

\section{Numerical experiments}

For all the examples considered in the following, the resulting discrete gradient schemes are nonlinearly implicit systems, which we here solve using Newton's method.  
Note that whenever $\overline{S}(x,\xh,h)$ is independent of $\xh$, the Jacobian of
$$
F(\xh) = \xh - x - h \overline{S}(x,\xh,h)\dg H(x,\xh)
$$
is given by
\begin{equation}\label{eq:jacobian}
J(\xh) = I - h \overline{S}(x,\xh,h) D_2 \dg H(x,\xh),
\end{equation}
where $D_2 \dg H(x,\xh)$ may also be used for the calculation of $Q(x,\xh)$. The extra computational cost of using a higher-order scheme with an explicit $\overline{S}$ thus only lies in the computation of this $\overline{S}$ once each time step. A scheme with an implicitly defined $\overline{S}(x,\xh,h)$ can give a much more complicated Jacobian, but a quasi-Newton method using the approximate Jacobian given by \eqref{eq:jacobian} may still be very efficient.

\subsection{Hénon--Heiles system}
The Hénon--Heiles system can be written in the form \eqref{eq:constSform} with
\begin{equation}\label{eq:henonheiles}
S =
\begin{pmatrix}
0 & I\\
-I & 0
\end{pmatrix},
\qquad  H(q,p) = \frac{1}{2}(q_1^2+q_2^2+p_1^2+p_2^2) + q_1^2 q_2 - \frac{1}{3}q_2^3,
\end{equation}
where $I$ is the $2 \times 2$ identity matrix. We use here the same initial conditions used in \cite{Quispel08}: $q_1=\frac{1}{10}$, $q_2 = -\frac{1}{2}$, $p_1 = p_2 = 0$. The order of some of the energy-preserving methods proposed in this paper are confirmed by the left plot in Figure \ref{fig:HH}. We compare the performance of the fourth order discrete gradient methods obtained by using the $\overline{S}$ given by \eqref{eq:4thorderDGM} coupled with three different discrete gradients: the Itoh--Abe discrete gradient \eqref{eq:iadg}, the Furihata discrete gradient \eqref{eq:furihatadg}, and the AVF discrete gradient \eqref{eq:avfdg}. The symmetrized Itoh--Abe discrete gradient \eqref{eq:siadg} is for this $H$ identical to the Furihata discrete gradient. The AVF and Furihata discrete gradient methods perform in this case very similarly, and thus the error from the Furihata discrete gradient method is excluded from the right plot in Figure \ref{fig:HH}. We observe that, although it initially performs on par with the AVF method, the Itoh--Abe discrete gradient method gives a lower global error than the other fourth order methods as time goes on. Note however that this method requires the most computations at every time step.

\begin{figure}
        \centering
        \subfloat{
                \centering
                \includegraphics[width=0.46\textwidth]{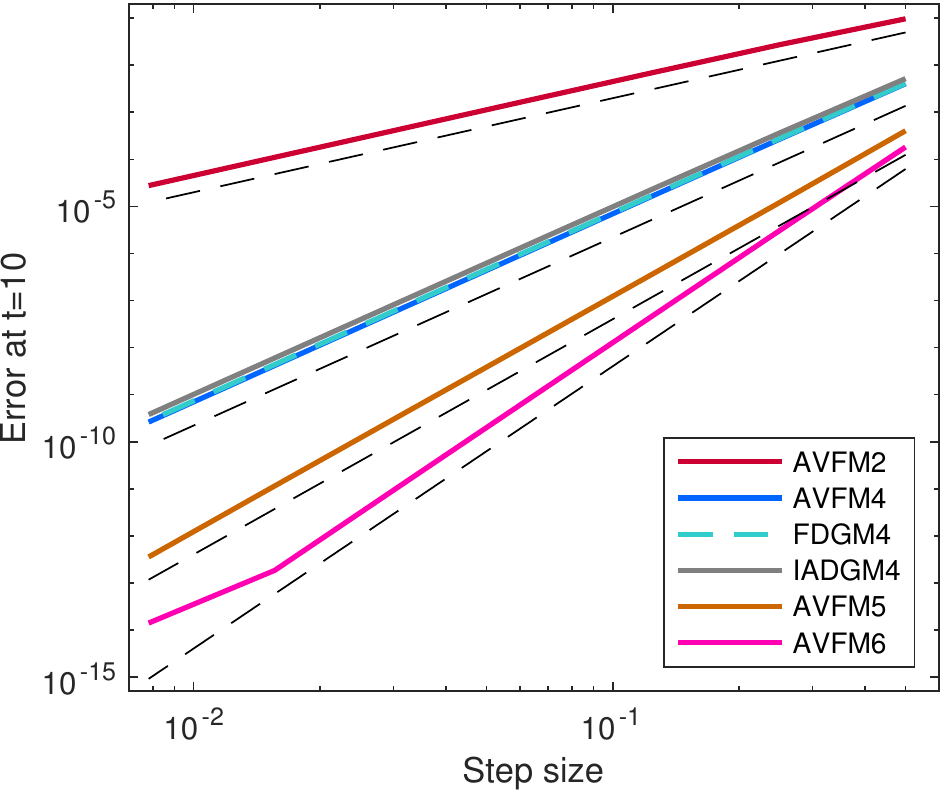}
        }\quad
        \subfloat{
                \centering
                \includegraphics[width=0.46\textwidth]{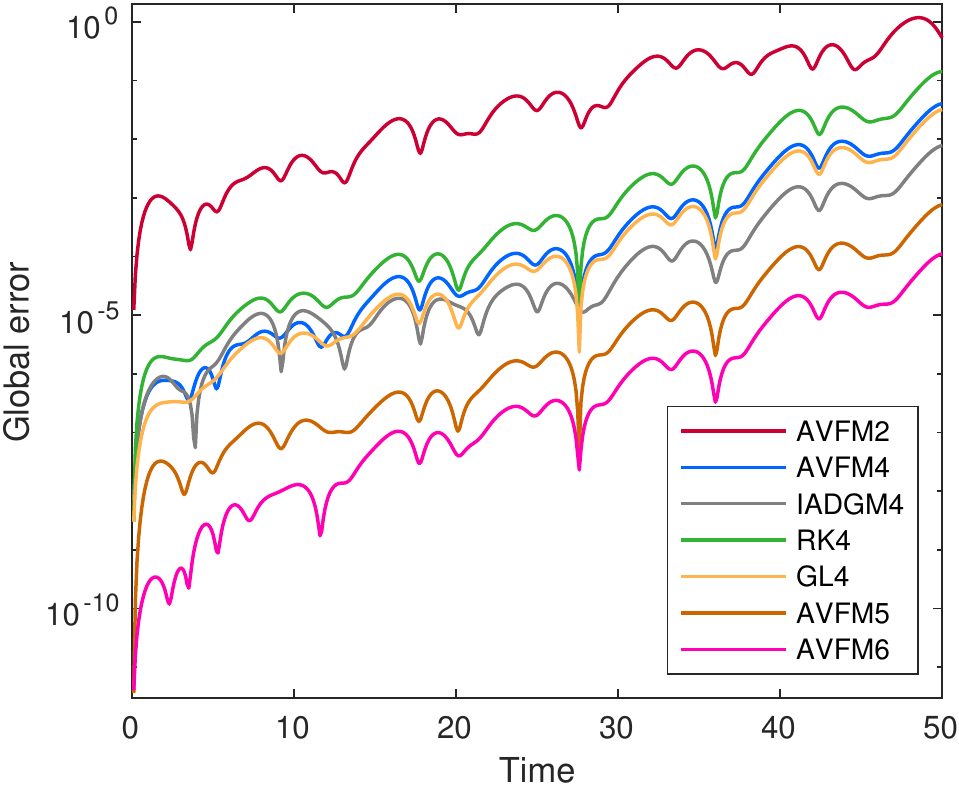}
        }
        \caption{Error plots for the Hénon--Heiles system \eqref{eq:henonheiles} solved by various discrete gradient methods: AVFM2 is the standard AVF method \eqref{eq:avfm2}; AVFM4 is the AVF discrete gradient method with $\overline{S}$ given by \eqref{eq:4thordavf}; FDGM4 is the Furihata discrete gradient method with $\overline{S}$ given by \eqref{eq:4thsymS}; IADGM4 is the Itoh--Abe discrete gradient method with $\overline{S}$ given by \eqref{eq:4thorderDGM}; AVFM5 is the scheme \eqref{eq:5thavfm}; AVFM6 is \eqref{eq:6thavfmexp}. RK4 is the classic Runge--Kutta method and GL4 is the fourth order Gauss--Legendre method, included for comparison. The black dashed lines in the order plot are reference lines of order two, four, five and six. The step size in the right plot is $h=0.1$.}
\label{fig:HH}
\end{figure}

\subsection{Lotka--Volterra}
The methods should also be tested on a skew-gradient system with non-constant $S$. We choose the Lotka--Volterra system also used for numerical experiments in \cite{Cohen11}. It is given by
\begin{equation}\label{eq:lotkavolterra}
S = \frac{1}{2}
\begin{pmatrix}
0 & -x_1 x_2 & x_1 x_3\\
x_1 x_2 & 0 & -2 x_2 x_3 \\
-x_1 x_3 & 2 x_2 x_3 & 0
\end{pmatrix},
\qquad H(x) = 2 x_1 + x_2 + 2x_3 + \ln(x_2) - 2\ln(x_3),
\end{equation}
and initial conditions $x_1 = 1$, $x_2=\frac{19}{10}$, $x_3 = \frac{1}{2}$. For this $H$, the Itoh--Abe, Furihata and AVF discrete gradients are all equivalent. We consider fourth order discrete gradient methods where $\nabla{S}$ is given either dependent on or independent of $\xh$; that is, \eqref{eq:4thSimp} or \eqref{eq:4thSavf}. The implicitly given \eqref{eq:4thSavf} yields a significantly lower error in the solution of the corresponding discrete gradient method, as can be witnessed from the left plot in Figure \ref{fig:LV}. In contrast to what we observed for the canonical Hamiltonian system studied above, none of the discrete gradient methods give a global error lower than that of the fourth order Gauss--Legendre method.

\begin{figure}
        \centering
        \includegraphics[width=0.46\textwidth]{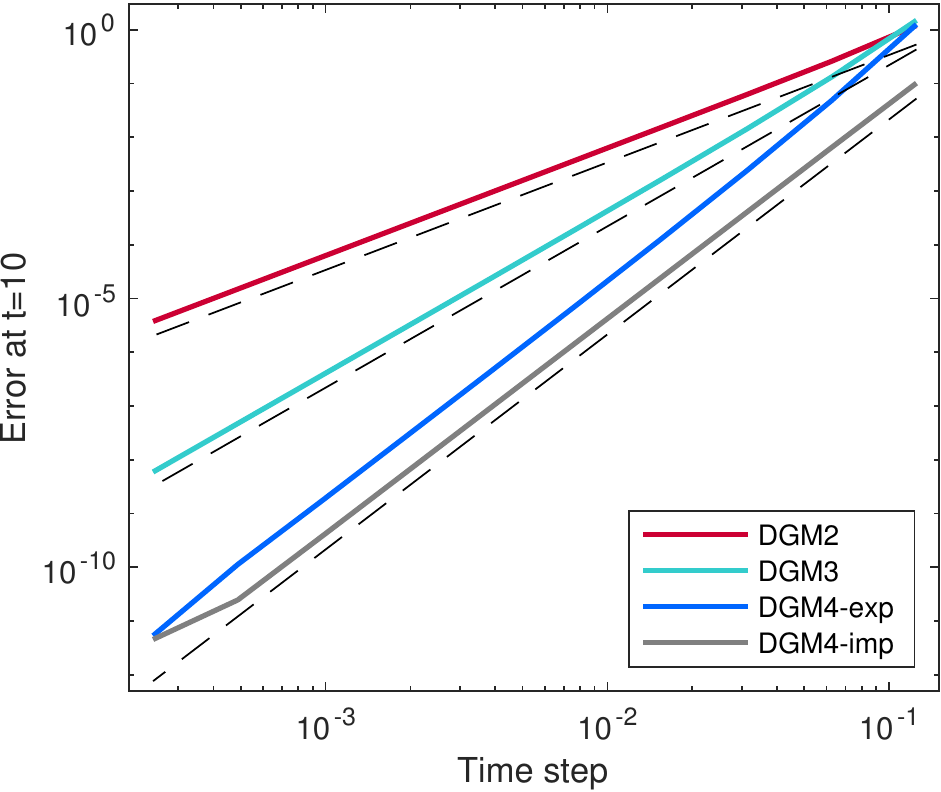}
                \caption{Order or discrete gradient methods applied to the Lotka--Volterra system \eqref{eq:lotkavolterra}, with different $\overline{S}$: $\overline{S}(x,\xh) = S(\frac{x+\xh}{2})$ for DGM2, \eqref{eq:3thSavf} for DGM3, \eqref{eq:4thSavf} for DGM4-exp, \eqref{eq:4thSimp} for DGM4-imp. The dashed lines are reference lines of order two, three and four.}
\label{fig:LVorder}
\end{figure}
\begin{figure}
        \centering
        \subfloat{
                \centering
                \includegraphics[width=0.46\textwidth]{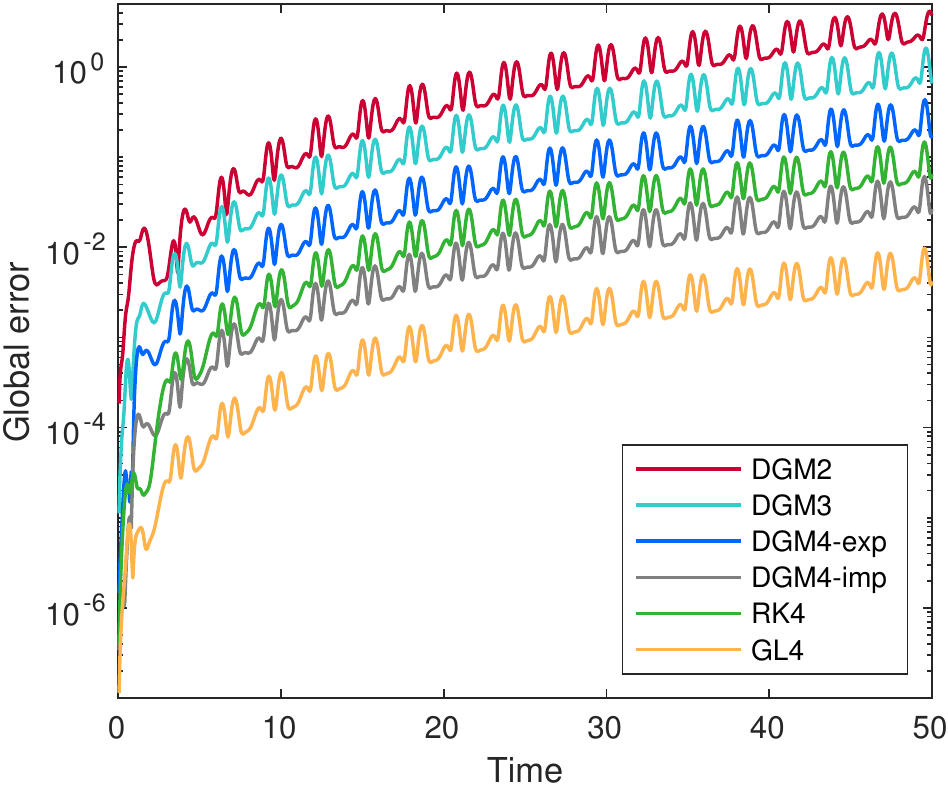}
        }\quad
        \subfloat{
                \centering
                \includegraphics[width=0.46\textwidth]{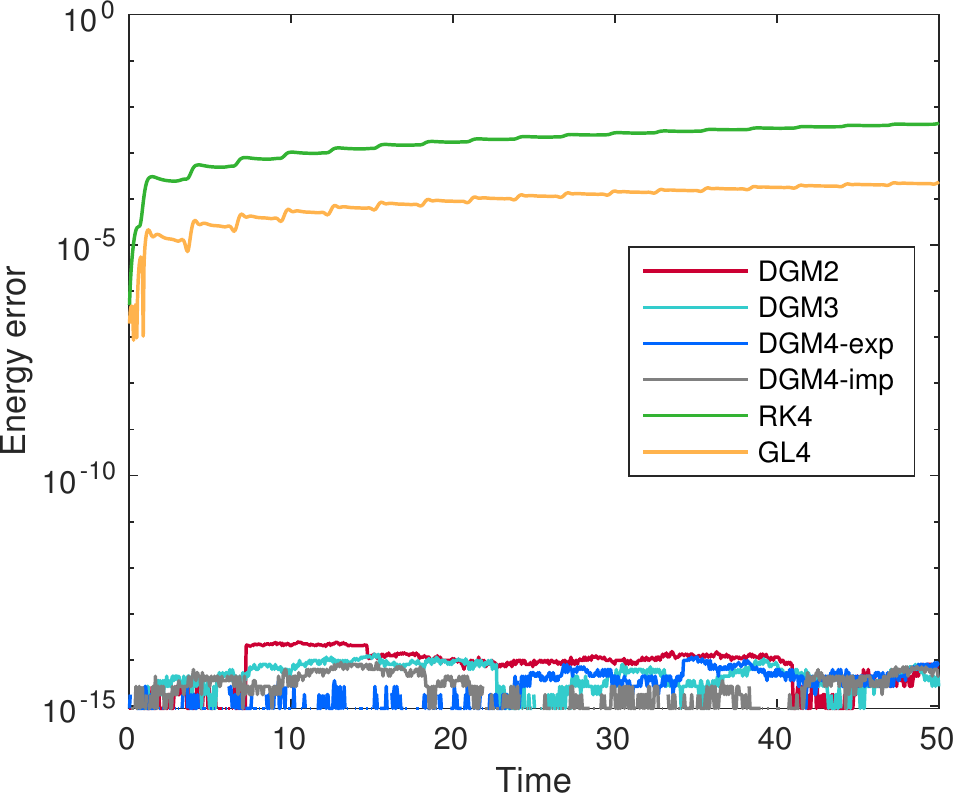}
        }
        \caption{Error in the solution and in the energy for discrete gradient methods with $\overline{S}$ given by $\overline{S}(x,\xh) = S(\frac{x+\xh}{2})$ for DGM2, \eqref{eq:3thSavf} for DGM3, \eqref{eq:4thSavf} for DGM4-exp and \eqref{eq:4thSimp} for DGM4-imp, applied to the Lotka--Volterra system, with step size $h=0.05$. For comparison, errors from using the standard fourth order Runge--Kutta (RK4) and Gauss--Legendre (GL4) methods are also included.}
\label{fig:LV}
\end{figure}

\subsection{Hamiltonian neural networks}

We present here results on the pendulum problem also considered in \cite{Greydanus19, Matsubara19}. Training data is first generated by adding noise to solutions of the system
\eqref{eq:constSform} with
\begin{equation*}
S =
\begin{pmatrix}
0 & 1\\
-1 & 0
\end{pmatrix},
\qquad  H(q,p) = 2 m g l (1- \cos{q}) + \frac{l^2}{2 m} p^2,
\end{equation*}
with $l=m=1$ and $g=3$, for various times $t \in \left[0,20\right]$ and $50$ randomly sampled initial coordinates. Then a Hamiltonian neural network (HNN) is used to learn the Hamiltonian. If the automatic discrete differentiation algorithm is used for the training, the corresponding discrete gradient is learned simultaneously.
To compute $Q(x,y)$, we require that the network also learns the Jacobian of the discrete gradient with respect to the second argument. This necessitates a modification of the network developed by Matsubara et al \cite{Matsubara19}, which we have done using the autograd class of PyTorch. The resulting Jacobian is in any case very useful for integrating the system by any discrete gradient method, since we then can employ a root-finding algorithm that use \eqref{eq:jacobian}, e.g. Newton's method.

We have tested second, third and fourth order discrete gradient methods both for training the network and for integrating the obtained dynamical system. For the fourth order scheme, we have used the $\overline{S}$ given by \eqref{eq:4thsymS}  and for the third order scheme we have used
\begin{equation}\label{eq:3rdsymS}
\overline{S}(x,\cdot,h) = S + h S Q(x,z) S - \frac{1}{12} h^2 \, S\nabla^2 H(x)S\nabla^2 H(x)S,
\end{equation}
where $z = x + \frac{2}{3} h S \nabla H(x)$. Note that $\nabla H$ and $\nabla^2 H$ can be obtained from the network through the relations $\nabla H(x) = \dg H(x,x)$ and $\nabla^2 H(x) = 2 D_2\dg H(x,x)$.

As is also noted in \cite{Matsubara19}, energy preservation seems to be more important than high order for the method used to train the system. In Figure \ref{fig:hnn_train} we see an example of this: the second order discrete gradient method outperforms the fourth order Runge--Kutta method. Whether or not higher-order discrete gradient methods give an improved performance that justifies the extra computational cost is not clear on this example. A thorough investigation of this for more complex problems is a matter for future studies.
\begin{figure}
        \centering
        \includegraphics[width=0.42\textwidth]{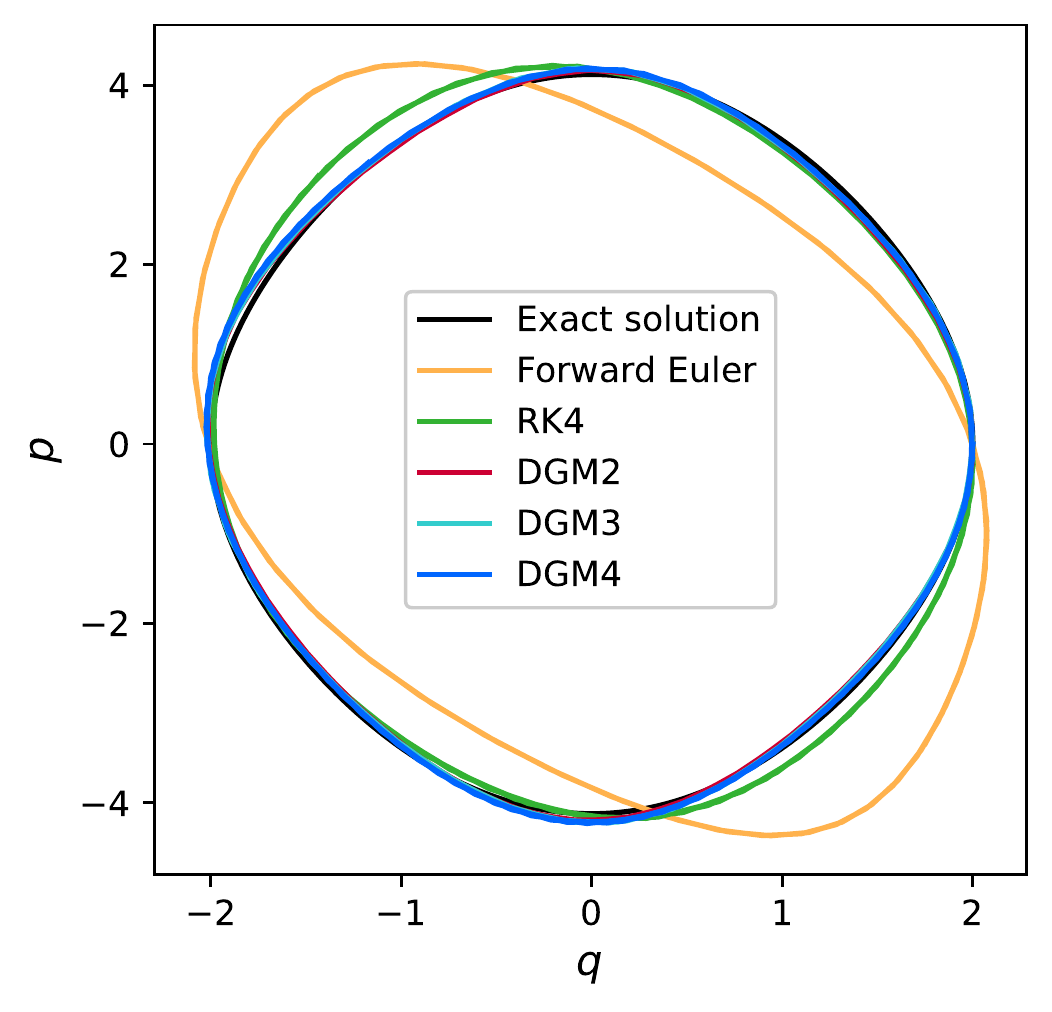}
                \caption{The trajectories that goes through $(q,p) = (2,0)$ for the dynamical systems learned by an HNN with the Euler method, the classic Runge--Kutta method and discrete gradient methods with $\overline{S}$ given by $S$ (DGM2), \eqref{eq:3rdsymS} (DGM3) and \eqref{eq:4thsymS} (DGM4). The plots for DGM2 and DGM3 are almost entirely hidden behind the plot for DGM4. The step size used in the training is $h=0.5$.}
\label{fig:hnn_train}
\end{figure}
For energy-preserving integration of the learned system, much can be gained in accuracy by using a higher-order discrete gradient method, at the expense of not much additional computational cost.  The order plots in Figure \ref{fig:hnn} are obtained by comparing to a solution found by the fourth order discrete gradient method with a smaller step size.  Python code for the neural networks used to produce the plots in Figure \ref{fig:hnn_train} and Figure \ref{fig:hnn}, and additional results on a mass-spring system, is available at \url{https://github.com/solveeid/dgnet4}.
\begin{figure}
        \centering
        \subfloat{
                \centering
                \includegraphics[width=0.46\textwidth]{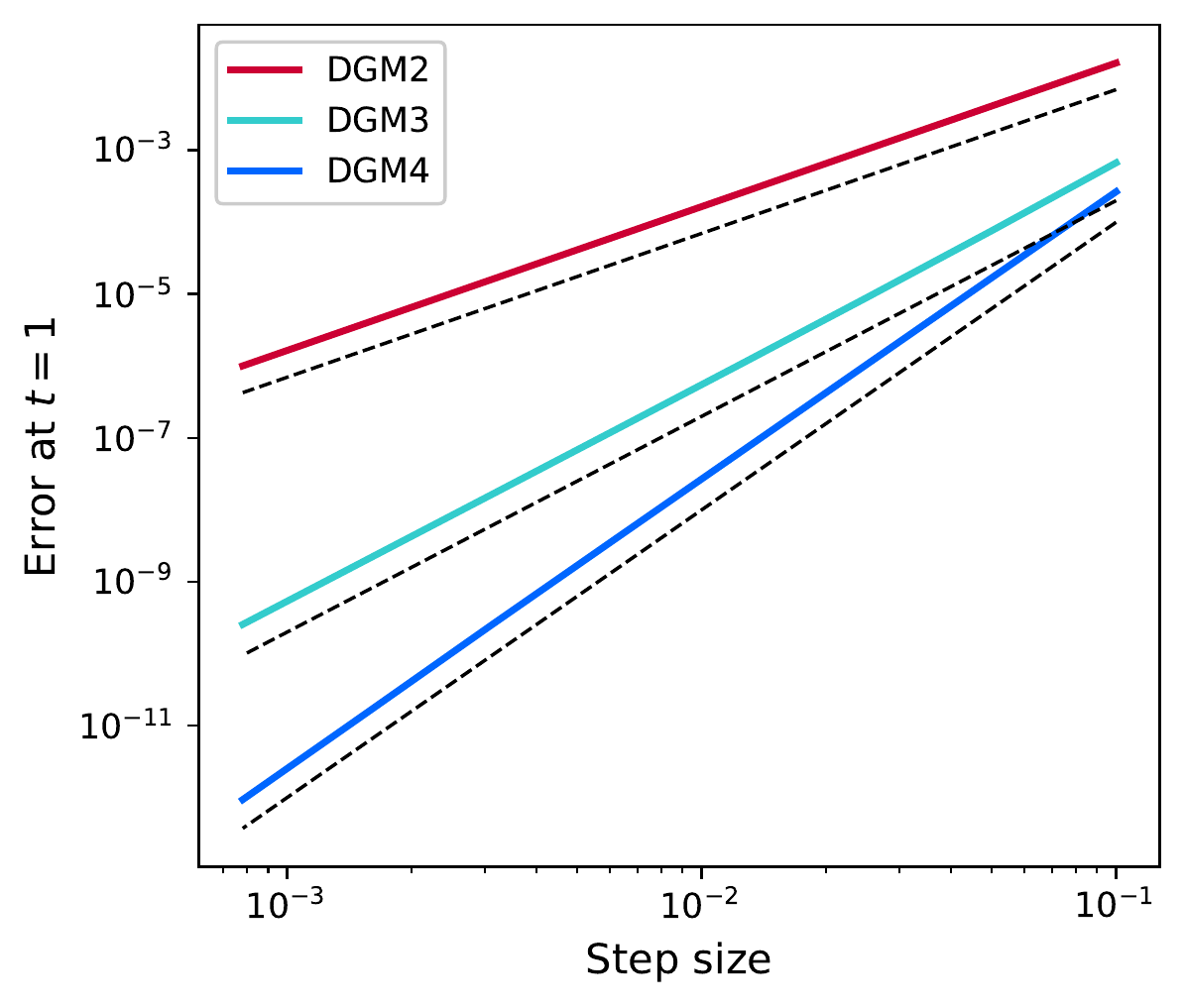}
        }\quad
        \subfloat{
                \centering
                \includegraphics[width=0.46\textwidth]{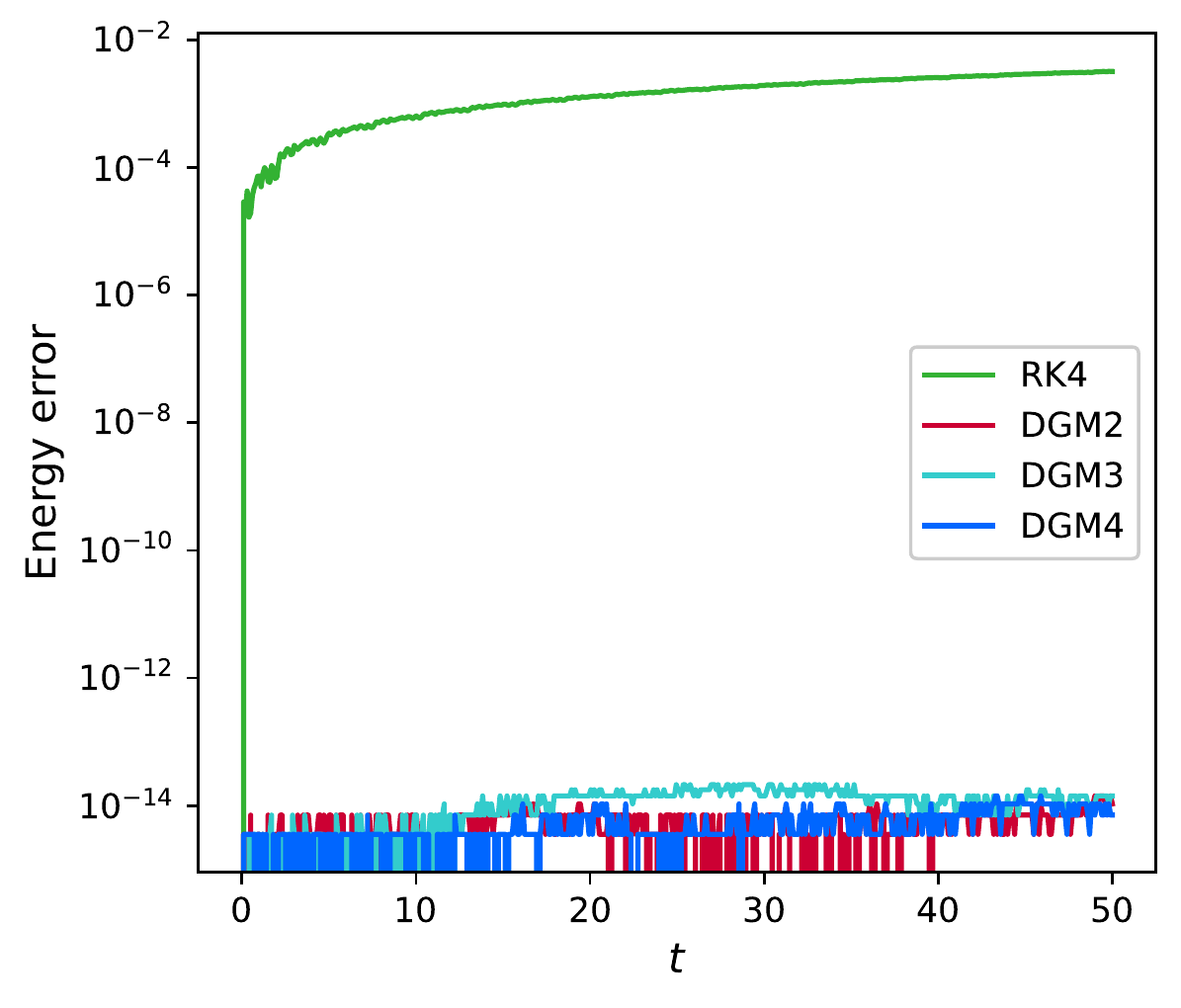}
        }
        \caption{Comparison of different integrators for the pendulum system found using an HNN with the second order discrete gradient method.  \textit{Left:} Order plots for discrete gradient methods with $\overline{S}$ given by $S$ (DGM2), \eqref{eq:3rdsymS} (DGM3) and \eqref{eq:4thsymS} (DGM4). The dashed lines are reference lines of order two, three and four. \textit{Right:} Energy errors for the same methods, and the fourth-order Runge--Kutta method. Step size $h=0.1$.}
\label{fig:hnn}
\end{figure}

\section{Conclusions and future work}

The main purpose of this paper has been to develop order theory for discrete gradient methods. This was achieved through the introduction of the function $Q$, Lemma \ref{th:Qlem} and a generalization of B- and P-series results.  Propositions \ref{th:Phitwo},  \ref{th:Phitwogen}, \ref{th:Phitwoall} and \ref{th:Phitwoallgen} present results which simplifies the derivation of the conditions from the general theory.  The techniques introduced in this paper for building on B- and P-series methods can possibly be used on more methods than the discrete gradient methods.  Future research may utilize this.

We have proposed some higher order schemes satisfying the derived order conditions. The development of specific schemes has however not been the main focus of this paper; analysis to find more optimal schemes is something we leave for the future. After such an analysis is performed, the methods could be tested on more advanced problems than those considered above, e.g.\ for the temporal discretization of Hamiltonian partial differential equations, and their performance as measured by accuracy relative to computational cost could be compared to existing methods.

The use of neural networks to train dynamical systems with preservation properties is an emerging field of study where new developments are coming with a high frequency. To our knowledge, the use of higher-order one-step energy-preserving integrators in this setting has not previously been studied, and may not even have been known to be possible until the novel results of this paper. The potential utility of the methods remain largely unexplored and will be considered in future work. This will include using methods of order higher than four and analysis of the methods compared to and in connection with other recent developments, both in the training of the network and for the integration of the resulting system.  It would also be interesting to investigate how the methods perform on more advanced examples, including Hamiltonian PDEs, and for Lagrangian neural networks \cite{Cranmer20, Aoshima21}.

The order theory presented in this paper can also be developed further in several different directions. The schemes given with $\overline{S}$ independent of $\xh$ are linearly implicit when $H$ is quadratic; if the order theory is extended to the polarized discrete gradient methods of \cite{Matsuo01dissipative,Dahlby11,Eidnes21}, we could get higher order linearly implicit multi-step schemes for systems with polynomial first integrals of any degree. Another avenue could be to consider order conditions for the discrete Riemannian gradient methods presented in \cite{Celledoni20}. Then the results in Section \ref{sect:gengen} are especially interesting, since the integral in the AVF discrete Riemannian gradient can be challenging to compute analytically. Lastly, building on \cite{Berland05}, order conditions for the exponential AVF method is given in \cite{Mei21}. This could be extended to exponential integrators using any discrete gradient by the theory presented in this paper.

\subsection*{Acknowledgments}
The author is grateful to Brynjulf Owren for the invaluable suggestions he has provided during the writing of this paper.
The research was supported by the Research Council of Norway, through the projects SPIRIT (no.\ 231632) and BigDataMine (no.\ 309691). The author would like to thank the Isaac Newton Institute for Mathematical Sciences, Cambridge, for support and hospitality during the programme \textit{Geometry, compatibility and structure preservation in computational differential equations}, where work on this paper was undertaken. This work was supported by EPSRC grant no.\ EP/R014604/1.

\bibliography{ordercond_bibl}

\begin{thebibliography}{10}

\bibitem{McLachlan99}
R.~I. McLachlan, G.~R.~W. Quispel, and N.~Robidoux, ``Geometric integration
  using discrete gradients,'' {\em R. Soc. Lond. Philos. Trans. Ser. A Math.
  Phys. Eng. Sci.}, vol.~357, no.~1754, pp.~1021--1045, 1999.

\bibitem{Chorin1978}
A.~J. Chorin, M.~F. McCracken, T.~J.~R. Hughes, and J.~E. Marsden, ``Product
  formulas and numerical algorithms,'' {\em Comm. Pure Appl. Math.}, vol.~31,
  no.~2, pp.~205--256, 1978.

\bibitem{Itoh88}
T.~Itoh and K.~Abe, ``Hamiltonian-conserving discrete canonical equations based
  on variational difference quotients,'' {\em J. Comput. Phys.}, vol.~76,
  no.~1, pp.~85--102, 1988.

\bibitem{Marsden1992}
J.~E. Marsden, {\em Lectures on mechanics}, vol.~174 of {\em London
  Mathematical Society Lecture Note Series}.
\newblock Cambridge University Press, Cambridge, 1992.

\bibitem{Kriksin1993}
Y.~A. Kriksin, ``A conservative difference scheme for a system of {H}amiltonian
  equations with external action,'' {\em Zh. Vychisl. Mat. i Mat. Fiz.},
  vol.~33, no.~2, pp.~206--218, 1993.

\bibitem{Gonzalez96}
O.~Gonzalez, ``Time integration and discrete {H}amiltonian systems,'' {\em J.
  Nonlinear Sci.}, vol.~6, no.~5, pp.~449--467, 1996.

\bibitem{Harten83}
A.~Harten, P.~D. Lax, and B.~van Leer, ``On upstream differencing and
  {G}odunov-type schemes for hyperbolic conservation laws,'' {\em SIAM Rev.},
  vol.~25, no.~1, pp.~35--61, 1983.

\bibitem{Quispel08}
G.~R.~W. Quispel and D.~I. McLaren, ``A new class of energy-preserving
  numerical integration methods,'' {\em J. Phys. A}, vol.~41, no.~4,
  pp.~045206, 7, 2008.

\bibitem{Celledoni09}
E.~Celledoni, R.~I. McLachlan, D.~I. McLaren, B.~Owren, G.~R.~W. Quispel, and
  W.~M. Wright, ``Energy-preserving {R}unge-{K}utta methods,'' {\em M2AN Math.
  Model. Numer. Anal.}, vol.~43, no.~4, pp.~645--649, 2009.

\bibitem{Hairer06}
E.~Hairer, C.~Lubich, and G.~Wanner, {\em Geometric numerical integration},
  vol.~31 of {\em Springer Series in Computational Mathematics}.
\newblock Springer-Verlag, Berlin, second~ed., 2006.
\newblock Structure-preserving algorithms for ordinary differential equations.

\bibitem{McLaren04}
D.~I. McLaren and G.~R.~W. Quispel, ``Integral-preserving integrators,'' {\em
  J. Phys. A}, vol.~37, no.~39, pp.~L489--L495, 2004.

\bibitem{McLaren09}
D.~I. McLaren and G.~R.~W. Quispel, ``Bootstrapping discrete-gradient
  integral-preserving integrators to fourth order,'' in {\em Nonlinear
  dynamics} (M.~Daniel and S.~Rajasekar, eds.), pp.~157--171, Narosa Publishing
  House, 2009.

\bibitem{Hairer09}
E.~Hairer, ``Energy-preserving variant of collocation methods,'' {\em JNAIAM.
  J. Numer. Anal. Ind. Appl. Math.}, vol.~5, no.~1-2, pp.~73--84, 2010.

\bibitem{Cohen11}
D.~Cohen and E.~Hairer, ``Linear energy-preserving integrators for {P}oisson
  systems,'' {\em BIT}, vol.~51, no.~1, pp.~91--101, 2011.

\bibitem{Norton15}
R.~A. Norton, D.~I. McLaren, G.~R.~W. Quispel, A.~Stern, and A.~Zanna,
  ``Projection methods and discrete gradient methods for preserving first
  integrals of {ODE}s,'' {\em Discrete Contin. Dyn. Syst.}, vol.~35, no.~5,
  pp.~2079--2098, 2015.

\bibitem{Norton14}
R.~A. Norton and G.~R.~W. Quispel, ``Discrete gradient methods for preserving a
  first integral of an ordinary differential equation,'' {\em Discrete Contin.
  Dyn. Syst.}, vol.~34, no.~3, pp.~1147--1170, 2014.

\bibitem{Furihata99}
D.~Furihata, ``Finite difference schemes for {$\partial u/\partial
  t=(\partial/\partial x)^\alpha\delta G/\delta u$} that inherit energy
  conservation or dissipation property,'' {\em J. Comput. Phys.}, vol.~156,
  no.~1, pp.~181--205, 1999.

\bibitem{Celledoni12}
E.~Celledoni, V.~Grimm, R.~I. McLachlan, D.~I. McLaren, D.~O'Neale, B.~Owren,
  and G.~R.~W. Quispel, ``Preserving energy resp. dissipation in numerical
  {PDE}s using the ``{A}verage {V}ector {F}ield'' method,'' {\em J. Comput.
  Phys.}, vol.~231, no.~20, pp.~6770--6789, 2012.

\bibitem{Eidnes18}
S.~Eidnes, B.~Owren, and T.~Ringholm, ``Adaptive energy preserving methods for
  partial differential equations,'' {\em Adv. Comput. Math.}, vol.~44, no.~3,
  pp.~815--839, 2018.

\bibitem{Cai15}
J.~Cai, J.~Hong, Y.~Wang, and Y.~Gong, ``Two energy-conserved splitting methods
  for three-dimensional time-domain {M}axwell's equations and the convergence
  analysis,'' {\em SIAM J. Numer. Anal.}, vol.~53, no.~4, pp.~1918--1940, 2015.

\bibitem{Jiang17}
C.~Jiang, J.~Sun, H.~Li, and Y.~Wang, ``A fourth-order {AVF} method for the
  numerical integration of sine-{G}ordon equation,'' {\em Appl. Math. Comput.},
  vol.~313, pp.~144--158, 2017.

\bibitem{Greydanus19}
S.~Greydanus, M.~Dzamba, and J.~Yosinski, ``Hamiltonian neural networks,'' {\em
  Advances in Neural Information Processing Systems}, vol.~32,
  pp.~15379--15389, 2019.

\bibitem{Chen19}
Z.~Chen, J.~Zhang, M.~Arjovsky, and L.~Bottou, ``Symplectic recurrent neural
  networks,'' {\em arXiv preprint, arXiv:1909.13334}, 2019.

\bibitem{Matsubara19}
T.~Matsubara, A.~Ishikawa, and T.~Yaguchi, ``Deep energy-based modeling of
  discrete-time physics,'' {\em Advances in Neural Information Processing
  Systems}, vol.~33, pp.~13100--13111, 2020.

\bibitem{Matsuo02}
T.~Matsuo, M.~Sugihara, D.~Furihata, and M.~Mori, ``Spatially accurate
  dissipative or conservative finite difference schemes derived by the discrete
  variational method,'' {\em Japan J. Indust. Appl. Math.}, vol.~19, no.~3,
  pp.~311--330, 2002.

\bibitem{Yaguchi12}
T.~Yaguchi, T.~Matsuo, and M.~Sugihara, ``The discrete variational derivative
  method based on discrete differential forms,'' {\em J. Comput. Phys.},
  vol.~231, no.~10, pp.~3963--3986, 2012.

\bibitem{Furihata11}
D.~Furihata and T.~Matsuo, {\em Discrete variational derivative method}.
\newblock Chapman \& Hall/CRC Numerical Analysis and Scientific Computing, CRC
  Press, Boca Raton, FL, 2011.
\newblock A structure-preserving numerical method for partial differential
  equations.

\bibitem{Zhong19}
Y.~D. Zhong, B.~Dey, and A.~Chakraborty, ``Symplectic {ODE}-net: Learning
  {H}amiltonian dynamics with control,'' in {\em International Conference on
  Learning Representations}, 2019.

\bibitem{Norsett79}
S.~P. N{\o}rsett and A.~Wolfbrandt, ``Order conditions for {R}osenbrock type
  methods,'' {\em Numer. Math.}, vol.~32, no.~1, pp.~1--15, 1979.

\bibitem{Faou04}
E.~Faou, E.~Hairer, and T.-L. Pham, ``Energy conservation with non-symplectic
  methods: examples and counter-examples,'' {\em BIT}, vol.~44, no.~4,
  pp.~699--709, 2004.

\bibitem{Chartier06}
P.~Chartier, E.~Faou, and A.~Murua, ``An algebraic approach to invariant
  preserving integrators: the case of quadratic and {H}amiltonian invariants,''
  {\em Numer. Math.}, vol.~103, no.~4, pp.~575--590, 2006.

\bibitem{Cranmer20}
M.~Cranmer, S.~Greydanus, S.~Hoyer, P.~Battaglia, D.~Spergel, and S.~Ho,
  ``Lagrangian neural networks,'' {\em arXiv preprint, arXiv:2003.04630}, 2020.

\bibitem{Aoshima21}
T.~Aoshima, T.~Matsubara, and T.~Yaguchi, ``Deep discrete-time {L}agrangian
  mechanics,'' {\em ICLR2021 Workshop on Deep Learning for Simulation (SimDL)},
  5 2021.

\bibitem{Matsuo01dissipative}
T.~Matsuo and D.~Furihata, ``Dissipative or conservative finite-difference
  schemes for complex-valued nonlinear partial differential equations,'' {\em
  J. Comput. Phys.}, vol.~171, no.~2, pp.~425--447, 2001.

\bibitem{Dahlby11}
M.~Dahlby and B.~Owren, ``A general framework for deriving integral preserving
  numerical methods for {PDE}s,'' {\em SIAM J. Sci. Comput.}, vol.~33, no.~5,
  pp.~2318--2340, 2011.

\bibitem{Eidnes21}
S.~Eidnes, L.~Li, and S.~Sato, ``Linearly implicit structure-preserving schemes
  for {H}amiltonian systems,'' {\em J. Comput. Appl. Math.}, vol.~387,
  pp.~112489, 12, 2021.

\bibitem{Celledoni20}
E.~Celledoni, S.~Eidnes, B.~Owren, and T.~Ringholm, ``Energy-preserving methods
  on {R}iemannian manifolds,'' {\em Math. Comp.}, vol.~89, no.~322,
  pp.~699--716, 2020.

\bibitem{Berland05}
H.~Berland, B.~Owren, and B.~Skaflestad, ``{$B$}-series and order conditions
  for exponential integrators,'' {\em SIAM J. Numer. Anal.}, vol.~43, no.~4,
  pp.~1715--1727, 2005.

\bibitem{Mei21}
L.~Mei, L.~Huang, and X.~Wu, ``Energy-preserving exponential integrators of
  arbitrarily high order for conservative or dissipative systems with highly
  oscillatory solutions,'' {\em J. Comput. Phys.}, vol.~442, pp.~Paper No.
  110429, 22, 2021.

\end{thebibliography}
\bibliographystyle{ieeetr}

\end{document}